\documentclass[11pt,oneside,english]{amsart}
\usepackage[LGR,T1]{fontenc}
\usepackage[latin9]{inputenc}
\usepackage{geometry}
\usepackage{textcomp}
\usepackage{amsbsy}
\usepackage{amstext}
\usepackage{amsthm}
\usepackage{amssymb}

\makeatletter

\DeclareRobustCommand{\greektext}{%
  \fontencoding{LGR}\selectfont\def\encodingdefault{LGR}}
\DeclareRobustCommand{\textgreek}[1]{\leavevmode{\greektext #1}}
\ProvideTextCommand{\~}{LGR}[1]{\char126#1}

\numberwithin{equation}{section}
\numberwithin{figure}{section}
\theoremstyle{plain}
\newtheorem{thm}{\protect\theoremname}[section]
  \theoremstyle{definition}
  \newtheorem{example}[thm]{\protect\examplename}
  \theoremstyle{remark}
  \newtheorem{rem}[thm]{\protect\remarkname}
  \theoremstyle{plain}
  \newtheorem{cor}[thm]{\protect\corollaryname}
  \theoremstyle{definition}
  \newtheorem{defn}[thm]{\protect\definitionname}
  \theoremstyle{plain}
  \newtheorem{prop}[thm]{\protect\propositionname}
  \theoremstyle{plain}
  \newtheorem{lem}[thm]{\protect\lemmaname}
  \theoremstyle{plain}
  \newtheorem{conjecture}[thm]{\protect\conjecturename}

\def\makebbb#1{
    \expandafter\gdef\csname#1\endcsname{
        \ensuremath{\Bbb{#1}}}
}\makebbb{R}\makebbb{N}\makebbb{Z}\makebbb{C}\makebbb{H}\makebbb{E}\makebbb{H}\makebbb{P}\makebbb{B}\makebbb{Q}\makebbb{E}

\usepackage{babel}

\usepackage{babel}

\makeatother

\usepackage{babel}

\makeatother

\usepackage{babel}
  \providecommand{\conjecturename}{Conjecture}
  \providecommand{\corollaryname}{Corollary}
  \providecommand{\definitionname}{Definition}
  \providecommand{\examplename}{Example}
  \providecommand{\lemmaname}{Lemma}
  \providecommand{\propositionname}{Proposition}
  \providecommand{\remarkname}{Remark}
\providecommand{\theoremname}{Theorem}

\begin{document}

\title{statistical mechanics of interpolation nodes, pluripotential theory
and complex geometry}

\author{Robert J. Berman}
\begin{abstract}
This is mainly a survey, explaining how the probabilistic (statistical
mechanical) construction of Kähler-Einstein metrics on compact complex
manifolds, introduced in a series of works by the author, naturally
arises from classical approximation and interpolation problems in
$\C^{n}.$ A fair amount of background material is included. Along
the way, the results are generalized to the non-compact setting of
$\C^{n}.$ This yields a probabilistic construction of Kähler solutions
to Einstein's equations in $\C^{n},$ with cosmological constant $-\beta,$
from a gas of interpolation nodes in equilibrium at positive inverse
temperature $\beta.$ In the infinite temperature limit, $\beta\rightarrow0,$
solutions to the Calabi-Yau equation are obtained. In the opposite
zero-temperature the results may be interpreted as ``transcendental''
analogs of classical asymptotics for orthogonal polynomials, with
the inverse temperature $\beta$ playing the role of the degree of
a polynomial. 
\end{abstract}

\address{Robert J. Berman, Mathematical Sciences, Chalmers University of Technology
and the University of Gothenburg, SE-412 96 Göteborg, Sweden}

\email{robertb@chalmers.se}

\keywords{pluricomplex Green functions, complex Monge-Ampère operators, point
processes, large deviations. MSC 2010 subject classifications: 32U35,
32W20, 60G55, 60F10 }
\maketitle

\section{Introduction}

The main purpose of the present work is to explain how the probabilistic
(statistical mechanical) construction of Kähler-Einstein metrics on
compact complex manifolds - which is the subject of the series of
works \cite{berm8 comma 5,berm8,berm6,berm11} - naturally arises
from classical approximation and interpolation problems in $\C^{n}.$
In particular, the connection to Siciak's conjecture \cite{si2} (proved
in \cite{b-b-w}) about the convergence of Fekete points towards the
pluripotential equilibrium measure is stressed. 

In a nutshell, Fekete points may be defined as optimal interpolation
nodes. As will be explained below, from the statistical mechanical
point of view they are minimal energy configurations, which arise
at zero-temperature. When the temperature is increased (twisted) Kähler-Einstein
metrics emerge. See \cite{berm5} for different motivations coming
from physics, in particular emergent gravity and quantum gravity and
\cite{berm8 comma 5} for motivations coming from algebraic geometry.
We also provide sufficent background from both complex geometry and
probability in order to explain the proofs of the main results and
to put the results into a general context. This work is thus mainly
expository. But compared to the original papers more emphasize will
be put on the notion of Gamma-convergence. 

We also take the opportunity to show how the results in the compact
setting can be used to obtain new results in the non-compact setting
of $\C^{n}.$ This leads, among other things, to ``transcendental''
analogs of classical asymptotics for orthogonal polynomial, with the
inverse temperature $\beta$ playing the role of the degree of a polynomial.
The paper is concluded with an outlook section where open problems
are discussed. 

The main results may be summarized as follows in the non-compact setting
of $\C^{n}.$ Assume given a smooth function $\phi$ in $\C^{n}$
with sufficent growth at infinity (super logarithmic) and a positive
number $\beta.$ We will introduce a probabilistic construction of
solutions $\omega_{\beta}$ to Einstein's equation 
\[
\mbox{Ric\,}\ensuremath{\omega+\beta\omega=\tau_{\beta,\phi}}
\]
 for a Kähler metric $\omega$ on $\C^{n},$ where $\tau_{\beta,\phi}$
is the symmetric tensor
\[
\tau_{\beta,\phi}=\beta\omega^{\phi},
\]
 with $\omega^{\phi}$ denoting the complex Hessian of the given function
$\phi.$ The solution $\omega_{\beta}$ will be constructed by an
explicit sampling procedure involving ``random interpolation nodes''
in $\C^{n}$ with $\phi$ plaiyng the role of a ``weight function''.
From a statistical mechanical point of view the role of the interpolation
nodes is played by interacting particles in thermal equilibrium. The
parameter $\beta$ corresponds to the inverse temperature and $\phi$
to an exterior confining potential. This can be seen as a higher-dimensional
generalization of the classical principle that interpolation nodes
in $\C$ tend to behave as repulsive electric charges. However, in
the present higher dimensional setting the corresponding physical
model is provided by the theory of general relativity, rather than
electrostatics. Indeed, the metric $\omega_{\beta}$ is a solution
to Einstein's equation on $\C^{n}$ (with Euclidean signature) with
cosmological constant $-\beta$ and prescribed (trace-reversed) stress-energy
tensor $T_{\beta,\phi}$ (the solution $\omega_{\beta}$ is uniquely
determined by condition that its Kähler potential $\psi_{\beta}$
has logarithmic growth at infinity). From the complex geometric point
of view such metrics are \emph{twisted Kähler-Einstein metrics}, which
play an important role in current complex geometry in connection to
the Yau-Tian-Donaldson conjecture, where the role of $\C^{n}$ is
played by a compact complex manifold (see the survey \cite{do}).

We also show that in the ``zero temperature limit'' $(\beta\rightarrow\infty)$
the Kähler potential $\psi_{\beta}$ of $\omega_{\beta}$ converges
towards Siciaks weighted extremal function, which is the potential
of the weighted pluripotential equilibrium measure. In the opposite
``infinite temperature limit'' $(\beta\rightarrow0)$ we instead
recover a solution to the Calabi-Yau equation in $\C^{n}.$ Finally,
the tropicalization of the negative temperature regime is discussed,
where Kähler-Einstein metrics on toric varieties arise as $\beta$
approaches a certain critical negative inverse temperature. Relations
to Optimal Transport theory are also briefly mentioned. 

\subsection{Organization}

We start in Section \ref{sec:A-birds-eye} by giving a bird's-eye
view, providing motivation from interpolation theory and stating the
main results from \cite{b-b-w,berm8}. The new results are given in
Theorems \ref{thm:conv in law in C^n intro}, Theorem \ref{thm:zero temp limit singular case intro}
and Theorem \ref{thm:inf temp intro}. In the following two sections
we zoom in on the details, providing background from complex geometry
and probability. The proofs of the main results are then explained
in Sections \ref{sec:Proofs-for-the limit N}, \ref{sec:Proofs-for-the limit beta}.
Finally, in Section \ref{sec:Outlook-and-open} an outlook on open
problems is provided. In the appendix the present setting is compared
with the classical Curie-Weiss model for magnetization in spin systems.

\subsection{Acknowledgements}

Thanks to Slawomir Kolodziej for the invitation to contribute to the
upcoming volume in Annales Polonici Mathematici, in memory of Jozef
Siciak, which prompted the present work. It is also a pleasure to
thank Sebastien Boucksom, David Witt-Nyström, Vincent Guedj and Ahmed
Zeriahi for the stimulating collaborations \cite{b-b,b-b-w,bbgz},
which paved the way for the probabilistic approach and Mingchen Xia
for comments on a draft of the present manuscript. This work was supported
by grants from the ERC and the KAW foundation. 

\section{\label{sec:A-birds-eye}A bird's-eye view}

\subsection{The classical one-dimensional setting }

A central problem in approximation theory is to find good points (``nodes'')
for interpolating polynomials of large degree, on a given compact
subset $K$ in $\C^{n}.$ We start by recalling the classical one-dimensional
case of this problem, which has been studied extensively both from
a theoretical, as well as a numerical point of view with classical
contributions from Gauss, Fekete, Fejer, Szegö, Chebishev and many
others. We thus assume given a compact subset $K$ of $\C.$ Denote
by $\mathcal{P}_{k}(\C)$ the complex vector space of all polynomials
on $\C$ of degree at most $k.$ This space has dimension 
\[
N_{k}:=\dim\mathcal{P}_{k}(\C)=k+1
\]
A given configuration $(z_{1},...z_{N_{k}})$ of $N_{k}$ points on
$K$ determines an evaluation map 
\begin{equation}
\text{ev}_{(z_{1},...,z_{N_{k}})}:\,\,\mathcal{P}_{k}(\C)\rightarrow\C^{N_{k}},\,\,\,\,p_{k}\mapsto(p_{k}(z_{1}),...,p_{k}(z_{N_{k}}))\label{eq:evalu map}
\end{equation}
The worst case scenario is when the determinant of this map vanishes.
Indeed, then the evaluation map cannot be inverted, which means that
there are values in $\C^{N_{k}},$ which cannot be interpolated by
any polynomial $p_{k}$ of degree at most $k.$ But even if the determinant
does not vanish a generic configuration of points typically leads
to unstable interpolation for large degree polynomial. This is illustrated
by the case when $K$ is an interval in $\R\subset\C$ and the points
are equidistant. Interpolation then results in polynomials that oscillate
wildly between the interpolation points, especially close to the boundary
(Runge's phenomenen). Accordingly, a natural candidate for good interpolation
points are configurations of points $(z_{1},...z_{N_{k}})$ on $K$
\emph{maximizing} the absolute value of the determinant of the evaluation
map \ref{eq:evalu map}. Such points are called \emph{Fekete points}
for $K.$ Picking a base $e_{1},...e_{N_{k}}$ in the vector space
$\mathcal{P}_{k}(\C)$ the determinant of the evaluation map \ref{eq:evalu map}
becomes
\begin{equation}
D(z_{1},...,z_{N_{k}}):=\det(e_{i}(z_{j}))_{i,j\leq N_{k}},\label{eq:Vandermonde det in terms of e i}
\end{equation}
 which is called the \emph{Vandermonde determinant} in the standard
case when $e_{i}(z):=z^{i}$ is taken as the monomial base. Since
the set $K$ is assumed compact it always admits Fekete points, for
any given degree $k.$ Finding explicit Fekete points is, in general,
an intractable problem which has only been solved for very particular
examples of $K$ (se below). But their asymptotic distribution as
$k\rightarrow\infty$ is, in general, described by the following classical
result (see \cite{s-t} for a proof):
\begin{thm}
\label{thm:fekete on compac in C}Assume that $K$ is compact subset
of $\C$ which is non-polar. Let $(z_{1},...z_{N_{k}})\in K^{N}$
be a sequence of Fekete points on $K.$ Then 
\[
\frac{1}{N_{k}}\sum_{i=1}^{N_{k}}\delta_{z_{i}}\rightarrow\mu_{K}
\]
 in the weak topology of measures on $K,$ where $\mu_{K}$ is the
potential theoretic equilibrium measure of $K.$ In other words, $K$
is the unique minimizer of the logarithmic energy, i.e. of the following
functional, defined on the space $\mathcal{M}_{1}(K)$ of all probability
measures on $K:$
\begin{equation}
E(\mu)=-\frac{1}{2}\int_{\C\times\C}\log|z-w|^{2}\mu(z)\otimes\mu(w)\label{eq:def of log energy of mu}
\end{equation}
\end{thm}

The starting point of the proof is the observation that $D(z_{1},...,z_{N_{k}})$
may be factorized as 
\[
D(z_{1},...,z_{N_{k}})=\prod_{i<j}(z_{i}-z_{j})
\]
This follows directly from the fact that the determinant $D(z_{1},...,z_{N_{k}})$
vanishes if two points coincide and restricts to a polynomial of degree
$k$ on $\C$ when all but one arguments are frozen. Replacing the
maximizing problem for $|D|$ with the equivalent minimization problem
for $-\log|D|^{2}$ thus reveals that the Fekete points on $K$ are
precisely the minimiers of the discrete (normalized) logarithmic energy
\begin{equation}
E^{(N)}(z_{1},...z_{N}):=-\frac{1}{N(N-1)}\log|D|^{2}=-\frac{1}{N(N-1)}\frac{1}{2}\sum_{i\neq j}g(z_{i},w_{j}),\label{eq:E and D in terms of g intro}
\end{equation}
 with a repulsive logarithmic pair interaction potential: 
\[
g(z,w):=-\log|z-w|^{2}
\]

In physical terms this is the (normalized) interaction energy betweem
$N$ identical particles interacting by the pair interaction potential
$g(z,w),$ defined by Green function of the Laplacian on $\C$(with
the divergent self-interactions removed). The minimizers thus represent
$N$ equally charged particles confined to $K$ and in electrostatic
equilibrium. If the pair interaction $g$ were continuous it would
follow that the corresponding energies are related by

\[
E_{g}^{(N)}(z_{1},...z_{N})=E_{g}(\frac{1}{N}\sum_{i=1}^{N}\delta_{z_{i}})+o(1)
\]
 as $N\rightarrow\infty,$ which implies Theorem \ref{thm:fekete on compac in C}
in the continuous setting. Even if the previous relation drastically
fails when $g$ is logarithmic (since the right hand side is alway
divergent) a truncation argument can be used to essentially reduce
the proof of Theorem \ref{thm:fekete on compac in C} to the case
when $g$ is continuous. 
\begin{example}
When $K$ is a circle the Fekete points are precisely the configurations
consisting of equidistant points and their large $k-$limit $\mu_{K}$
is thus the unique invariant probality measure on the circle. In this
case the space $\mathcal{P}_{k}(\C)$ identifies with Fourier series
with ``band-limited'' frequences is $[0,k],$ For $K=[-1,1]\Subset\R\subset\C$
the Fekete points are precisely the Gauss-Lobatoo points (which may
be realized as the zeroes of Jacobi polynomials) and $\mu_{K}$ is
the arcsine distribution $\pi^{-1}(1-|x|^{2})^{-1/2}dx$ \cite{fe}. 
\end{example}

The case when $K$ is a circle illustrates a striking phenomen: even
if $E^{(N)}$ has, in general, a multitude of minimizers in $K^{N}$
the functional $E(\mu)$ always has a unique minimizer. The equidistribution
result in Theorem \ref{thm:fekete on compac in C} applies more generally
to \emph{asymptotic Fekete points} on $K,$ i.e. configuration of
points on $X$ such that $E^{(N_{k})}(z_{1},...z_{N_{k}})$ is equal
to the minimum value of $E^{(N_{k})},$ up to an error tending to
zero as $k\rightarrow\infty.$ This rather flexible notion enpasses
many classical configuration of points such a\emph{s Lesbegue points}
(aka \emph{optimal interpolation nodes, }as they minimize the $L^{\infty}-$
operator norm of the evaluation map) and zeroes of degree $k$ polynomials
$p_{k}$ which are orthogonal with respect to a measure with support
$K$ (e.g. zeros of Chebishev polynomials when $K=[-1.1]).$ 

From a physical point of view the upshot of Theorem \ref{thm:fekete on compac in C}
and its proof is that \emph{nearly optimal interpolation behave as
repulsive charged particles in (or close to) electrostatic equilibrium. }

\subsubsection{The equilibrium potential as an envelope}

Before turning to the higher dimensional setting we recall some alternative
expressions for the logarithmic energy $E(\mu)$ and its minimizer
$\mu_{K}.$ First, if follows directly from the defininition that 

\begin{equation}
E(\mu)=-\frac{1}{2}\int_{\C}\psi_{\mu}\mu,\,\,\,\psi_{\mu}(z):=\int_{\C}\log|z-w|^{2}\mu(w)\label{eq:greens formul for E in plane}
\end{equation}
where $\psi_{\mu}$ is a \emph{potential} for the measure $\mu,$
i.e. $\psi_{\mu}$ is as a solution of the (normalized) Laplace equation
in $\C:$
\begin{equation}
\frac{1}{4\pi}\Delta\psi=\mu\label{eq:Lapace equation}
\end{equation}
The Laplace equation only determines the solution up to an additive
constnat and the particular solution $\psi_{\mu}$ above is singled
out by the normalizing condition that $\psi(z)-\log|z|\rightarrow0$
as $|z|\rightarrow\infty.$

A potential $\psi_{K}$ for the equilibrium measure$\mu_{K}$ of $K:$
\begin{equation}
\mu_{K}=\frac{1}{4\pi}\Delta\psi_{K}\label{eq:mu K in terms of phi K in C}
\end{equation}
 can be obtained by a classical envelope construction (due to Perron):
\begin{equation}
\psi_{K}(z)=\sup\{\psi:\,\,\,\psi\,\text{sh\,on\,\ensuremath{\C,}}\psi(z)=\log|z|^{2}+O(1)\,\,|z|\gg1\text{,\,}\,\psi\leq0\,\,\text{on\,\ensuremath{K,}}\}\label{eq:psi K is env when n=00003D1}
\end{equation}
 where we recall that a function $\psi$ in $L_{loc}^{1}$ is said
to be \emph{subharmonic (sh)} if it is (strongly) upper semi-continuous
and $\Delta\psi\geq0,$ i.e. the distributional Laplacian of $\psi$
defines a measure on $\C.$ By construction the potential $\psi_{K}$
integrates to zero against $\mu_{K}$ and thus has a different normalization
than $\psi_{\mu_{K}}.$ The envelope representation above can be deduced
from the sub-differential (Euler-Lagrange) condition $\partial E(\mu)\geq0$
at the minimizer $\mu_{K}$ combined with the domination principle
for the Laplacian. 

\subsubsection{The weighted setting and elliptic Fekete points on the two-sphere}

We also recall that the previous setup naturally extends to the more
general weighted setting where the compact set $K$ is replaced by
a \emph{weight function }$\phi,$ i.e. a lsc function on $\C$ assumed
finite and continuous on a non-polar subset of $\C.$ Then the point-wise
norm of a degree $k$ polynomial is replaced by the weighted point-wise
norm: 
\[
|p_{k}(z)|_{k\phi}^{2}:=|p_{k}(z)|^{2}e^{-k\phi(z)}
\]
and the corresponding interpolation problem amounts to finding good
points for interpolating degree $k$ polynomials, making interpolation
stable with respect to the weighted sup-norms. Accordingly, the point-wise
norm of the Vandermonde determinant is replaced by the weighted norm:
\begin{equation}
|D(z_{1},...,z_{N})|_{k\phi}^{2}:=|D(z_{1},...,z_{N})|^{2}e^{-k\phi(z_{1})}\cdots e^{-k\phi(z_{1})}\label{eq:def of weighted norm of Vand}
\end{equation}
In the case when $\phi$ is the indicator function of a compact set
$K$ the weighted setting thus reduces to the previous setting. Accordingly,
the corresponding weighted Fekete points are the minimizers of 
\[
E_{\phi}^{(N)}(z_{1},...z_{N_{k}}):=E^{(N)}(z_{1},...z_{N})+N^{-1}\sum_{i=1}^{N}\phi(z_{i}),
\]
 showing that the weight function $\phi$ plays the role of an exterior
potential, from the physics point of view. In order to ensure the
existence of weighted Feketete points the weight $\phi$ is assumed
to have super logarithmic growth:
\begin{equation}
\phi(z)\geq\log|z|^{2}+O(1),\,\,\,|z|\gg1\label{eq:super log growth}
\end{equation}
In the critical case where equality holds above the support of the
corresponding equlibrium measure $\mu_{\phi}$ may be non-compact,
in general. The model case is when the weight is given by 
\[
\phi_{FS}(z):=\log(1+|z|^{2}),
\]
 that we shall refer to as the \emph{Fubini-Study weight }(with an
eye towards the higher dimensional setting considered below). Then
$E_{\phi}^{(N)}(z_{1},...z_{N})$ identifies, under sterographic projection,
with the standard logarithmic energy on the two-sphere, which is invariant
under the action of the isometry group $O(3)$ of the two-sphere.
By symmetry the corresponding equilibrium measure $\mu_{\phi_{FS}},$
emerging when $N\rightarrow\infty,$ thus identifies with the standard
$O(3)-$ invariant probability measure on the two-sphere. But the
problem of locating the corresponding minimizers (called\emph{ elliptic
Fekete points}) is extremely challenging for \emph{a finite} $N.$
For example, Smale asked in the year of 2000, in his list of mathematical
problems for the next century \cite[Problem 7]{sm}, for an algorithm
which can locate configuration of points with almost minimal energy
$E_{\phi}^{(N)}$ (in a polymomial time, tolarating an error of the
order $O(N^{-2}\log N).$ 

\subsection{\label{subsec:Fekete-points-in}Fekete points in $\C^{n}$ and the
pluripotential equilibrium measure}

Let now $K$ be a compact subset of $\C^{n}$ and denote by $\mathcal{P}_{k}(\C^{n})$
the space of all polynomials on $\C^{n}$ with total degree at most
$k.$ This is a complex vector space of dimension 
\begin{equation}
N_{k}:=\dim\mathcal{P}_{k}(\C^{n})=\frac{k^{n}}{n!}+o(k^{n})\label{eq:def of N k intro}
\end{equation}
Fekete points for $K$ are defined, precisely as in the one-dimensional
setting above, as maximizers of the absolute value of the determinant
of the corresponding evaluation map from $\mathcal{P}_{k}(\C^{n})$
to $\C^{N_{k}}$ (such points are also called \emph{extremal points}
or \emph{extremal systems} in the literature). Fixing a base $e_{j}$
in $\mathcal{P}_{k}(\C^{n})$ consisting of multinomials we will use
the same notation $D^{(N_{k})}(z_{1},...z_{N_{k}})$ for the higher
dimensional incarnation of the Vandermonde determinant, defined by
formula \ref{eq:Vandermonde det in terms of e i}.

Given a general configuration $(z_{1},...z_{N})$ of $N$ points we
will denote by $\delta_{N}$ the corresponding discrete probability
measure on $\C^{n},$ called the\emph{ empirical measure:}

\[
\delta_{N}:=\frac{1}{N}\sum_{i=1}^{N}\delta_{z_{i}}.
\]
 In the light of Theorem \ref{thm:fekete on compac in C} it is natural
to ask 
\begin{enumerate}
\item Does the sequence of empirical measure $\delta_{N_{k}}$ corresponding
to a sequence $(z_{1},...,z_{N})$ of Fekete points on $K$ admit
a \emph{unique} weak limit $\mu_{K}$ in the space $\mathcal{M}_{1}(K)$
of probability measures on $K?$
\item If so, is there a direct analytic characterization of $\mu_{K}$? 
\end{enumerate}
A suitable generalization to $\C^{n}$ of the equilibrium measure
appearing in Theorem \ref{thm:fekete on compac in C} would be a natural
analytic candidate for $\mu_{K}.$ However, for various reasons the
naive generalization to $\C^{n}$ of the logarithmic energy $E(\mu),$
defined by formula \ref{eq:def of log energy of mu}, is not the right
object to look at in the higher dimensional setting of $\C^{n}.$
In the seventies Siciak \cite{si0,si1} and Zaharjuta \cite{z} instead
introduced a natural higher dimensional generalization of the equilibrium
potential $\psi_{K},$ by generalizing the envelope expression \ref{eq:psi K is env when n=00003D1}:
\[
\psi_{K}(z)=\sup_{\psi\in\mathcal{L}(\C^{n})}\{\psi(z)\,:\,\psi\leq0\,\,\text{on\,\ensuremath{K}}\},
\]
 where $\mathcal{L}(\C^{n})$ denotes the Lelong class, i.e. the class
of all plurisubharmonic (psh) functions $\psi$ on $\C^{n}$ with
sublogarithmic growth: 
\begin{equation}
\mathcal{L}(\C^{n}):=\left\{ \psi:\,\,\,\text{\ensuremath{\psi}\,\ensuremath{\text{psh},\,\,\,}}\psi(z)\leq\log|z|^{2}+O(1),\,\,|z|\gg1\right\} .\label{eq:def of L intro}
\end{equation}
The subclass of all psh $\psi$ with \emph{logarithmic growth} is
is denoted by $\mathcal{L}_{+}(\C^{n}),$ i.e. $\psi$ is psh and
\begin{equation}
\psi(z)=\log|z|^{2}+O(1)\,\,|z|\gg1\label{eq:log growth intro}
\end{equation}
 We recall that a function $\psi$ on $\C^{n}$ is \emph{plurisubharmonic
(psh)} if it is usc and subharmonic along complex lines. These functions
form the backbone of\emph{ pluripotential theory} and appear naturally
in polynomial approximation theory. Indeed, if $p_{k}(z)$ is a degree
$k$ polynomial on $\C^{n},$ then $k^{-1}\log|p_{k}(z)|$ is psh
and in the Lelong class $\mathcal{L}(\C^{n}).$ 

The function $\psi_{K}$ is bounded precily when $K$ is non-pluripolar,
i.e not locally contained in the $-\infty-$set of a psh function.
We will then identify $\psi_{K}$ with its upper-semi continuous regularization,
which defines an element in $\mathcal{L}_{+}(\C^{n}).$ The \emph{pluripotential
equilibrium measure} of a compact set $K\subset\C^{n},$ assumed non-pluripolar,
is then defined by generalizing formula \ref{eq:mu K in terms of phi K in C}:
\[
\mu_{K}:=MA(\psi_{K}),
\]
 where $MA$ denotes the complex Monge-Ampère operator, which is the
fully-nonlinear operator generalization the linear Laplace operator
to pluripotential theory: 
\[
MA(\psi):=(\frac{i}{2\pi}\partial\bar{\partial}\psi)^{n}=\frac{1}{\pi^{n}}\det(\frac{\partial^{2}\psi}{\partial z_{i}\partial\bar{z}_{j}})d\lambda,
\]
when $\psi$ is smooth and the extension to locally bounded psh is
due to Bedford-Taylor (see Section \ref{subsec:The-local-setting}).
The measure $\mu_{K}$ is supported on $K$ (as follows from the maximum/domination
principle) and defines a probability measure on $K.$ With this pluripotential
generalization of the classical equilibrium measure in hand it was
conjectured by Siciak \cite{si2} that the analog of Theorem \ref{thm:fekete on compac in C}
holds for any non-pluricompact subset of $K$ \cite{si2}. This was
confirmed in \cite{b-b-w}:
\begin{thm}
\label{thm:Fekete points C n intro}Assume that $K$ is a compact
subset of $\C^{n}$ which is non-pluripolar. Let $(z_{1},...z_{N_{k}})\in K^{N_{k}}$
be a sequence of Fekete points on $K.$ Then the corresponding empirical
measures $\delta_{N}$ converge, as $k\rightarrow\infty,$ to the
pluripotential equilibrium measure $\mu_{K}$ of $K,$ in the weak
topology of measures on $K.$ 
\end{thm}

The proof of the theorem in the higher dimensional setting is very
different from the proof in the classical one-dimensional setting,
as the Vandermonde determinant cannot be factorized when $n>1.$ This
means that there is no tractable expression for the corresponding
higher-dimensional analog of the discrete logarithmic energy:
\begin{equation}
E^{(N_{k})}(z_{1},...z_{N_{k}}):=-\frac{1}{N_{k}k}\log\left|D^{(N_{k})}(z_{1},...z_{N_{k}})\right|^{2}\label{eq:Def of E N C n intro}
\end{equation}
that we shall refer as the \emph{determinantal energy.} Still, the
proof in \cite{b-b-w}, which builds on \cite{b-b}, naturally leads
to a notion of\emph{ pluricomplex energy} $E(\mu)$ (further studied
in \cite{bbgz}). In fact, as shown in\cite{berm8}, the proof in
\cite{b-b-w} can be upgraded to give the more precise result
\begin{equation}
E^{(N_{k})}\rightarrow E,\,\,\,k\rightarrow\infty\label{eq:Gamma conv towards E intro}
\end{equation}
 in the sense of Gamma-convergence on the space $\mathcal{M}_{1}(K)$
of probability measures on $K,$ where $E$ is a functional on $\mathcal{M}_{1}(K),$
whose unique minimizer is the equilibrium measure $\mu_{K}$ (see
\cite{s1} for a different proof of the Gamma-convergence in the one-dimensional
case). This notion of convergence, introduced by Di Georgi, plays
a prominent role in the calculus of variations. Its definition ensures
that any minimizer of $E^{(N_{k})}$ converges to the unique minimizer
of $E.$ The functional $E$ coincides with the \emph{pluricomplex
energy} introduced in \cite{bbgz}. It may be explicitely expressed
in terms of the \emph{potential }$\psi_{\mu}$ of the measure $\mu$
i.e. the unique (modulo constants) solution of the complex \emph{Monge-Ampère
equation }
\begin{equation}
MA(\psi)=\mu,\,\,\,\,\psi\in\mathcal{L}_{+}(\C^{n})\label{eq:calabi-yau eq intro}
\end{equation}
 known as the \emph{Calabi-Yau equation} in the complex (Kähler) geometry
literature (generalizing the Laplace equation \ref{eq:Lapace equation}.
In fact, if $\mu$ is sufficently regular then $-\psi_{\mu}$ can
be viewed as the differential (gradient) of $E(\mu)$ (see Section
\ref{subsec:The-pluricomplex-energy}).
\begin{example}
In the case when $K$ is a symmetric convex body in $\R^{n}\Subset\C^{n}$
the equilibrium measure $\mu_{K}$ can be computed explicitely \cite{b-t3}.
Comparing with the one-dimensional case reveals that its density is
comparable to inverse of the squar-root of the distance $d(x,\partial K)$
to the boundary. 
\end{example}

The proof of Theorem \ref{thm:Fekete points C n intro} in \cite{b-b-w}
is given in the general complex geometric setting of a big line bundle
$L$ over a compact complex manifold $X.$ Then the role of the space
$\mathcal{P}_{k}(\C^{n})$ is played by the space $H^{0}(X,L^{\otimes k})$
of global holomorphic sections and the role of the weight function
$\phi$ is played by a (possibly singular) metric on the line bundle
$L.$ From this point of view Theorem \ref{thm:Fekete points C n intro}
appears as the special case when $L\rightarrow X$ is the hyperplane
line bundle $\mathcal{O}(1)\rightarrow\P^{n}$ over the $n-$dimensional
complex projective space $\P^{n},$ compactifying $\C^{n}.$ In this
case the corresponding metric on the restricted line bundle $L\rightarrow K$
is the trivial one, but the more general weighted version of Theorem
\ref{thm:Fekete points C n intro} also follows by using a general
metric $\phi$ on $L.$ In Section \ref{sec:Proofs-for-the limit N}
we will explain the proof in the technically simpler case where $L$
is a positive line bundle, i.e. $(X,L)$ is a \emph{polarized manifold,}
which is enough for the application to $\C^{n}$ (see Theorem \ref{thm:Gamma conv polarized}).
But the case when $L$ is big allows one, in particular, to also cover
the case when $X$ is a singular variety.
\begin{example}
Another concrete application of Theorem \ref{thm:Gamma conv polarized}
is to the case when $K$ is taken as a compact real algebraic variety
in $\R^{n+1}$ and $\mathcal{P}_{k}$ as the restriction to $K$ of
the space of all polynomials on $\R^{n+1}$ of total degree at most
$k.$ Then the analog of Theorem \ref{thm:Fekete points C n intro}
follows from Theorem \ref{thm:Gamma conv polarized}, applied to the
complexification $X$ in $\P^{n+1}$ of $K.$ For example, $K$ may
be taken as the unit-sphere $S^{n}\Subset\R^{n+1}.$ Then $\mathcal{P}_{k}$
is the space of all sperical polynomials of degree at most $k,$ i.e.
the space of all spherical harmonics band-limited to $[0,k].$ The
corresponding Fekete points are often refererred to as \emph{extremal
systems} (see \cite{s-w} where numerical approximations are computed
on the two-sphere, $n=2,$ when $k\leq191,$ which corresponds to
$N_{k}=36864$ points). By symmetry the pluripotential equilibrium
measure $\mu_{K}$ is the $O(n+1)$ invariant probability measure
on the sphere. 
\end{example}

An interesting feature of the proof of the general version of Theorem
\ref{thm:Fekete points C n intro}, for a polarized complex manifold
$(X,L)$ (Theorem \ref{thm:Gamma conv polarized}), is that it essentially
reduces the proof of Theorem \ref{thm:Fekete points C n intro} for
a compact subset $K$ of $\C^{n}$ to the weighted setting where $K$
is replaced by a smooth weight $\phi$ on $\C^{n}$ with logarithmic
growth (i.e. equality holds in formula \ref{eq:super log growth}).
Then the analog of $\psi_{K}$ is the weighted extremal function $\psi_{\phi}$
defined by the upper semi-continuous regularization of the envelope
\begin{equation}
\psi_{\phi}(z):=\sup_{\psi\in\mathcal{L}(\C^{n})}\{\psi(z)\,:\,\psi\leq\phi\,\,\text{on\,\ensuremath{\C^{n}}}\}\label{eq:def of psi phi intro}
\end{equation}
 and the correponding weighted equilibrium measure is the probability
measure defined as 
\begin{equation}
\mu_{\phi}:=MA(\psi_{\phi})\label{eq:def of MA phi intro}
\end{equation}
More generally, Theorem \ref{thm:Gamma conv polarized} applies to
the general setting where one is given a closed non-pluripolar subset
$K$ of $\C^{n}$ and a continuous weight function $\phi$ on $K$
with logarithmic growth. Then the corresponding extremal function
$\psi_{(K,\phi)}$ is defined by
\begin{equation}
\psi_{(K,\phi)}(z):=\sup_{\psi\in\mathcal{L}(\C^{n})}\{\psi(z)\,:\,\psi\leq\phi\,\,\text{on\,\ensuremath{K}}\}\label{eq:def of psi K phi}
\end{equation}

\subsection{\label{subsec:Statistical-mechanics-and}Statistical mechanics and
the emergence of twisted Kähler-Einstein metrics at positive temperature}

Comparing with the one-dimensional situation it is tempting to think
of Fekete points on a compact subset $K$ of $\C^{n}$ (and more generally,
nearly optimal interpolation nodes) as interacting particles confined
to $K,$ forming a microscopic equilibrium state. In \cite{berm8,berm11}
this analogy is pushed further by introducing temperature (and thus
randomness) into the picture. To explain this let us first recall
the general statistical mechanical setup. Consider a system of $N$
identical particles on a space $X$ interacting by a microscopic interaction
energy $H^{(N)}(x_{1},..x_{N})$ (the\emph{ Hamiltonian}) on $X^{N}$
assumed symmetric (as the particles are identical) and lower semi-continuous.
Given a measure $dV$ on $X$ the distribution of particles, in thermal
equilibrium at a fixed temperature $\beta^{-1}\in]0,-\infty[,$ is
described by the corresponding \emph{Gibbs measure}. This is the symmetric
probability measure on $X^{N}$ defined by 
\begin{equation}
\mu_{\beta}^{(N)}:=\frac{1}{Z_{N,\beta}}e^{-\beta H^{(N)}(x_{1},...,x_{N})}dV^{\otimes N},\,\,\,\,Z_{N,\beta}:=\int_{X^{N}}e^{-\beta H^{(N)}(x_{1},...,x_{N})}dV^{\otimes N}\label{eq:def of Gibbs measure intro}
\end{equation}
where the normalizing constant $Z_{N,\beta}$ is called the \emph{partition
function} and is assumed to be finite (which is automatic if $X$
is compact, as $H^{(N)}$ is assumed lsc and hence bounded from below).
The probability space $(X^{N},\mu_{\beta}^{(N)})$ is called the\emph{
canonical ensemble} in the physics literature. The empirical measure
\[
\delta_{N}:=\frac{1}{N}\sum_{i=1}^{N}\delta_{x_{i}}:\,\,\,X^{N}\rightarrow\mathcal{M}_{1}(X)
\]
 can now be viewed as a random measure, i.e. a random variable on
$(X^{N},\mu_{\beta}^{(N)})$ taking values in the space of probability
measures on $X.$ Heuristic reasoning, inspired by mean field theory
in physics suggests the following ``Mean Field Approximation'':
the random measures $\delta_{N}$ converge in law towards a deterministic
limit $\mu_{\beta},$ in the many particle limit $N\rightarrow\infty$
(i.e. the random measure $\delta_{N}$ is well approximated by $\mu_{\beta}$
) under the following assumptions: 
\begin{itemize}
\item The microscopic ``energy per particle''
\[
E^{(N)}(x_{1},...x_{N}):=\frac{1}{N}H^{(N)}(x_{1},...,x_{N})
\]
converges, in the many particule limit $N\rightarrow\infty,$ towards
a ``macroscopic energy'' $E(\mu)$ on the space $\mathcal{M}_{1}(X)$
in an apprioate sense (to be specified!)
\item The corresponding macroscopic\emph{ free energy functional $F_{\beta}$
on $\mathcal{M}_{1}(X)$ }admits a unique minimizer, to wit $\mu_{\beta}:$
\[
F_{\beta}(\mu):=E(\mu)-\beta^{-1}S_{dV}(\mu),
\]
 where $S_{dV}(\mu)$ is the ``physical entropy'' of $\mu$ relative
to $dV$ (see Remark \ref{rem:The-entropy-is}). 
\end{itemize}
In order to make the ``Mean Field Approximation'' into a mathematically
precise statement the meaning of the convergence in the first point
above needs to be specified (see Section \ref{subsec:Mean-field-approximations}
for further discussion). The main two classes of cases where the Mean
Field Approximation has been rigorously established is in the
\begin{itemize}
\item ``Regular case'': $X$ is compact and $E^{(N)}$ is continuous,
uniformly wrt $N$ \cite{e-h-t}.
\item ``Linear case'': $E^{(N)}$ is a sum of the form \ref{eq:E and D in terms of g intro}
with a lsc pair interaction potentail $g\in L^{1}(X^{2},dV^{\otimes2})$
\cite{k2,clmp} 
\end{itemize}
In the latter case a trunctation argument can be employed to essentially
reduce the problem to the ``regular case'' (a similar argument applies
to a finite sum of $m-$point interactions \cite{berm10}). However,
in the present setting $E^{(N)}$ (as defined by formula \ref{eq:Def of E N C n intro})
is both very non-linear (when $n>1)$ and singular. What saves the
situation is that $E^{(N)}$ is superharmonic so that the following
key result in \cite{berm8} can be applied:
\begin{thm}
\label{thm:LDP subharm intro} Assume that $X$ is a compact Riemannian
manifold without boundary and that the sequence $\frac{1}{N}H^{(N)}(x_{1},...,x_{N})$
Gamma converges to a functional $E(\mu)$ on $\mathcal{M}_{1}(X).$
If moreover, $H^{(N)}$ is uniformly quasi-superharmonic, i.e. there
exists a constant $C$ such that for all $N$
\[
\Delta_{x_{i}}H^{(N)}(x_{1},x_{2},...x_{N})\leq C\,\,\,\,X^{N},
\]
 then the laws of the random measures $\delta_{N}$ on $(X^{N},\mu_{\beta}^{(N)})$
satisfy a Large Deviation Principle (LDP) with speed $\beta N$ and
rate functional $F_{\beta}(\mu)-C_{\beta},$ where 
\[
C_{\beta}=\inf_{\mathcal{M}_{1}(X)}F_{\beta}
\]
In particular, if $E$ is convex and $\beta>0$ then $\delta_{N}$
converges in law towards the unique minimizer $\mu_{\beta}$ of $F_{\beta}.$
\end{thm}

The general notion of a LDP is recalled in Section \ref{subsec:The-notion-of}.
But loosely speaking, the LDP implies that the risk that the empirical
measure $\delta_{N}$ deviates from the minimizer $\mu_{\beta}$ is
exponentiall small, as $N\rightarrow\infty,$ which implies convergence
in law towards $\mu_{\beta}.$ The idea of the proof of the previous
theorem is to apply geometric analysis and comparison geometry to
the configuration space of $N$ points on $X,$ viewed as a Riemannian
orbifold (see Section \ref{sec:Proofs-for-the limit N}). 

\subsubsection{Random interpolation nodes}

In the present setting, where $N$ is taken as the sequence $N_{k}$
of integers defined by formula \ref{eq:def of N k intro}, we take
the Hamiltonian to be 
\begin{equation}
H^{(N_{k})}(z_{1},...z_{N_{k}})=-\frac{1}{k}\log\left|D^{(N_{k})}(z_{1},...z_{N_{k}})\right|_{k\phi}^{2}\label{eq:def of H N weighted intro}
\end{equation}
for a given weight function $\phi$ on $\C^{n},$ assumed continuous
(which, from the statistical mechanics point of view, plays the role
of an exterior potential, just as in the one-dimensional case discussed
above). Fixing also a continuous volume form $dV$ on $\C^{n}$ the
corresponding Gibbs measure is thus given by the following propobability
measure on $(\C^{n})^{N}:$
\begin{equation}
\mu_{\beta}^{(N_{k})}=\frac{1}{Z_{N,\beta}}|D^{(N_{k})}(z_{1},...,z_{N_{k}})|_{k\phi}^{2\beta/k}dV^{\otimes N},\label{eq:prod meas mubetaN intro}
\end{equation}
(which, as explained in Section \ref{eq:det section as a determinant},
can be viewed as a deformation of a \emph{determinantal random point
process}). Denoting by $d\lambda$ Lesbesgue measure on $\C^{n}$
we will say that the given ``weighted measure'' $(\phi,dV)$ is\emph{
admissble wrt $\beta$} if there exists $\epsilon>0$ such that
\begin{equation}
\phi-\beta^{-1}\log(dV/d\lambda)\geq\left(1+n/\beta+\epsilon\right)\log|z|^{2}\,\,\,|z|\gg1\label{eq:adm}
\end{equation}
and $\phi$ has at most iterated exponential growth, i.e. there exists
a positive constant $C$ and a positive integer $m$ such that
\[
\phi(z)\leq Ce^{(m)}(z),
\]
 where $e^{(m)}(t)$ denotes the $m$ times iterated exponential function.
The first assumption ensures that $Z_{N,\beta}$ is finite. The second
one is of a technical nature and can, most likely, be removed (but
it is satisifed for any reasonable $\phi).$ In particular, if $dV=d\lambda$
and $\phi$ has strictly super-logarithmic growth, i.e. there exists
$\epsilon>0$ such that
\begin{equation}
\phi\geq(1+\epsilon)\log|z|^{2}+O(1),\,\,\,|z|\gg1,\label{eq:strictly super log growth}
\end{equation}
then the first assumption is satisfied, for $\beta$ sufficently large.
Note that we might as well have assumed that $dV=d\lambda$ since
only the combination $e^{-\beta\phi}dV$ appears in the definition
of $\mu_{\beta}^{(N_{k})}.$ However, when investigating the ``infinite
temperature limit'' $(\beta\rightarrow0)$ it will be necessary to
consider the case when $dV$ is a probability measure (see Section
\ref{subsec:The-zero-and intro}). 

In view of the connection to interpolation nodes we will refer a random
configuration $(x_{1},...,x_{N})$ in the probability space $(\C^{n},\mu_{\beta}^{(N_{k})})$
as \emph{``$N_{k}$ random interpolation nodes at inverse temperature
$\beta".$} Since $H^{(N_{k})}(z_{1},...z_{N_{k}})$ is strongly\emph{
repulsive }in the sense that $H^{(N_{k})}(z_{1},...z_{N_{k}})\rightarrow\infty$
if two points merge, random interpolation nodes behave as repulsive
particles.

Using a compactification argument to replace the non-compact space
$\C^{n}$ with the complex projective space $\P^{n}$ the following
stochastic analog of Theorem \ref{thm:conv in law in C^n intro} will
be deduced from Theorem \ref{thm:LDP subharm intro}:
\begin{thm}
\label{thm:conv in law in C^n intro}Given a weighted measure $(\phi,dV)$
which is admissble with respect to $\beta,$ the random measures $\delta_{N_{k}}$
on the probability spaces $((\C^{n})^{N},\mu_{\beta}^{(N_{k})})$
converge in law , as $N_{k}\rightarrow\infty,$ to the unique minimizer
$\mu_{\beta}$ of the following free energy type functional on the
space $\mathcal{M}_{1}(\C^{n})$ of probability measures on $\C^{n}:$
\[
F_{\phi,\beta}(\mu):=E_{\phi}(\mu)-\beta^{-1}S_{dV}(\mu),
\]
 where $E_{\phi}(\mu)$ is the weighted pluricomplex energy of $\mu$
(given by $E(\mu)+\int\phi\mu$ when $\mu$ has compact support).
Equivalently,
\begin{equation}
\mu_{\beta}=e^{\beta(\psi_{\beta}-\phi)}dV\label{eq:mu beta in terms of psi beta in thm intro}
\end{equation}
where $\psi_{\beta}$ is the unique solution of the complex Monge-Ampere
equation 
\begin{equation}
MA(\psi)=e^{\beta(\psi-\phi)}dV\label{eq:ma eq with beta intro}
\end{equation}
 for a function $\psi\in\mathcal{L}_{+}(\C^{n}).$ Moreover, 
\begin{equation}
-\lim_{k\rightarrow\infty}\frac{1}{kN_{k}}\log\int_{\C^{nN_{k}}}|D(z_{1},...,z_{N_{k}})|_{k\phi}^{2\beta/k}dV^{\otimes N}=\inf_{\mu\in\mathcal{M}_{1}(\C^{n})}F_{\phi,\beta}(\mu):=C_{\beta}\label{eq:asymptotics of integral D in thm intro}
\end{equation}
More precisely, the laws of the random measures $\delta_{N}$ satisfy
a Large Deviation Principle (LDP) on $\mathcal{M}_{1}(\C^{n})$ with
speed $\beta N$ and rate functional $F_{\phi,\beta}(\mu)-C_{\phi,\beta}.$
\end{thm}

Note that the deterministic setting, exposed in the previous section,
 appears (formally) in the ``zero temperature case'', i.e. when
$\beta=\infty,$ since $\mu_{\beta}^{(N)}$ then localizes on the
set of minima of $H^{(N)},$ i.e. on the maximizers of the weighted
Vandermonde determinant. One motivation for increasing the temperature,
i.e. introducing randomness into the picture, comes from Monte Carlo
methods, as used extensively in scientific computing \cite{liu} (see
\cite{c-l-e-h,l-n,l-n-a} for the logarithmic case of elliptic Fekete
points on the two-sphere). These are stochastic optimization methods
that can be used for locating approximate minima of a given function
$H^{(N)},$ by fixing a small temperature, i.e. a large $\beta$ (the
corresponding Gibbs measure typically arises as the stationary measure
for the stochastic gradient flow of $H^{(N)},$ as explained in Section
\ref{subsec:Dynamics}).

Theorem \ref{thm:conv in law in C^n intro} applies (and so does Theorem
\ref{thm:LDP subharm intro}) to the more general situation when the
inverse temperature $\beta_{N}$ is allowed to dependent on $N$ as
long as its large-$N$ limit exists:
\begin{equation}
\lim_{N\rightarrow\infty}\beta_{N}:=\beta\in[0,\infty]\label{eq:lim beta N is beta}
\end{equation}
In fact, the proof is considerably simpler when $\beta=\infty$ and
the result then applies more generally to measures $dV$ satisfying
a Bernstein-Markov property (see formula \ref{eq:BM prop intro} below).
This leads to three different regimes (phases):
\begin{itemize}
\item The zero-temperature regime: $\beta=\infty$ (then $\mu_{\beta}$
minimizes the energy $E_{\phi})$
\item The positive temperature regime: $\beta\in]0,\infty[$ (then $\mu_{\beta}$
minimizes the free energy $F_{\beta,\phi}$
\item The infinite temperature regime: $\beta=0$ (then $\mu_{\beta}$ maximizes
entropy $S_{dV})$
\end{itemize}
\begin{example}
The standard quadratic weight $\phi(z):=|z|^{2}$ is admissble for
any $\beta$ when $dV=d\lambda.$ This is equivalent to the non-weighted
setting $\phi=0,$ but with $dV$ taken as a Gaussian measure. Thus,
for $\beta\in]0,\infty[$ the measure $\mu_{\beta}$ interpolates
between the Gaussian measure at $\beta=0$ and the normalized Lesbegue
measure on the unit-ball $B\subset\C^{n}$ at $\beta=\infty$ (compare
Theorem \ref{thm:inf temp intro}). This ``cross-over'' behaviour
has recently been studied when $n=1$ in \cite{b-c-f,a-b}.
\end{example}

\begin{rem}
Establishing the LDP in Theorem \ref{thm:conv in law in C^n intro}
for a fixed weight $\phi$ is, in fact, equivalent to establishing
the logarithmic asymptotics of the corresponding partition functions
for\emph{ all }weights $\phi$ (formula \ref{eq:asymptotics of integral D in thm intro});
see Step 2 in the proof of Theorem \ref{thm:conv in law in C^n intro}
. In other words, establishing the LDP in question is equivalent to
computing the leading logarithmic asymptotics of the weighted $L^{p_{k}}-$norms
of the Vandermonde determinant $D^{(N_{k})},$ as $k\rightarrow\infty$
with $p_{k}k/2\rightarrow\beta.$ Interestingly, in the case when
$n=1$ such $L^{p}-$norms have been computed explicitely for all
fixed $N$ and $p$ for some particular weighted measures $(\phi,dV);$
see the survey \cite{f-o}, explaining the links to the Selberg integral,
random matrix theory and representation theory.
\end{rem}

\subsubsection{\label{subsec:Complex-geometry}Complex geometry}

The main motivation for increasing the temperature does not come from
interpolation theory, but rather from complex geometry (and then the
temperature need not be small). Indeed, when $\phi$ is smooth, the
function $\psi_{\beta}$ determined by the measure $\mu_{\beta}$
by
\begin{equation}
\psi_{\beta}:=\beta^{-1}\log\frac{\mu_{\beta}}{dVe^{-\phi}}\label{eq:def of psi beta intro}
\end{equation}
 (appearing in formula \ref{eq:mu beta in terms of psi beta in thm intro})
is the Kähler potential of a Kähler metric 
\[
\omega_{\beta}:=\frac{i}{2\pi}\partial\bar{\partial}\psi_{\beta}
\]
which defines a twisted Kähler-Einstein metric on $\C^{n}:$ 
\begin{equation}
\mbox{Ric\,}\ensuremath{\omega_{\beta}+\beta\omega_{\beta}=\tau},\label{eq:twisted ke eq intro}
\end{equation}
where the twisting form $\tau$ is explicetly expressed is terms of
the given data $(\beta,\phi,dV)$ (see Example \ref{exa:The-Ricci-curvature}).

Kähler-Einstein metics (i.e. the case when $\eta=0)$ and their twisted
generalization play a key role in current Kähler geometry, in particular
in connection to the Yau-Tian-Donaldson conjecture \cite{do}. Thus
Theorem \ref{thm:conv in law in C^n intro} furnishes a probabilistic
approach for constructing solutions $\omega_{\beta}$ to the twisted
Kähler-Einstein equation \ref{eq:twisted ke eq intro}, by first sampling
configurations of $N$ points with respect to the probability measure
on $(\C^{n})^{N}$ defined by formula \ref{eq:prod meas mubetaN intro}
in order to obtain $\mu_{\beta}$ as $N\rightarrow\infty$ (with overwhelming
probability) and then taking the logarithm to obtain the Kähler potential
$\psi_{\beta}$ defined by formula \ref{eq:def of psi beta intro}.
Moreover, the rate functional $F_{\beta}$ may in this context be
identified with the twisted K-energy functional in Kähler geometry
\cite{berm6}.

The convergence in law in Theorem \ref{thm:conv in law in C^n intro}
implies, in particular, that the expectations $\E(\delta_{N_{k}})$
of the empirical measure $\delta_{N_{k}}$converge weakly towards
$\mu_{\beta},$ as measures on $\C^{n}.$ By symmetry $\E(\delta_{N_{k}})$
may be explicitely expressed as the measure 
\[
\nu_{k}:=\frac{\int_{(\C^{n})^{N-1}}|D(\cdot,z_{2}...,z_{N_{k}})|_{k\phi}^{2\beta/k}dV^{\otimes N}}{\int_{(\C^{n})^{N}}|D(z_{1},z_{2}...,z_{N_{k}})|_{k\phi}^{2\beta/k}dV^{\otimes N}}
\]
on $\C^{n}.$ Theorem \ref{thm:conv in law in C^n intro} thus implies
the following convergence of the explicit sequence $\nu_{k}$ of measures
on $\C^{n}$ and its logarithms: 
\begin{cor}
\label{cor:conv of one pt correl measr etc intro}Given a weighted
measure $(\phi,dV)$ which is admissble with respect to $\beta,$
the sequence of measures converge in the weak topology of measures
on $\C^{n}$ towards $\mu_{\beta}.$ As a consequence, the corresponding
sequence of Kähler potentials 
\[
\psi_{\beta}^{(N_{k})}:=\frac{1}{\beta}\log\frac{\nu_{k}}{dVe^{-\phi}}
\]
 converge in $L_{loc}^{1}(\C^{n})$ towards the Kähler potential $\psi_{\beta}$
of a twisted Kähler-Einstein metric $\omega_{\beta}$ (solving the
equation \ref{eq:twisted ke eq intro}).
\end{cor}

As will be explained in Section \ref{subsec:Comparison-with-determinantal}
the convergence of $\psi_{\beta}^{(N_{k})}$ towards $\psi_{\beta}$
may be interpreted as a ``positive temperature analogy'' of a classical
convergence result of Siciak.

The existence of solutions to Monge-Ampère equations of the form \ref{eq:ma eq with beta intro}
on a compact Kähler manifold $X$ was first established by Aubin \cite{au}
and Yau \cite{y}, when $\beta>0.$ The case $\beta=0$ (i.e. the
Calabi-Yau equation) was settled in Yau's solution of the Calabi conjecture
\cite{y}. Interestingly, the minus the ``inverse temperature''
$-\beta$ coincides with the deformation parameter in Aubin's continuity
method for constructing solutions to the equation \ref{eq:ma eq with beta intro}
on a compact Kähler manifold $X$ for $\beta$\emph{ negative \cite{au}.
}This is motivated by the Yau-Tian-Donaldson conjecture concerning
the existence of Kähler-Einstein metrics with \emph{positive} Ricci
curvature, which very recently has been settled for smooth compact
Fano manifolds $X$ (see the survey \cite{do}). From a statistical
mechanical point of view the case $\beta<0$ introduces an additional
fourth phase, complementing the three phases described above. In practise,
switching the sign of $\beta$ equivalently corresponds to switching
the sign of the corresponding interaction energy, thus replacing the
repulsive interaction energy $E^{(N)}$ with the opposite \emph{attractive}
energy 
\begin{equation}
E_{attr}^{(N)}(z_{1},...,z_{N}):=\frac{1}{N_{k}k}\log\left|D^{(N_{k})}(z_{1},...z_{N_{k}})\right|^{2}\label{eq:attr E N}
\end{equation}
 at positive inverse temperature $\gamma=-\beta.$ This attractive
setting appears to be extremely challenging (see Section \ref{subsec:The-case-of-neg}).

\subsection{\label{subsec:The-zero-and intro}The zero and infinite temperature
limits and the Calabi-Yau equation}

Let us come back to the original ``zero temperature regime'' $(\beta=\infty)$
and explain how it reappears when $\beta\rightarrow\infty.$ Since
the measure $\mu_{\beta},$ emerging in the large $N-$limit described
in Theorem \ref{thm:conv in law in C^n intro}, is a minimizer of
the free energy functional $F_{\beta,\phi}$ it is not hard to see
that $\mu_{\beta}$ converges, as $\beta\rightarrow0,$ to the minimizer
of $E_{\phi},$ which coincides with weighted equilibrium measure
$\mu_{\phi}$ (formula \ref{eq:def of MA phi intro}) when $\beta\rightarrow\infty.$
Moreover, the corresponding result also holds on the level of potentials.
In other words, the corresponding Kähler potentials $\psi_{\beta}$
(formula \ref{eq:def of psi phi intro}) converge to the corresponding
extremal function $\psi_{\phi},$ which is $C^{1,1}-$smooth when
$\phi$ is smooth (but not $C^{2}-$smooth, in general). More precisely: 
\begin{thm}
Assume that $\phi$ has strictly super-logarithmic growth (formula
\ref{eq:strictly super log growth}), that$\phi$ is in $C_{loc}^{2}(\C^{n})$
and that $dV$ has a positive density in $C_{loc}^{2}(\C^{n}).$ Then
$\psi_{\phi}$ is in $C_{loc}^{1,1}(\C^{n})$ and
\[
\lim_{\beta\rightarrow\infty}\psi_{\beta}=\psi_{\phi}
\]
in the Hölder space $C_{loc}^{1,\alpha}(\C^{n})$ for any $\alpha<1.$ 
\end{thm}

The result follows from a straight-forward modification of the proof
in the setting of polarized manifolds \cite{berm11}. The $C^{1,1}-$regularity
implies that
\[
\mu_{\phi}=1_{S_{\phi}}MA(\phi),
\]
which typically has a sharp discontinuity across boundary points of
the support $S_{\phi},$ which is a compact subset of $\C^{n}$ \cite{berm 1 komma 1}.
This should be contrasted with the fact that for any finite $\beta$
the support of the Monge-Ampère measure $MA(\psi_{\beta})$ is all
of $\C^{n}.$ Thus, from a statistical mechanical point of view, the
``zero-temperature limit'' $\beta\rightarrow\infty$ is reminiscent
of a second order ``gas-liquid'' phase transition. 

More generally, one can consider the situation where the volume form
$dV$ is replaced by a general measure on $\C^{n},$ that we shall
still denote by $dV.$ We denote by $K$ the support of $dV,$ which,
for simplicity will be assumed to be compact. From the pluripotential
point of view a natural condition is that the measure $dV$ is\emph{
determining} for $(K,\phi),$ i.e. 
\[
\psi\leq\phi\,\,\text{almost everywhere wrt }dV\implies\psi\leq\phi\,\,\text{on\,\ensuremath{K}}
\]
The following result appears to be new, also when $n=1;$ see \cite{berm14},
where the setting Coulomb gases is considered (the notion of convergence
in energy, coinciding with Sobolev $H^{1}-$convergence when $n=1,$
is recalled in Section \ref{sec:Proofs-for-the limit beta}). 
\begin{thm}
\label{thm:zero temp limit singular case intro}Let $\phi$ be a continuous
function on $\C^{n}$ and assume that $dV$ does not charge pluripolar
subsets and is determining wrt $(K,\phi).$ Then, as $\beta\rightarrow\infty,$
\[
\psi_{\beta}\rightarrow\psi_{(K,\phi)}
\]
 in energy. Moreover, if $dV$ is determining wrt $(K,\phi)$ for
any $\phi\in C(K),$ then 
\[
F_{\beta}:=E-\beta^{-1}S_{dV}\rightarrow E
\]
in the sense of Gamma-convergence on $\mathcal{M}_{1}(K).$ \footnote{Conversely, if the Gamma-convergence holds, then $dV$ is determining
wrt $(K,\phi)$ for any $\phi\in C(K),$ as shown in \cite{berm14}.}
\end{thm}

As pointed out in Section \ref{subsec:Comparison-with-determinantal}
the convergence of $\psi_{\beta}$ towards Siciak's weighted extremal
function $\psi_{(K,\phi)}$ (formula \ref{eq:def of psi K phi}) can
be interpreted as a transcendental analog of a classical result of
Siciak \cite{si1} and its $L^{2}-$version in \cite{b-sh}.

The Gamma-convergence in the previous theorem may be equivalently
formulated as the following approximation property of independent
interest (see \cite{berm14}):
\begin{cor}
Let $\mu_{0}$ be a measure on $\C^{n}$ with compact support $K$
and assume that that $\mu_{0}$ does not charge pluripolar subsets
and is determining for $(K,\phi)$ for any $\phi\in C(K).$ Then,
for any given probability measure $\mu$ on $K$ there exists a sequence
$\mu_{j}$ of probability measures, absolutely continuous wrt $\mu_{0},$
converging weakly to $\mu$ and with the property that $E(\mu_{j})\rightarrow E(\mu).$ 
\end{cor}

\begin{rem}
In the case when $n=1$ (or more generally, in the ``linear setting''
discussed in Section \ref{eq:mean field eq in heur}) this type of
approximation property is used as a hypothesis to derive large deviation
results for Gibbs measures in the ``zero-temperature regime'', where
$\beta_{N}\rightarrow\infty$ \cite{c-g-z,d-l-r,berm10,gz}. These
connections are discussed in \cite{berm14}. We note that when $\mu_{0}$
is supported in $\R^{n}$ $\mu_{0}$ is determining for $(K,0)$ iff
it is determining for $(K,\phi)$ for all $\phi\in C(K)$ (using that,
by the Stone-Weierstrass theorem, $-\phi$ is the uniform limit of
$\log|p_{k}|^{2}$ for some $p_{k}\in\mathcal{P}_{k}(\C^{n})).$
\end{rem}

In the opposite ``infinite temperature limit'', where $\beta\rightarrow0,$
one has to assume that $dV$ is a probability measure. Then, under
appropriate growth and regularity assumptions on $dV,$ it follows
from well-known results that $\psi_{\beta}$ converges to a particular
solution $\psi_{0}$ of the Calabi-Yau equation \ref{eq:calabi-yau eq intro}
with $\mu=dV$ (see \cite{e-g-z}). In Section \ref{sec:Proofs-for-the limit N}
we will show that the convergence holds in great generality, only
assuming that the probability measure $dV$ has finite weighted pluricomplex
energy $E_{\phi}:$
\begin{thm}
\label{thm:inf temp intro}Let $dV$ be a probability measure on $\C^{n}$
and fix a continuous function $\phi$ on $\C^{n}$ with super logarithmic
growth \ref{eq:super log growth}. Assume that $dV$ has finite weighted
pluricomplex energy, $E_{\phi}(dV)<\infty.$ Then the functions $\psi_{\beta}$
converge to $\psi_{0}$ in energy, as $\beta\rightarrow0,$ where
$\psi_{0}$ is the unique finite energy solution in $\mathcal{L}(\C^{n})$
to 
\[
MA(\psi)=dV
\]
satisfying the normalization condition 
\begin{equation}
\int_{X}(\psi_{0}-\phi)dV=0\label{eq:norm cond in thm zero}
\end{equation}
\end{thm}

\subsection{\label{subsec:Comparison-with-determinantal}Comparison with determinantal
point processes, orthogonal polynomials and Bergman kernels}

Consider now the case when the inverse temperature $\beta$ is taken
to depend on the number of points $N_{k}$ as 
\[
\beta_{N_{k}}=k,
\]
which means that the density of the probability measure $\mu_{\beta_{N_{k}}}^{(N_{k})}$
(formula \ref{eq:prod meas mubetaN intro}) is a squared Vandermonde
determinant. In this case $\mu_{\beta_{N_{k}}}^{(N_{k})}$ may be
alternatively expressed as 
\[
\mu_{\beta_{N_{k}}}^{(N_{k})}=\frac{1}{N!}\det\left(\mathcal{K}_{k}(x_{i},x_{j})\right)_{i,j\leq N_{k}}dV^{\otimes N_{k}},
\]
 where $\mathcal{K}_{k}(z,w)=K_{k}(z,w)e^{-k\phi(z)}e^{-k\phi(w)}$
and $K_{k}$ is the\emph{ Bergman reproducing kernel} of the Hilbert
space space $\mathcal{P}_{k}(\C^{n}),$ endowed with the the scalar
product defined by the weighted $L^{2}-$norm 
\[
\left\Vert p_{k}\right\Vert _{L^{2}(\phi,dV)}:=\left(\int_{\C^{n}}|p_{k}|^{2}e^{-k\phi}dV\right)^{1/2}
\]
 Concretely, 

\[
K_{k}(z,w):=\sum_{i=1}^{N_{k}}p_{i}^{(k)}(z)\overline{p_{i}^{(k)}(w)},
\]
in terms of a fixed orthonormal base $\{p_{i}^{(k)}\}_{i=1}^{N_{k}}$
of polynomials in the Hilbert space$\mathcal{P}_{k}(\C^{n}).$ In
probabilistic terminology this means that the probability measure
$\mu_{\beta_{N_{k}}}^{(N_{k})}$ then defines a \emph{determinantal
random point process with $N_{k}$ particles on $\C^{n}$ }(see Section
\ref{subsec:Determinantal-point-processes} for background on determinantal
point processes). From a statistical mechanical point this case falls
into the ``zero temperature regime'' (see formula \ref{eq:lim beta N is beta})
and, as shown in \cite{berm 1 komma 5}, the convergence in law in
Theorem \ref{thm:conv in law in C^n intro} then holds more generally
as soon as the weighted measure $(dV,\phi)$ satisfies a weighted
Bernstein-Markov property with respect to the weighted set $(K,\phi).$
This means that for any given $\epsilon>0$ there exists a constant
$C$ such that 
\begin{equation}
\sup_{K}|p_{k}|^{2}e^{-k\phi}\leq Ce^{\epsilon k}\left\Vert p_{k}\right\Vert _{L^{2}(dV,\phi)}\label{eq:BM prop intro}
\end{equation}
 for all $p_{k}\in\mathcal{P}_{k}(\C^{n})$ and positive integers
$k$ (see \cite{b-l-p-w} for a survey of the Bernstein-Markov-property).
Moreover, defining $\psi_{\beta_{N_{k}}}^{(N_{k})}$ as in Corollary
\ref{cor:conv of one pt correl measr etc intro}, we then have

\begin{equation}
\psi_{\beta_{N_{k}}}^{(N_{k})}\rightarrow\psi_{(K,\phi)},\,\,\,\,k\rightarrow\infty\label{eq:conve of psi beta N}
\end{equation}
in $L^{\infty}(\C^{n}),$ where we recall that $\psi_{(K,\phi)}$
denotes Siciak's weighted extremal function for the weighted set $(K,\phi)$
(formula \ref{eq:def of psi K phi}). Since, under appropriate conditions,
\[
\psi_{\beta}\rightarrow\psi_{(K,\phi)},\,\,\,\,\beta\rightarrow\infty
\]
(by Theorem \ref{thm:zero temp limit singular case intro}) the convergence
\ref{eq:conve of psi beta N} can be interpreted as zero-temperature
analog of Corollary \ref{cor:conv of one pt correl measr etc intro}.
To see that the convergence \ref{eq:conve of psi beta N} holds first
observe that $\psi_{\beta_{N_{k}}}^{(N_{k})}$ may be re-expressed
as
\begin{equation}
\psi_{\beta_{N_{k}}}^{(N_{k})}:=k^{-1}\log K_{k}(z,z)=k^{-1}\log\sup\frac{|p_{k}|^{2}}{\left\Vert p_{k}\right\Vert _{L^{2}(\phi,dV)}^{2}}\label{eq:psi k as log K}
\end{equation}
(see Section \ref{subsec:Determinantal-point-processes}) which converges
towards $\psi_{(K,\phi)}$ in $L^{\infty}(\C^{n}),$ as first shown
in a different probabilistic setting of random polynomials in \cite{b-sh}.
Indeed, by the assumed weighted Bernstein-Markov-property of $(dV,\phi)$
we may as well replace the $L^{2}-$norm in the right hand side of
formula \ref{eq:psi k as log K} with an $L^{\infty}-$norm on $K.$
The $L^{\infty}-$convergence towards Siciak's extremal function then
follows from a classical result of Siciak \cite{si1}. 
\begin{rem}
In the classical setting of orthogonal polynomials on $\R,$ where
$K_{k}(z,z)$ is usually called the Christoffel-Darboux kernel, asymptotics
of $k^{-1}\log K_{k}(z,z)$ are closely related to the notion of ``$k$
th root asymptotics of orthonormal polynomials'' \cite{st-t}. Moreover,
when $\phi=0$ and $K$ is regular in the sense of potential theory,
the Bernstein-Markov-condition on $dV$ is equivalent to the condition
that$dV$ has regular asymptotic distribution in the sense of \cite[Chapter 4]{st-t}.
The case $n=1$ also naturally appears in Random Matrix Theory, as
recalled in Section \ref{subsec:Determinantal-point-processes}.
\end{rem}

The discussion above suggests viewing the function $\psi_{\beta},$
solving the Monge-Ampère equation \ref{eq:ma eq with beta intro},
as a transcendental analog of $k^{-1}\log K_{k}(z,z),$ with $\beta$
playing the role of the degree $k.$ Accordingly, the asymptotics
in Theorem \ref{thm:zero temp limit singular case intro} can be viewed
as a transcendental analog of the classical $k$ th root asymptotics
for orthogonal polynomials (this point of view is developed from a
complex geometric point of view in \cite{berm11}, where\textbf{ $k^{-1}\log K_{k}(z,z)$}
appears as a Bergman metric). This interpretation is reinforced by
the result in \cite{b-b-w} showing that the condition that $(dV,\phi)$
is determining for $(K,\phi),$ appearing in Theorem \ref{thm:zero temp limit singular case intro}
is equivalent to the following ``transcendental'' Bernstein-Markov-property
of $(dV,\phi):$ for any given $\epsilon>0$ there exists a constant
$C>0$ such that
\[
\sup_{K}e^{\beta(\psi-\phi)}e\leq Ce^{\epsilon\beta}\int e^{\beta(\psi-\phi)}dV
\]
 for any $\psi\in\mathcal{L}(\C^{n})$ and $\beta>0$ (which clearly
implies the ordinary Bernstein-Markov-property \ref{eq:BM prop intro}).
\begin{rem}
When $dV$ is supported in $[0,1]\subset\C$ (and $\phi=0)$ the condition
that $dV$ is determining coincides with \emph{Ullman's criterion}
for a measure $dV$ on $\R$ to have regular asymptotic distribution
(discussed in \cite[Chapter 4]{st-t}).
\end{rem}

\subsection{\label{subsec:Tropicalization-at-negative}Tropicalization at negative
temperature, toric geometry and optimal transport}

In complex geometry the process of tropicalization can sometimes to
be invoked to replace an elusive complex geometric problem by a more
tractable convex geometric one (see, for example, \cite{i-m} for
application to enumerative problems in algebraic geometry). In \cite{ber-o,berm13}
this philosophy is applied to the challenging \emph{negative }temperature
regime of the present setting, $\beta<0$ (discussed in the end of
Section \ref{subsec:Complex-geometry} above and in Section \ref{subsec:The-case-of-neg}).
The motivation comes from the Yau-Tian-Donaldson conjecture and, in
particular, the construction of Kähler-Einstein metrics with \emph{positive}
Ricci curvature on toric varieties, as explained in \cite{berm13}.
Here we will focus on the connections to the present setting. 

We first recall that, from an algebraic point of view, tropicalization
amounts to replacing the ring $\C$ with the tropical semi-ring over
$\R,$ where addition of $a$ and $b$ is defined by $\max\{a,b\}$
and multiplication by ordinary addition. Accordingly, a polynomial
$p(z)$ on $\C^{n}$ (or more, generally, a Laurent polynomial on
$\C^{*n},$ is replaced by a piece-wise affine convex function $\phi(x)$
on $\R^{n}:$ 
\begin{equation}
p(z)=\sum_{m}c_{m}z^{m}\mapsto\max_{m}x\cdot m\label{eq:p maps to pl}
\end{equation}
(up to an additive constant) where the max ranges over all $m\in\Z^{n}$
such that $c_{m}\neq0.$ Analytically, tropicalization can be viewed
as a the limit of the complex geometric setting, where $\C^{*n}$
is collapsed onto $\R^{n}$ through the following map: 
\begin{equation}
\C^{*n}\rightarrow\R^{n},\,\,\,z\mapsto x:=\mbox{Log}(z):=(\log(|z_{1}|^{2}),...,\log(|z_{n}|^{2})).\label{eq:Log map intro}
\end{equation}
Indeed, the association \ref{eq:p maps to pl} amounts to replacing
the plurisubharmonic functions $\psi(z)=\log|p|^{2}$ on $\C^{*n}$
by the convex function $\psi_{trop}(x)$ on $\R^{n}$ defined by
\begin{equation}
\psi_{trop}(x):=\lim_{c\rightarrow\infty}c^{-1}\psi(z^{m})/2\label{eq:trop as limit of scaling of psi}
\end{equation}
Expanding the Vandermonde determinant as a sum over permutations $\sigma\in S_{N}$
reveals that the tropicalization of the attractive determinantal interaction
energy $E_{attr}^{(N_{k})}$ on $\C^{nN_{k}},$ appearing in formula
\ref{eq:attr E N}, is the piece-wise affine convex function $E_{trop}^{(N_{k})}$
on $\R^{nN_{k}}$ defined by
\begin{equation}
E_{trop}^{(N_{k})}(x_{1},...,x_{N_{k}}):=\frac{1}{N}\max\sum_{\sigma\in S_{N_{k}}}x_{1}\cdot p_{\sigma(1)}+\cdots+x_{N}\cdot p_{\sigma(N_{k})},\label{eq:def of tropical E}
\end{equation}
 where $p_{1},...,p_{N}$ are are points of $P\cap(k^{-1}\Z)^{n},$
where $P$ denotes the convex polytope $P\Subset\R^{n}$ defined by
the unit-simplex. 

\subsubsection{The setting of a convex body and toric varieties}

Before proceeding to defining the corresponding tropicalization of
the Gibbs measure defined by \ref{eq:prod meas mubetaN intro}, it
will be more natural to consider a more general class of Gibbs measures
on $(\C^{*n})^{N_{k}},$ associated to a given convex body $P$ in
$\R^{n}.$ Up to a harmless translation we may as well assume that
$0$ is contained in the interior of $P.$ The setup in the previous
sections may then be generalized by replacing the space $\mathcal{P}_{k}(\C^{n})$
by the space $\mathcal{P}_{kP}(\C^{*n})$ of Laurent polynomials on
the complex torus $\C^{*n},$ spanned by all multinomials $z^{m}(=z_{1}^{m_{1}}\cdots z_{n}^{m_{n}})$
with exponents $m\in kP\cap\Z^{n}.$ Thus the dimension $N_{k}$ of
$\mathcal{P}_{kP}(\C^{*n})$ coincides with the number of integer
points of the convex body $P.$ Similarly, the Lelong class $\mathcal{L}(\C^{n})$
is replaced by the space $\mathcal{L}_{P}(\C^{*n})$ of all psh functions
$\psi$ on $\C^{*n}$ satisfying the following growth condition:
\[
\psi_{P}\leq\text{Log \ensuremath{^{*}\phi_{P}+O(1),\,\,\,\phi_{P}(x)=\sup_{p\in P}x\cdot p}}
\]
(the subspace $\mathcal{L}_{P,+}(\C^{*n})$ is defined by requiring
that equality holds above).
\begin{rem}
\label{rem:toric}When $P$ is a rational polytope this more general
setting corresponds to replacing $\P^{n}$ with a toric variety $X_{P}$
compactifying $\C^{*n}$ (compare Section \ref{subsec:Compactification-of-}).
The case of a general convex body $P$ was introduced in \cite{berm 1 komma 1}
in the context of Bergman kernel asymptotics. See aso \cite{b-b-l}
for a systematic study of the class $\mathcal{L}_{P}(\C^{*n})$ in
connection to Fekete points.
\end{rem}

Denoting by $dV_{\C^{*n}}$ the standard $\C^{*n}-$invariant volume
form on $\C^{*n}$ it is tempting to try to define a probability measure
on $(\C^{*n})^{N_{k}}$ as the the normalization of the following
measure on $(\C^{*n})^{N_{k}}$ 
\begin{equation}
|D^{(N_{k})}(z_{1},...,z_{N_{k}})|^{2\beta/k}dV_{\C^{*n}}^{\otimes N}\label{eq:non-normal vandermonde measure on torus}
\end{equation}
 where $D^{(N_{k})}$ now denotes the Vandermonde determinant associated
to the space $\mathcal{P}_{kP}(\C^{*n}).$ Heuristically, i.e. assuming
that the analog of Theorem \ref{thm:conv in law in C^n intro} holds
also for $\beta$ negative, one would expect that the large $N-$limit
is described by solutions $\psi_{\beta}$ on $\C^{*n}$ to the equation
\begin{equation}
MA(\psi)=e^{\beta\psi}dV_{\C^{*n}}^{\otimes N}\label{eq:ke eq for psi on torus}
\end{equation}
Geometrically, this means that $\psi_{\beta}$ is the Kähler potential
of a Kähler-Einstein metric $\omega_{\beta}$ with constant Ricci
curvature: 
\begin{equation}
\mbox{Ric\,}\ensuremath{\omega_{\beta}=-\beta\omega_{\beta}}.\label{eq:ke eq for omega beta on torus}
\end{equation}
 Since we have assumed that $\beta$ is negative this corresponds
to constant \emph{positive }Ricci curvature. Moreover, by a scaling
argument one can assume that $\beta=-1.$ 

\subsubsection{The tropical Gibbs measure}

Unfortunately, the Vandermonde measure in formula \ref{eq:non-normal vandermonde measure on torus}
always has infinite mass, i.e. the corresponding normalization constant
$Z_{N,\beta}$ is infinite, as follows from symmetry considerations.
Indeed, the measure is invariant (modulo a multiplicative constant)
under the diagonal action of $\C^{*n}$ on $(\C^{*n})^{N_{k}}.$ Since
$\C^{*n}$ is non-compact this implies that $Z_{N,\beta}$ is infinite.
This is a reflection of the fact that the group $\C^{*n}$ acts on
the solution space to the Kähler-Einstein equation \ref{eq:ke eq for psi on torus}.
The symmetry in question may be broken by turning to the weighted
setting, i.e. by replacing $|D^{(N_{k})}|$ with the weighted Vandermonde
determinant $|D^{(N_{k})}|_{k\phi},$ for a given continuous function
$\phi(z)$ on $\C^{*n}.$ We then consider the corresponding measure
$\mu_{\beta}^{(N)}$ of the form \ref{eq:prod meas mubetaN intro},
for an appropriate volume form $dV$ on $\C^{*n}.$ Motivated by Kähler-Einstein
geometry it is natural to take

\begin{equation}
dV=e^{-\phi}dV_{\C^{*n}},\label{eq:volume form with phi in Lelong}
\end{equation}
and assume that $\phi\in\mathcal{L}_{P,+}(\C^{*n}).$ This ensures
that one gets backs the Vandermonde measure, formula \ref{eq:non-normal vandermonde measure on torus},
when $\beta=-1$ and that $Z_{N,\beta}$ is finite for $\beta$ sufficiently
close to $0.$ Thus the weighted setting provides a regularization
of the original divergent setting. The corresponding probability measure
$\mu_{\beta}^{(N)}$ can, equivalently, be viewed as the Gibbs measure
at unit-temperature associated to $N$ times the interaction energy
\[
E_{\beta}^{(N_{k})}(z_{1},...z_{N_{k}}):=-\beta E_{attr}^{(N_{k})}(z_{1},...z_{N})+(1+\beta)\sum_{i=1}^{N_{k}}\phi(z_{i})
\]
 and volume form $dV_{\C^{*n}}^{\otimes N},$ i.e. 
\[
\mu_{\beta}^{(N)}=\frac{1}{Z_{N,\beta}}e^{-NE_{\beta}^{(N_{k})}}dV_{\C^{*n}}^{\otimes N}.
\]
 In view of formula \ref{eq:trop as limit of scaling of psi} and
\ref{eq:def of tropical E}, the tropicalization of $E_{\gamma}^{(N_{k})}$
is given by 
\[
E_{\beta,trop}^{(N_{k})}(x_{1},...x_{N_{k}}):=-\beta E_{trop}^{(N_{k})}(x_{1},...,x_{N_{k}})+(1+\beta)\sum_{i=1}^{N_{k}}\phi_{P}(x_{i})
\]

Finally, one is thus led to define the tropicalization of the Gibbs
measure $\mu_{\beta}^{(N)}$ on $(\C^{*n})^{N_{k}},$ associated to
a given convex body $P,$ as the following Gibbs measure on $(\R^{n})^{N_{k}}:$
\[
\mu_{\beta,trop}^{(N)}:=\frac{1}{Z_{N_{k},\beta}}e^{-NE_{\beta,trop}^{(N_{k})}}dx^{\otimes N_{k}},
\]

The following result \cite{berm13} shows that the convergence in
Theorem \ref{thm:conv in law in C^n intro} can be extended to the
negative temperature regime, down to a critical negative inverse temperature
$\beta_{c}(=-R_{P}),$ if the original complex problem is replaced
by its tropical analogue. {[}See also the ArXiv version of \cite{ber-o},
where the proof of the result originally appeared{]}.
\begin{thm}
\label{thm:tropical}Let $P$ be a convex body in $\R^{n}$ containing
$0$ in its interior. Then, 
\begin{itemize}
\item For $k$ sufficiently large, the corresponding Gibbs measure $\mu_{\beta,trop}^{(N_{k})}$
is a well-defined probability measure, i.e. $Z_{N_{k},\beta}<\infty$
iff $\beta<-R_{P},$ where $R_{P}\in]0,1]$ is the following invariant
of $P:$ 
\begin{equation}
R_{P}:=\frac{\left\Vert q\right\Vert }{\left\Vert q-b_{P}\right\Vert },\label{eq:inv R of conv bod}
\end{equation}
 where $q$ is the point in $\partial P$ where the line segment starting
at the barycenter $b_{P}$ of $P$ and passing through $0$ meets
$\partial P.$ 
\item If $\beta>-R_{P},$ then the random measure $\delta_{N_{k}}$ on $((\R^{n})^{N_{k}},\mu_{\beta,trop}^{(N_{k})})$
converges in law, as $N\rightarrow\infty,$ towards $\mu_{\beta}=e^{\beta(u_{\beta}-\phi_{P})}e^{-\phi_{p}}dx,$
where $u_{\beta}(x)$ is the unique convex solution to the following
real Monge-Ampère equation on $\R^{n}$ 
\begin{equation}
\det(\nabla^{2}u)=e^{\beta(u-\phi_{P})}e^{-\phi_{p}}dx\label{eq:real ma eq with gamma}
\end{equation}
 subject to the ``second boundary condition'' 
\begin{equation}
\overline{(\nabla u)(\R^{n})}=P\label{eq:sec bound con}
\end{equation}
\item If $R_{P}=1,$ i.e. if $0$ is the barycenter of $P,$ then $u_{\beta}$
converges uniformly on $\R^{n},$ as $\beta\rightarrow-1,$ to a smooth
solution $u_{1}$ of the equation 
\begin{equation}
\det(\nabla^{2}u)=e^{-u},\label{eq:real ma eq ke}
\end{equation}
 subject to \ref{eq:sec bound con}.
\end{itemize}
\end{thm}

The pull-back $\psi$ to $\C^{*n},$ under the map Log of any solution
$u$ of the equation \ref{eq:real ma eq ke} solves the Kähler-Einstein
equation \ref{eq:ke eq for psi on torus} and is in $\mathcal{L}_{P}(\C^{*n}).$
Moreover, when $P$ is a rational polytope, the last property ensures
that $\omega$ extends to define a Kähler-Einstein metric on the toric
(log) Fano variety $X_{P}$ compactifying $\C^{*n}.$ By the toric
case of the Yau-Tian-Donaldson conjecture $X_{P}$ admits a Kähler-Einstein
metric iff $b_{P}=0.$ Thus Theorem \ref{thm:tropical} provides a
probabilistic construction of Kähler-Einstein metrics on toric Fano
varieties, when such metrics exist. 

For a general convex body $P$ it follows from \cite{ber-ber} that
the equation \ref{eq:real ma eq ke} admits a solution iff $b_{P}=0.$
Moreover, the solution is unique modulo the action of $\R^{n}$ by
translations. In fact, the previous theorem holds more generally when
$\phi_{P}$ is replaced by a given convex weight function $\phi$
such that $\phi=\phi_{P}+O(1).$ When $b_{p}=0$ the corresponding
limit of $u_{\beta},$ as $\beta\rightarrow-1,$ singles out a particular
solution of the equation \ref{eq:real ma eq ke}, depending on the
choice of weight function $\phi.$ From the point of view of statistical
mechanics this can be interpreted as an instance of \emph{spontaneous
symmetry breaking }(see the appendix). Moreover, the critical inverse
temperature $-R_{P}$ is a tropical (and higher dimensional analog)
of the critical negative temperature of the Coulomb gas in $\C,$
exhibited in \cite{k2,clmp}. 

\subsubsection{\label{subsec:Optimal-transport}Optimal transport formulation}

The tropicalized setting may be formulated in terms of optimal transport
theory. Indeed, setting 

\[
\mu_{0}=N^{-1}\sum_{i=1}^{N}\delta_{x_{i}},\,\,\,\mu_{1}:=N^{-1}\sum_{i=1}^{N}\delta_{p_{i}},
\]
the tropicalization $E_{trop}^{(N)}(x_{1},...,x_{N})$ of the attractive
determinantal interaction energy $E_{attr}(z_{1},...,z_{N})$ (formula
\ref{subsec:Optimal-transport}) coincides with minus the minimal
cost $C(\mu_{0},\mu_{1})$ of transportation between the probability
measures $\mu_{0}$ and $\mu_{1},$ with respect to the standard quadratic
cost function $c(x,y):=-x\cdot y:$ 

\begin{equation}
C(\mu_{0},\mu_{1}):=\inf_{\gamma\in\Pi(\mu_{0},\mu_{1})}C[\gamma]\,\,\,C[\gamma]:=\int c\gamma\label{eq:def of optimal cost}
\end{equation}
where the minimum runs over all transport plans $\gamma,$ i.e. all
probability measures on $\R^{n}\times\R^{n}$ with marginals $\mu_{0}$
and $\mu_{1},$ respectively. It follows readily from this interpretation
of $E_{trop}^{(N)}$ that, denoting by $\nu_{P}$ the uniform probability
measure on $P,$ 
\[
N^{-1}\sum_{i=1}^{N}\delta_{x_{i}}\rightarrow\mu\implies-E_{trop}^{(N)}(x_{1},...x_{N})\rightarrow C(\mu,\nu_{P}),
\]
 if the convergence towards $\mu$ holds weakly and moreover the first
moments converge. The convergence towards $C(\mu,\nu_{P})$ is a tropical
analog of the Gamma-convergence of $-E_{attr}^{(N)}(z_{1},...,z_{N})$
towards $E(\mu)$ (formula \ref{eq:Gamma conv towards E intro}).
In fact, when $\mu$ is invariant under the torus $T^{n}$ the minimal
cost $C(\mu,\nu_{P})$ coincides with the (non-weighted) pluricomplex
energy $E(\mu),$ as observed in \cite{berm7} (see also \cite{ber-ber}).
This is a reflection of the, essentially well-known fact \cite{berm12},
that a convex function $u$ on $\R^{n}$ solves the real Monge-Ampère
equation 
\[
\det(\nabla^{2}u)=\mu,
\]
 (which is the infinite temperature case, $\beta=0,$ of the equation\ref{eq:real ma eq with gamma})
iff the corresponding gradient map 
\[
T:\,\,x\mapsto\nabla u(x)
\]
 is the optimal map $T$ transporting $\mu$ to the uniform measure
$\nu_{P}$ on the convex body $P.$ This means that $\gamma:=(I\times T)_{*}\mu$
realizes the infimum in formula \ref{eq:def of optimal cost}.

\subsection{Further references and developments}

Before zooming in on the details in the account above, let us mention
some other related developments.

\subsubsection{Fekete points}

The rate of convergence for the Fekete points, as $N\rightarrow\infty$
has been made quantitative in \cite{l-o,d-m-n1}. For example, in
the case when $(X,L)$ is a polarized manifold and $L$ is endowed
with a smooth and strictly positively curved metric, the following
sharp rate is established in \cite{l-o} when $K=X$:
\begin{equation}
d_{1}(\delta_{N},\mu_{K})\sim k^{-1/2},\label{eq:quant rate}
\end{equation}
 where $d_{1}$ denotes the Wasserstein $L^{1}-$distance on $\mathcal{M}_{1}(X).$
Generalizations to less regular metrics and with $K$ an appropriate
proper compact subset are given in \cite{d-m-n1}. Recently, the problem
of numerical simulation of approximate Fekete points has been studied;
see \cite{b-d-s-v} and references therein. 

\subsubsection{Large deviations for Vandermonde determinants}

In the case $\beta=\infty$ asymptotics of integrals of Vandermonde
determinants and large deviations in $\C^{n}$ have also been obtained
in \cite{b-l}. The relevance of the Bernstein-Markov condition on
the reference measure $dV$ in this context originates in \cite{b-l0}.
In the case of a polarized manifold $(X,L)$ quantitative rates of
convergences for the corresponding empirical measures have been obtained
in \cite{l-o,d-n}. For example, as shown in \cite{l-o}, when $(X,L)$
is a polarized manifold and $L$ is endowed with a smooth and strictly
positively curved metric a rate as in formula \ref{eq:quant rate}
holds when $d_{1}(\delta_{N},\mu_{K})$ is replaced by the corresponding
expectations $\E d_{1}(\delta_{N},\mu_{K}).$ Similar results hold
if a quantitative version of the Bernstein-Markov property is assumed
\cite{d-n} (which, for example, is satisfied if $K$ is a smooth
totally real submanifold). It would be interesting establish quantitative
rates in the case when $\beta<\infty.$ In another direction, in the
case when $(X,L)$ is an abelian variety (i.e $X$ is a compact complex
torus) a real analog of Theorem \ref{thm:LDP for temp deform on polar},
which holds also for negative $\beta,$ has been established in \cite{hu}.

\subsubsection{The limits $\beta\rightarrow0$ and $\beta\rightarrow\infty$}

When $dV$ is a continuous volume form, parametrized complex Monge-Ampère
equations of the form \ref{eq:MA eq with beta in section proof beta}
have previously been used in \cite{e-g-z} to give a new proof of
the existence of a continuous solutions of the corresponding Calabi-Yau
equation. Briefly, the solution $\varphi_{\beta}$ for $\beta>0$
is first obtained as the sup of viscosity subsolutions. Then it is
shown that $\varphi_{\beta}$ converges to a solution of the Calabi-Yau
equation, when $\beta\rightarrow0.$ In the present setting the role
of viscosity subsolutions is, loosely speaking, played by the explicit
Kähler potentials in Corollary \ref{cor:conv of one pt correl measr etc intro}.
This probabilistic approach thus appears to be very different from
the viscosity approach \footnote{On the other hand, it is somewhat reminiscent of the ``vanishing
viscosity method'' for constructing viscosity solutions $v$ to a
PDE as a limit $v_{\epsilon}$ of equations obtained by adding a regularizing
diffusion term, proportional to a small ``viscosity parameter''
$\epsilon$ (whose role here is played by $N_{k}^{-1}$). } Incidentally, as shown in \cite{glz}, the method of regularizing
psh envelopes $Pu$ by solutions $\varphi_{\beta}$ in the opposite
``zero-temperature limit'' $\beta\rightarrow\infty$ (Theorems \ref{thm:zero temp limit singular case intro},
\ref{thm:zero temp in sing case polar}) can be put to use in the
viscosity theory of \cite{glz}. For example, a key result in \cite{glz}
says that if $u_{1}$ and $u_{2}$ are two super solutions, then so
is the psh-envelope $P(\min\{u_{1},u_{2}\}).$ This is shown by generalizing
the convergence of $\varphi_{\beta}$ in Theorem \ref{thm:zero temp in sing case polar}
to a very irregular setting where $u$ is merely assumed to a bounded
Borel function and $dV$ a general measure not charging pluripolar
subsets. The method of regularizing envelopes by solutions of Monge-Ampère
type equations has also been extended and applied in various other
directions \cite{l-ng,d-r,to,c-z}. As observed in \cite{ru} the
equation \ref{eq:ma eq in thm temp deformed text} can also be obtained
from the Ricci flow via a backwards Euler discretization. Accordingly,
the corresponding continuity path is called the Ricci continuity path
in \cite{jmr}, where it (or rather its ``conical'' generalization)
plays a crucial role in the construction of Kähler-Einstein metrics
with edge/cone singularities, by deforming the trivial solution at
$\beta=\infty$ (for $\varphi_{\infty}$ a given Kähler potential)
to a Kähler-Einstein metric at $\beta=\pm1.$ 

\section{\label{sec:Complex-geometric-setup}Complex geometric setup}

In this section we provide some background from complex geometry and
pluripotential theory (see the books \cite{dembook,g-z3} for further
general background).

\subsection{\label{subsec:The-local-setting}The local setting }

Setting $z:=x+iy\in\C^{n}$ the space $\Omega^{1}(\C^{n})$ of all
complex one-forms on $\C^{n}$ decomposes as a sum
\begin{equation}
\Omega^{1}(\C^{n})=\Omega^{1,0}(\C^{n})+\Omega^{0,1}(\C^{n}),\label{eq:decomp of Omega one}
\end{equation}
of the two subspaces spanned by $\{dz_{i}\}$ and $\{d\bar{z}_{i}\},$
respectively. This induces a decomposition of the exterior algebra
of all complex differential forms $\Omega^{\cdot}(\C^{n})$ into forms
of bidegree $(p,q),$ where $p\leq n$ and $q\leq n.$ Accordingly,
the exterior derivative $d$ decomposes as $d=\partial+\bar{\partial},$
where 

\[
\partial\phi:=\sum_{i=1}^{n}\frac{\partial\phi}{\partial z_{i}}dz_{i},\,\,\,\frac{\partial}{\partial z_{i}}:=(\frac{\partial}{\partial x_{i}}-i\frac{\partial}{\partial y_{i}})/2,
\]
 and taking its complex conjugate defines the $(0,1)-$form $\bar{\partial}\phi.$
In particular, 
\begin{equation}
\omega^{\phi}:=\frac{i}{2\pi}\partial\bar{\partial}\phi=\frac{i}{2\pi}\sum_{i,j\leq n}\frac{\partial^{2}\phi}{\partial z_{i}\partial\bar{z}_{j}}dz_{i}\wedge d\bar{z}_{j},\label{eq:omega phi}
\end{equation}
 defines a real $(1,1)-$form (the normalization above turns out to
be convenient). Such a smooth form $\omega$ is said to be \emph{positive}
(Kähler), written as $\omega\geq0$ $(\omega>0),$ if the corresponding
Hermitian matrix is semi-positive (positive definite) at any point.

If $\phi$ is smooth, then $\phi$ is said to be \emph{plurisubharmonic
(psh)} if $\omega^{\phi}\geq0.$ A general function $\phi\in L_{loc}^{1}$
is said to be psh if it is strongly upper semi-continuous and $\omega^{\phi}\geq0$
holds in the weak sense of currents. More generally, given a real
smooth $(1,1)-$form $\omega$ a function $\varphi$ is said to be\emph{
$\omega-$psh }if 
\begin{equation}
\omega_{\varphi}:=\omega+\frac{i}{2\pi}\partial\bar{\partial}\varphi\geq0\label{eq:def of omega var phi}
\end{equation}
If $\omega:=\omega^{\phi_{0}}$ this means that $\phi$ is psh iff
$\varphi:=\phi-\phi_{0}$ is $\omega-$psh. 

For $\phi$ psh and smooth the Monge-Ampère measure is the measure
defined by
\[
MA(\phi):=(\omega^{\phi})^{n},
\]
 in terms of exterior products. The extension to locally bounded psh
$\phi$ was introduced by Bedford-Taylor. The extension may be characterized
by the property that it is continuous under decreasing limits of psh
locally bounded functions $\phi_{j}.$

\subsection{\label{subsec:Polarized-compact-manifolds}Polarized compact manifolds
$(X,L)$ and metrics on $L$}

The local complex analytic notions above naturally extend to the global
setting where $\C^{n}$ is replaced by a complex manifold (using that
the decomposition \ref{eq:decomp of Omega one} is invariant under
a holomorphic change of coordinates). However, if $X$ is compact,
then all psh functions $\phi$ on $X$ are constant (by the maximum
principle). Instead the role of a (say, smooth) psh function $\phi$
on $\C^{n}$ is played by a positively curved metric on a\emph{ line
bundle} $L\rightarrow X$ (using additive notations for metrics).
To briefly explain this standard complex geometric framework, first
recall that a line bundle $L$ over a complex manifold $X$ is, by
definition, a complex manifold (called the total space of $L)$ with
a surjective holomorphic map 
\[
\pi:\,L\rightarrow X
\]
such that the fibers $L_{x}$ of $\pi$ are one-dimensional complex
vector spaces and such that $\pi$ is locally trivial. In other words,
any point $x\in X$ admits a neighborhood $U$ and an isomorphism
(i.e. a $\C^{*}-$equivariant biholomorphic map) between $L_{|U}$
and $U\times\C,$ making the following diagram commutative:
\[
\begin{array}{ccc}
L & \longleftrightarrow & U\times\C\\
\searrow &  & \swarrow\\
 & U
\end{array}
\]
Fixing such an isomorphism holomorphic sections of $L\rightarrow U$
may be identified with holomorphic functions on $U.$ In particular,
the function $1$ over $U$ corresponds to a non-vanishing holomorphic
section $e_{U}$ of $L\rightarrow U.$ Now, a \emph{smooth (Hermitian)
metric} $\left\Vert \cdot\right\Vert $ on the line bundle $L$ is,
by definition, a smooth family of Hermitian metrics on the one-dimensional
complex subspaces $L_{x}.$ Given a covering $U_{i}$ of $X$ and
trivializations $e_{U_{i}}$ of $L\rightarrow U_{i}$ a metric $\left\Vert \cdot\right\Vert $
on $L$ may be represented by the following family of local functions
$\phi_{U_{i}}$ on $U_{i}:$ 
\[
\left\Vert e_{U_{i}}\right\Vert ^{2}=e^{-\phi_{U_{i}}}
\]
(accordingly a metric on $L$ is often, in additive notation, denoted
by the symbol $\phi,$ which is called a \emph{weight}). The functions
$\phi_{U_{i}}$ do not agree on overlaps, in general. But the (normalized)
\emph{curvature form} $\omega$ of the metric $\left\Vert \cdot\right\Vert $
is a well-defined closed two-form on $X,$ locally defined by 
\[
\omega_{|U_{i}}:=\omega^{\phi_{U_{i}}}
\]
\begin{example}
If $X$ is a complex manifold then the \emph{trivial line bundle}
$L_{0}\rightarrow X$ is $X\times\C$ with its natural projection
to $X.$ Thus the function $1$ on $X$ defines a global holomorphic
trivialization of $L_{0}.$ The trivial metric $\left\Vert \cdot\right\Vert $
on $L_{0}$ is defined by $\left\Vert 1\right\Vert :=1,$ i.e. the
corresponding weight function $\phi$ vanishes identically on $X.$
Accordingly, so does its curvature form.
\end{example}

\begin{defn}
A polarized manifold $(X,L)$ is a compact complex manifold $X$ endowed
with a positive line bundle, i.e. a line bundle $L$ which admits
some smooth metric with strictly positive curvature form. 

According to Kodaira's embedding theorem any polarized manifold $X$
is projective, i.e. it may be realized as a complex (algebraic) submanifold
of projective space (see Remark \ref{rem:Kodaira}).
\end{defn}

Singular metrics on $L$ may be defined in a similar way. In particular,
a singular metric is said to have \emph{positive curvature} if the
local weights $\phi_{U_{i}}$ are psh, i.e. the corresponding curvature
form $\omega$ defines a positive $(1,1)-$current on $X.$ The difference
between the weights of two different metrics is always a globally
well-defined function on $X$ (as a consequence, the curvature currents
of any two metrics on $L$ are cohomologous and represent the first
Chern class $c_{1}(L)\in H^{2}(X,\Z)).$
\begin{example}
\label{exa:The-Ricci-curvature}The (normalized) \emph{Ricci curvature
}of a Kähler metric $\omega$ on a complex manifold $X$ (identified
with a Riemannian metric) may be expressed as 
\begin{equation}
\mbox{Ric\,}\ensuremath{\omega=-\frac{i}{2\pi}\partial\bar{\partial}\log\left(\omega^{n}/d\lambda\right),}\label{eq:ricci curv}
\end{equation}
where $d\lambda$ denotes Lebesgue measure, expressed in terms of
local holomorphic coordinates $z_{1},...,z_{n}.$ Denoting by $K_{X}$
the canonical line bundle of $X,$ i.e. the top exterior power of
the holomorphic cotangent bundle $T^{*}X,$ this means that $-\mbox{Ric\,}\ensuremath{\omega}$
coincides with the curvature of the metric $\left\Vert \cdot\right\Vert $
on the line bundle $K_{X}$ induced by the volume form $\omega^{n}.$
Indeed, $e:=dz_{1}\wedge...\wedge dz_{n}$ is a local trivialization
of $K_{X}$ and $\left\Vert e\right\Vert ^{2}:=c_{n}\omega^{n}/e\wedge\bar{e},$
for a suitable constant $c_{n}$. By formula \ref{eq:ricci curv},
if $\psi_{\beta}$ is a smooth psh solution of the Monge-Ampère equation
\ref{eq:ma eq with beta intro}, then the corresponding Kähler metric
$\omega_{\beta}(:=\omega^{\psi_{\beta}})$ satisfies the twisted Kähler-Einstein
equation \ref{eq:twisted ke eq intro} with 
\[
\tau:=\beta\frac{i}{2\pi}\partial\bar{\partial}\phi-\frac{i}{2\pi}\partial\bar{\partial}\log\frac{dV}{d\lambda}
\]
(note that the last term vanishes if $dV$ is equal to Lebesgue measure
$d\lambda).$ 
\end{example}

\begin{rem}
Given a metric $\left\Vert \cdot\right\Vert $ on $L$ we will use
the same notation $\left\Vert \cdot\right\Vert $ for the induced
metric on the tensor products of $L$ over $X,$ obtained by imposing
that $\left\Vert \cdot\right\Vert $ be multiplicative. In particular,
if $\phi$ is a local weight for $\left\Vert \cdot\right\Vert $ (defined
wrt the local trivialization $s_{U})$ then $k\phi$ is a local weight
for the $k$ th tensor product of $L$ (defined wrt the local trivialization
$s_{U}^{\otimes k}).$ This motivates using the additive notation
$kL$ for tensor products. More generally, we will use the same notation
$\left\Vert \cdot\right\Vert $ for the induced metrics on the line
bundles $(kL)^{\boxtimes N}$ over the $N-$fold products $X^{N}.$ 
\end{rem}

\subsubsection{\label{subsec:Holomorphic-sections-of}Holomorphic sections of $L$}

The complex vector space of all global holomorphic sections of $L\rightarrow X$
will be denoted by $H^{0}(X,L).$ The line bundle $L$ is said to
be \emph{big} if there exists a positive number $V$ such that
\[
N_{k}:=H^{0}(X,kL)=\frac{V}{n!}k^{n}+o(k^{n}),
\]
 using additive notation $kL$ for tensor products of line bundles.
In particular, any positive line bundle $L$ is big (by the Hilbert-Samuel
theorem). The results in \cite{b-b-w,berm8 comma 5}  apply to general
big line bundles, but here we will only consider the simpler case
when $L$ is positive, i.e. the case of a polarized manifold $(X,L)$
( as in  \cite{berm6,berm8}).

There is natural correspondence between the space $H^{0}(X,L)$ and
the space of all 1-homogeneous holomorphic functions on the complex
manifold defined by the total space of the dual line bundle $L^{*}\rightarrow X,$
endowed with its standard $\C^{*}-$action. Indeed, if $s_{k}\in H^{0}(X,kL)$
and $w\in L^{*}.$ Then 
\[
S_{k}(w):=\left\langle s_{k},w^{\otimes k}\right\rangle ,
\]
 is a well-defined $k-$homogeneous function on $L^{*}$(since $s_{k}$
takes values in the $k$ th tensor power of $L)$$.$ 

Similarly, the space of metrics $\left\Vert \cdot\right\Vert $ on
$L$ may be identified with the space of all positively 1-homogeneous
non-negative functions $N$ on the total space of $L^{*}\rightarrow X.$
The curvature form of $\left\Vert \cdot\right\Vert $ is (strictly)
positive on $X$ iff $N^{2}$ is (strictly) psh on $L^{*}-\{0\}.$
Indeed, fixing local holomorphic coordinates $z$ on $X$ and a local
trivialization of $L$ induces local holomorphic coordinates $(z,w)$
on $L^{*}$ such that 
\[
N^{2}(z,w)=|w|^{2}e^{\phi(z)}
\]

\subsection{\label{subsec:Transcendental-polarized-manifol} $\omega-$psh functions
on Kähler manifolds}

Let $(X,L)$ be a polarized manifold. Fixing a reference metric $\left\Vert \cdot\right\Vert _{0}$
and setting 
\begin{equation}
\left\Vert \cdot\right\Vert _{\varphi}:=\left\Vert \cdot\right\Vert _{0}e^{-\varphi/2}\label{eq:change of metric def}
\end{equation}
 the space all metrics $\left\Vert \cdot\right\Vert $ on $L$ with
positive curvature current may be identified with the space $PSH(X,\omega)$
of all $\omega-$psh functions $\varphi$ on $X.$ The correspondence
is made so that $\omega_{\varphi}$ (formula \ref{eq:def of omega var phi})
is the normalized curvature of the corresponding metric $\left\Vert \cdot\right\Vert _{\varphi}$
on $L.$
\begin{example}
\label{exa:hol section becomes psh}If $s\in H^{0}(X,kL)$ and $\left\Vert \cdot\right\Vert $
is a metric on $L$ with curvature form $\omega,$ then the function
$k^{-1}\log\left\Vert s\right\Vert ^{2}$ is in $PSH(X,\omega),$
since $\log|s(z)|^{2}$ defines a singular metric on $L$ with positive
curvature current. Concretely, locally representing $s=s_{U}(z)e_{U}$
over $U\subset X,$ where $e_{U}$ is a trivializing holomorphic section
of $L\rightarrow U$ and $s_{U}(z)$ is a local holomorphic function,
we have that $\log|s_{U}(z)|^{2}$ is psh and hence $\log\left\Vert s\right\Vert ^{2}=\log\left(|s_{U}(z)|^{2}e^{-k\phi_{U}}\right)$
is $k\omega-$psh.

Given a real closed $(1,1)-$ form $\omega$ the space $PSH(X,\omega)$
can be defined without any reference to line bundles. In fact, $\omega$
arises as the normalized curvature form of a metric on a line bundle
over $X$ iff $\omega$ has integral periods (i.e. it defines a class
in the De Rham cohomology $H^{2}(X,\Z)).$ Accordingly, from an analytic
point of view it is more naturally to consider the more general setting
of a of a pair $(X,\omega)$ consisting of a compact complex manifold
$X$ and real closed $(1,1)-$ form $\omega,$ with continuous potentials
on $X,$ which is cohomologous to a Kähler form. In other words, 
\[
\omega=\omega_{0}+\frac{i}{2\pi}\partial\bar{\partial}u
\]
 for some Kähler form $\omega_{0}$ on $X$ and continuous function
$u$ on $X.$ 

We recall the following fundamental properties of the space $PSH(X,\omega)$
\cite{g-z}.
\end{example}

\begin{prop}
\label{prop:omega psh compact appr}Assume that $X$ is compact and
that $\omega$ is cohomologous to a Kähler form. 
\begin{itemize}
\item (Compactness) The subspace of $PSH(X,\omega)$ consisting of all functions
$\varphi$ normalized so that either $\int_{X}\varphi dV=0$ (for
a given volume form $dV$ on $X)$ or $\sup_{X}\varphi=0$ is compact
with respect to the $L^{1}-$topology.
\item (Approximation) Any element $\varphi\in PSH(X,\omega)$ can be written
as a decreasing limit of Kähler potentials $\varphi_{j},$ i.e. smooth
functions in $PHS(X,\omega)$ such that $\omega_{\varphi}$ is a Kähler
form (i.e. $\omega_{\varphi}>0).$ 
\end{itemize}
\begin{proof}
The compactness follows from embedding $PSH(X,\omega)$ into the compact
space $\mathcal{M}_{1}(X)$ (by mapping $\varphi$ to the Laplace
type measure $\omega_{\varphi}\wedge\tau^{n-1},$ for a fixed Kähler
form $\tau$ on $X).$ As for the approximation result it is a special
case of Demailly's general approximation results for positive currents.
Interestingly, the ``zero-temperature'' limit $\beta\rightarrow\infty$
studied in Section \ref{sec:Proofs-for-the limit beta} provides an
alternative ``PDE-proof'' of Demailly's approximation result for
$PSH(X,\omega).$ To see this first take a sequence $u_{j}\in C^{\infty}(X)$
decreasing to $\varphi.$ By monotonicity, $P_{\omega}(u_{j})$ decreases
$\varphi$ and hence it is enough to verify that $P_{\omega}(u)$
can be uniformly approximated by Kähler potentials when $u$ i smooth.
But, as shown in \cite{berm11}, $P_{\omega}(u)$ is the uniform limit
of the functions $\varphi_{\beta}$ solving the corresponding Monge-Ampère
equation \ref{eq:MA eq with beta in section proof beta} (for a fixed
volume form $dV)$ and $\varphi_{\beta}$ is smooth by \cite{au,y}.
As discussed in \cite{berm11} this regularization scheme can by viewed
as a transcendental analog of the well-known Bergman kernel approach
to regularization used in case of a polarized manifold $(X,L).$ 
\end{proof}
\end{prop}

\subsection{\label{subsec:The-non-pluripolar-Monge-Amp=0000E8re}The non-pluripolar
Monge-Ampère measure and the functional $\mathcal{E}(\varphi)$ }

Let $(X,L)$ be a polarized compact manifold and fix a smooth metric$\left\Vert \cdot\right\Vert $
on $L$ with curvature form $\omega.$ Then we can, as explained above,
identify (singular) positively curved metrics $\phi$ on $L$ with
$\omega-$psh functions $\varphi$ on $L.$ For a smooth $\omega-$psh
function the (normalized) Monge-Ampère measure of $\varphi$ (with
respect to $\omega)$ is defined as 
\begin{equation}
MA_{\omega}(\varphi):=\omega_{\varphi}^{n}/V\label{eq:def of ma meas text}
\end{equation}
 i.e. as the top exterior power of the corresponding curvature form
divided by the volume $V$ of the class $[\omega].$ More generally,
by the local pluripotential theory of Bedford-Taylor, the expression
in formula \ref{eq:def of ma meas text} makes locally sense for any
bounded $\omega-$psh function and defines a probability measure on
$X$ (which does not put charge on pluripolar subsets, i.e. sets which
are locally contained in the $-\infty-$locus of a psh function).
The corresponding Monge-Ampère measure $MA$ is continuous with respect
to decreasing sequences of bounded $\theta-$psh functions. In general,
following \cite{begz}, for any $\varphi\in PSH(X,\theta)$ we will
denote by $MA_{\omega}(\varphi)$ the \emph{non-pluripolar Monge-Ampère
measure,} which is a globally well-defined measure on $X$ not charging
pluripolar subsets. The measure $MA_{\omega}(\varphi)$ is defined
by replacing the ordinary wedge products with the so called non-pluripolar
products introduced in \cite{begz}) and satisfies 
\[
\int_{X}MA_{\omega}(\varphi)\leq1
\]

The Monge-Ampère measure $MA_{\omega},$ viewed as a one-form on the
space $PSH(X,\omega)\cap L^{\infty}$ is exact and we shall denote
by $\mathcal{E}$ the corresponding primitive. In other words, 

\begin{equation}
d\mathcal{E}_{|\varphi}=MA_{\omega}(\varphi),\label{eq:def of energyfunc as primitive}
\end{equation}
(in the sense that $d\mathcal{E}(\varphi+tu)/dt=\int_{X}MA(\varphi)u/V$
at $t=0).$ The functional $\mathcal{E}$ on $PSH(X,\omega)\cap L^{\infty}$
is only defined up to an additive constant that we shall fix by imposing
the normalization condition 
\begin{equation}
\mathcal{E}(\varphi_{(X,\omega)})=0,\label{eq:norm cond on energy}
\end{equation}
where $\varphi_{(X,\omega)}=0$ when $\omega\geq0$ and in general
it is defined as a global extremal function in the section below (formula
\ref{eq:def of global siciak for X}). Occasionally, we we will use
the notation $\mathcal{E}_{\omega}$ to indicate the dependence on
$\omega$ in the definition of the functional $\mathcal{E}.$ Integrating
the defining relation \ref{eq:def of energyfunc as primitive} along
a line segment in $PSH(X,\omega)\cap L^{\infty}$ reveals that 
\begin{equation}
\mathcal{E}(\varphi)=\frac{1}{(n+1)V}\int_{X}\varphi\sum_{j=0}^{n}\omega^{j}\wedge\omega_{\varphi}^{n-j}\label{eq:expl form for beu e}
\end{equation}
 in the case when $\omega\geq0$ (but the explicit formula for $\mathcal{E}$
will not really be used in the sequel). We will also denote by $\mathcal{E}$
the smallest upper semi-continuous extension of $\mathcal{E}$ to
all of $PSH(X,\omega),$ introduced in \cite{bbgz}, and write 
\[
\mathcal{E}^{1}(X,\omega):=\left\{ \varphi\in PSH(X,\omega):\,\,\mathcal{E}(\varphi)>-\infty\right\} ,
\]
 which is called the space of all functions on $X$ with \emph{finite
energy. }\footnote{We will adopt the notation from \cite{berm6}, which is different
from the one in\cite{bbgz}.}\emph{ }The Monge-Ampère measure of any $\varphi\in\mathcal{E}^{1}(X,\omega)$
defines a probability measure on $X.$

\subsection{\label{subsec:The-PSH-projection-}The PSH-projection $P_{\omega}$
and the global Siciak extremal function }

Next, we recall the definition of the operator $P_{\omega}$ introduced
in \cite{b-b}. Given $u\in C(X)$ we set 
\begin{equation}
(P_{\omega}u)(x):=\sup_{PSH(X,\omega)}\{\varphi(x):\,\,\,\varphi\leq u\,\,\text{on \ensuremath{X}}\}\label{eq:def of proj operator in khler case}
\end{equation}
which defines an operator 
\[
P_{\omega}:\,C(X)\rightarrow PSH(X,\omega_{0})
\]
from $C^{0}(X)$ to the space $PSH(X,\omega_{0})$ of all $\omega-$psh
functions on $X,$ i.e. all upper semi-continuous functions $\varphi$
in $L^{1}(X)$ such that $\omega_{\varphi}\geq0$ in the sense of
currents. 

Now the global Siciak extremal function of $(X,\omega)$ may be defined
by 
\begin{equation}
\varphi_{(X,\omega)}:=P_{\omega}(0).\label{eq:def of global siciak for X}
\end{equation}
 More generally, if $K$ is a locally non-pluripolar subset of $X,$
then the global Siciak extremal function of $(K,\omega)$ is defined
as the upper upper-semicontinuous regularization of $P_{(K,\omega)}(0),$
where $P_{(K,\omega)}(u)$ is defined as in formula \ref{eq:def of proj operator in khler case},
but only demanding that $\varphi\leq u$ on $K.$ The assumption on
$K$ ensures that $P_{(K,\omega)}(u)$ is a well-defined finite function
and hence $\omega-$psh (see \cite{g-z} where the global Kähler analogs
of Siciak's extremal function in $\C^{n}$ were first introduced).
\begin{rem}
Since we have assumed that $(X,L)$ is polarized (or more generally,
that $\omega$ is cohomologous to a Kähler form) the operator $P_{\omega}$
preserves $C(X)$ (as follows from Prop \ref{prop:omega psh compact appr}).
Hence, it defines a (non-linear) projection operator from $C^{0}(X)$
onto $PSH(X,\omega)\cap C^{0}(X).$ Anyway, for the present purposes
it is enough to observe that $P_{\omega}$ preserves $L^{\infty},$
as follows directly from the definition. 

The following important ``orthogonality relation'' holds for the
projection operator $P_{\omega}:$ 
\begin{equation}
\int_{X}(u-P_{\omega}u)MA_{\omega}(P_{\omega}u)=0\label{eq:og relation}
\end{equation}
 Indeed, since $P_{\omega}u\leq u,$ this is equivalent to saying
that $P_{\omega}u=u$ almost everywhere wrt the measure $MA_{\omega}(P_{\omega}u),$
which in turn follows from a ``balayage''/maximum principle type
argument (see \cite{b-b}).
\end{rem}

\subsection{\label{subsec:The-pluricomplex-energy}The pluricomplex energy $E_{\omega}(\mu)$
on polarized manifolds}

Following \cite{bbgz} the\emph{ pluricomplex energy} $E_{\omega}(\mu)$
of a probability measure $\mu$ is defined by
\begin{equation}
E_{\omega}(\mu):=\sup_{\varphi\in\mathcal{E}^{1}(X,\omega)}\mathcal{E}(\varphi)-\left\langle \varphi,\mu\right\rangle ,\label{eq:def of e as sup}
\end{equation}
which is automatically lsc and convex on $\mathcal{M}_{1}(X)$ (by
the approximation property in Prop \ref{prop:omega psh compact appr}
the sup may equivalently be taken over $PSH(X,\omega)\cap C(X)$ or
$PSH(X,\omega)\cap C^{\infty}(X)).$ 

As recalled in the following theorem the sup defining $E_{\omega}$
is in fact attained, if $E_{\omega}(\mu)<\infty$: 
\begin{equation}
E_{\omega}(\mu):=\mathcal{E}(\varphi_{\mu})-\left\langle \varphi_{\mu},\mu\right\rangle \label{eq:energy in terms of pot}
\end{equation}
for a unique function $\varphi_{\mu}\in\mathcal{E}^{1}(X,\theta)/\R.$
Moreover 
\begin{equation}
MA(\varphi_{\mu})=\mu.\label{eq:potential of meas}
\end{equation}
\begin{thm}
\label{thm:var sol of ma}\cite{bbgz} The following is equivalent
for a probability measure $\mu$ on $X:$ 

\begin{itemize}
\item $E_{\omega}(\mu)<\infty$
\item \textup{$\left\langle \varphi,\mu\right\rangle <\infty$ for all $\varphi\in\mathcal{E}^{1}(X,\omega)$}
\item $\mu$ has a potential $\varphi_{\mu}\in\mathcal{E}^{1}(X,\omega$),
i.e. equation \ref{eq:potential of meas} holds
\end{itemize}
Moreover, $\varphi_{\mu}$ is uniquely determined mod $\R,$ i.e.
up to an additive constant and can be characterized as the function
maximizing the functional whose sup defines $E_{\omega}(\mu)$ (formula
\ref{eq:def of e as sup}). 
\end{thm}

In particular, if $\mu$ charges a pluripolar subset then $E_{\omega}(\mu)=\infty.$
The previous theorem implies the following result \cite[Prop 2.7]{berm6}:
\begin{prop}
\label{prop:sub-gradient}Assume that $\mu\in\mathcal{M}_{1}(X)$
has finite energy, i.e. $E_{\omega}(\mu)<\infty.$ Then $-\varphi_{\mu}$
is a sub-gradient for the functional $E_{\omega}$ at $\mu,$ i.e.
\[
E_{\omega}(\nu)\geq E_{\omega}(\mu)-\left\langle \nu-\mu,\varphi_{\mu}\right\rangle ,
\]
 for any $\nu\in\mathcal{M}_{1}(X).$ As a consequence, if $\mu_{t}$
is an affine curve in $\mathcal{M}_{1}(X),$ defined for $t\in]-\epsilon,\epsilon[$
and such that $\mu_{0}=\mu,$ then 
\[
\frac{dE_{\omega}(\mu_{t})}{dt}_{|t=0}=-\left\langle \mu_{1}-\mu,\varphi_{\mu}\right\rangle 
\]
\end{prop}

\begin{rem}
\label{rem:E on volume forms}By the previous proposition the differential
of the restriction of the functional $E_{\omega}(\mu)$ to the subspace
$\mathcal{P}^{\infty}(X)$ of volume forms in $\mathcal{M}_{1}(X)$
may be represented by $\varphi_{\mu}:$ 
\[
dE_{\omega}(\mu)=-\varphi_{\mu}
\]
(in this special case the formula follows by standard Legendre duality,
since formula \ref{eq:energy in terms of pot} expresses $E(\mu)$
as the Legendre transform of the functional$-\mathcal{E}(-u)).$ One
can then recover $E_{\omega}$ on all of $\mathcal{M}_{1}(X)$ as
the greatest lower semi-continuous from $\mathcal{P}^{\infty}(X).$
When $\omega\geq0$ the following explicit formula for $E_{\omega}$
on $\mathcal{P}^{\infty}(X)$ follows from formula \ref{eq:expl form for beu e}
(and integration by parts):
\begin{equation}
E_{\omega}(\mu)=\frac{1}{V}\sum_{j=0}^{n-1}\frac{1}{j+2}\int d\varphi_{\mu}\wedge d^{c}\varphi_{\mu}\wedge\frac{\omega_{\varphi_{\mu}}^{j}}{j!}\wedge\frac{\omega^{n-1-j}}{(n-1-j)!},\label{eq:energy in terms of potential}
\end{equation}
 Moreover, $E_{\omega}(\mu)$ coincides with $(I_{\omega}-J_{\omega})(\varphi_{\mu}),$
in terms of Aubin's functionals $I$ and $J$ \cite{au2}.
\end{rem}

\subsection{\label{subsec:Compactification-of-}Compactification of $\C^{n}$
by $\P^{n}$}

We recall that the $n-$dimensional \emph{complex projective space
\[
\P^{n}:=\C^{n+1}-\{0\}/\C^{*})
\]
}comes with a tautological line bundle whose total space, with the
zero-section deleted, is $\C^{n+1}-\{0\}:$ the line over a point
$[Z_{0}:....:Z_{m}]\in\P^{n}$ is simply the line $\C(Z_{0},..,Z_{n}).$
The dual of the tautological line bundle is called the \emph{hyperplane
line bundle }and is usually denoted by $\mathcal{O}(1).$ The notation
reflects the fact that the space $H^{0}(\P^{n},\mathcal{O}(1))$ of
holomorphic sections of $\mathcal{O}(1)\rightarrow\P^{n}$ may be
identified with the space of $1-$homogeneous holomorphic functions
$F$ on $\C^{n+1}$ (by the correspondence in Section \ref{subsec:Holomorphic-sections-of}).
In particular, the homogeneous coordinates $Z_{i}$ on $\P^{n}$ define
holomorphic sections of $\mathcal{O}(1)\rightarrow\P^{n}.$ Similarly,
the metrics $\left\Vert \cdot\right\Vert $ on $\mathcal{O}(1)$ may
be identified with positively $1-$homogeneous non-negative functions
$N$ on $\C^{n+1}$ (assumed strictly positive on $\C^{n+1}-\{0\}).$
Accordingly,
\[
\left\Vert F\right\Vert =\frac{|F|}{N}
\]
descends from $\C^{n+1}$ to a well-defined function on $\P^{n+1}.$
In particular, the norm $N$ defined by the Euclidean norm on $\C^{n+1}$
induces a metric on $\mathcal{O}(1),$ called the \emph{Fubini-Study
metric, }that we shall denote by $\left\Vert \cdot\right\Vert _{FS}.$
Its curvature form $\omega_{FS}$ defines a Kähler form on $\P^{n}$
(as follows from the fact that the squared Euclidean norm $N^{2}$
defines a smooth and strictly plurisubharmonic function on $\C^{n\text{+1 }}).$
Thus $(\P^{n},\mathcal{O}(1))$ is a polarized manifold.

Following standard practice we will identify $\C^{n}$ with the open
subset 
\[
U_{0}:=\P^{n}-H_{0},\,\,\,H_{0}:=\{Z_{0}=0\},
\]
 i.e. with the image of the embedding 
\[
\C^{n}\rightarrow\P^{n},\,\,\,z\mapsto[1:z]
\]

Over $U_{0}$ the line bundle $\mathcal{O}(1)$ is trivialized by
$Z_{0}$ (since it is non-vanishing precisely on $U_{0}).$ Accordingly,
over $U_{0}$ the Fubini-Study metric $\left\Vert \cdot\right\Vert _{FS}$
is represented by the weight 
\[
\phi_{FS}(z):=-\log\left\Vert Z_{0}\right\Vert _{FS}^{2}:=-\log\frac{|Z_{0}|^{2}}{|Z_{0}|^{2}+...+|Z_{n}|^{2}}=\log(1+|z|^{2})
\]
 defining a smooth metric with strictly positive curvature form $\omega_{FS}.$ 

The space $H^{0}(\P^{n},k\mathcal{O}(1))$ may be naturally identified
with the space of all homogeneous polynomials $P_{k}$ of degree $k$
in $\C^{n+1}$ (by the correspondence in Section \ref{subsec:Holomorphic-sections-of})
and
\[
\left\Vert P_{k}\right\Vert ^{2}([Z])=\frac{|P_{k}(Z)|}{N(Z)^{k}}
\]

Hence, ``homogenization'' establishes a one-to-one correspondence
\[
\mathcal{P}_{k}(\C^{n})\longleftrightarrow(H^{0}(\P^{n},k\mathcal{O}(1)),\,\,\,p_{k}(z_{1},...z_{n})=P_{k}(1,z_{1},...,z_{n})
\]
 between the space $\mathcal{P}_{k}(\C^{n})$ of all polynomials $p_{k}$
on $\C^{k}$ of degree at most $k$ and $H^{0}(\P^{n},k\mathcal{O}(1)).$
In particular, over $U_{0}(=\C^{n})$ we can write 
\begin{equation}
P_{k}=p_{k}Z_{0}^{\otimes k}.\label{eq:P k in terms of p k}
\end{equation}
\begin{lem}
\label{lem:correspond funct metric in Cn}Setting 
\[
\phi(z):=-\log\left\Vert Z_{0}\right\Vert ^{2}
\]

yields a correspondence between 
\begin{itemize}
\item metrics $\left\Vert \cdot\right\Vert $ on $\mathcal{O}(1)\rightarrow\P^{n}$
with positive curvature current and functions $\phi(z)$ in the Lelong
class $\mathcal{L}(\C^{n})$ (formula \ref{eq:def of L intro})
\item usc metrics $\left\Vert \cdot\right\Vert $ on $\mathcal{O}(1)\rightarrow\P^{n}$
(i.e. the local weights are lsc) and lsc functions $\phi(z)$ on $\C^{n}$
with super logarithmic growth, i.e. $\phi(z)\geq\log(1+|z|^{2})+O(1).$
\end{itemize}
\end{lem}

\begin{proof}
Let us start with the second point. Consider a lsc function $\phi(z)$
on $\C^{n}$ satisfying the growth property in the lemma. This means
that $u(z):=\phi(z)-\phi_{FS}(z)$ is a lsc function on $U:=\P^{n}-H_{0}$
which is uniformly bounded from below. Now, any lsc function $u$
on a subset $U$ in a topological space admits, if $u$ is bounded
from below on $U,$ a canonical lsc extension $\bar{u}$ to the closure
$\bar{U},$ namely the greatest lsc extension of $u.$ The corresponding
metric on $\mathcal{O}(1)\rightarrow\P^{n}$ is then defined as $\left\Vert \cdot\right\Vert _{u}:=\left\Vert \cdot\right\Vert _{FS}e^{-\bar{u}/2},$
which proves the first point. The proof of the first point is similar,
using that a function $u$ is $\omega-$psh and bounded from above
on $X-A,$ where $A$ is a pluripolar subset, iff the smallest usc
of $u$ to $X$ is $\omega-$psh. This follows directly from the corresponding
classical local property of psh functions.
\end{proof}
The correspondence above is made so that 
\[
|p_{k}(z)|^{2}e^{-k\phi}=\left\Vert P_{k}(1,z)\right\Vert ^{2},
\]
as follows directly from formula \ref{eq:P k in terms of p k}.
\begin{rem}
\label{rem:Kodaira}If $X$ is a complex submanifold of $\P^{m}$
(which, by Chow's theorem, equivalently means that $X$ is a projective
non-singular algebraic variety), then $(X,\mathcal{O}_{X}(1))$ is
a polarized manifold, where $\mathcal{O}(1)_{X}$ denotes the restriction
of $\mathcal{O}(1)\rightarrow\P^{m}$ to $X.$ Indeed, the restriction
to $\mathcal{O}_{X}(1)$ of the Fubini-Study metric on $\mathcal{O}(1)$
has strictly positive curvature. Conversely, by the Kodaira embedding
theorem, if $(X,L)$ is a polarized manifold, then after perhaps replacing
$L$ by a large tensor power, $X$ may be holomorphically embedded
in a projective space $\P^{m}$ in such a way that $L$ gets identified
with $\mathcal{O}(1)_{X}$ and $H^{0}(X,kL)$ identifies with the
restriction to $X$ of the space of all $k-$homogeneous polynomials
over $\P^{m}.$ This means that a line bundle is positive iff it is
\emph{ample,} in the algebro-geometric sense.
\end{rem}

\subsubsection{\label{subsec:The-(non-)weighted-pluricomplex}The (non-)weighted
pluricomplex energy $E$ in $\C^{n}$}

Let $\mu$ be a probability measure on $\C^{n}$ such that 
\begin{equation}
\int_{\C^{n}}\log(1+|z|^{2})\mu<\infty\label{eq:integrability condition}
\end{equation}
We define the\emph{ non-weighted pluricomplex energy} $E(\mu)$ by
\begin{equation}
E(\mu)=\sup_{\psi\in\mathcal{E}^{1}(\C^{n})}\left(\mathcal{E}(\psi)-\mathcal{E}(\psi_{T^{n}})-\int_{\C^{n}}\psi\mu\right),\label{eq:def of E mu in text}
\end{equation}
 where we have identified $\mathcal{E}$ with a functional on the
Lelong class $\mathcal{L}(\C^{n}),$ using the correspondence in Lemma
\ref{lem:correspond funct metric in Cn}. We note that the extremal
function $\psi_{T^{n}}$ is explicitly given by
\[
\psi_{T^{n}}=\log\max_{i=1,...,n}\left\{ 1,|z_{i}|^{2}\right\} .
\]
Alternatively, we could directly have defined $\mathcal{E}$ (up to
an additive constant) on $\mathcal{L}(\C^{n})$ by requiring that
\[
d\mathcal{E}_{|\psi}=(dd^{c}\psi)^{n}
\]
when $\psi$ is psh, locally bounded and with logarithmic growth (note
that the definition \ref{eq:def of E mu in text} is independent of
the choice of additive constant for $\mathcal{E}$).

Assume now that $\mu\in\mathcal{M}_{1}(\C^{n})$ has compact support.
Embedd $\C^{n}$ in $\P^{n}$ as above and fix an auxiliary smooth
metric $\left\Vert \cdot\right\Vert $ on $\mathcal{O}(1)\rightarrow\P^{n}$
coinciding with the trivial metric on a bounded neighborhood $U$
of the support of $\mu.$ In particular, the curvature form $\omega$
vanishes on $U.$ 
\begin{lem}
\label{lem:non-weighted energ}Assume that $\mu\in\mathcal{M}_{1}(\C^{n})$
has compact support. The non-weighted pluricomplex energy $E(\mu)$
may be expressed as 
\[
E(\mu):=E_{\omega}(\mu)-\mathcal{E}_{\omega}(P_{T^{n}}0)
\]
In particular, if $K$ is a subset of the unit-polydisc, then $\omega$
can be taken as $\omega_{\psi_{T^{n}}}$ and then 
\[
E(\mu):=E_{\omega_{T^{n}}}(\mu),
\]
Moreover, $E(\mu)$ coincides with the classical energy \ref{eq:def of log energy of mu}
when $n=1.$ 
\end{lem}

\begin{proof}
The first formula follows directly from making the identifications
in Section \ref{subsec:Compactification-of-} and by rewriting 
\[
\int_{\C^{n}}\psi\mu=\int_{\C^{n}}\left(\psi-\psi_{T^{n}}\right)\mu+\int_{\C^{n}}\psi_{T^{n}}\mu,
\]
 where the last term is finite, by assumption. Moreover, when $n=1$
we can express the classical logarithmic energy as
\[
-\frac{1}{2}\int_{\C}\log|z-w|^{2}\mu(w)\otimes\mu(w)=-\frac{1}{2}\int_{\C}\psi_{\mu}\mu,\,\,\,\psi_{\mu}(z)=\int_{\C}\left(\log|z-w|^{2}-\phi(z)-\phi(w)\right)\mu
\]
 where $dd^{c}\psi_{\mu}+\omega=\mu,$  using that $\omega=dd^{c}\phi$
with $\phi=0$ close to the support of $\mu.$ Hence, the logarithmic
energy coincides with $E_{\omega}(\mu)$ up to an overall additive
constant. This means that it also coincides with the non-weighted
pluricomplex energy $E(\mu)$ up to an overall constant $C.$ Finally,
the constant $C$ can be seen to be equal to $0$ by taking $\mu=d\theta,$
the uniform measure on $S^{1}.$ 
\end{proof}
Given a continuous function $\phi$ on $\C^{n}$ with super logarithmic
growth (formula \ref{prop:omega psh compact appr}), we consider the
corresponding \emph{weighted pluricomplex energy }$E_{\phi}$ on $\mathcal{M}_{1}(\C^{n}),$ 

\[
E_{\phi}(\mu)=E_{\omega_{T^{n}}}(\mu)+\int(\phi-\psi_{T^{n}})\mu
\]
This is well-defined since both terms above take values in $]-\infty,\infty].$
In particular, if $\mu$ satisfies the intractability property \ref{eq:integrability condition},
then we can decompose
\[
E_{\phi}(\mu):=E(\mu)+\int\phi\mu.
\]
\begin{rem}
\label{rem:weighted energy with lsc phi}It is enough to assume that
$\phi$ is lsc (and has super-logarithmic growth) in order to define
$E_{\phi}(\mu),$ as above. But in order to ensure that $E_{\phi}$
is not identically equal to infinity one then has to assume that $\{\phi<\infty\}$
is an open non-pluripolar subset.
\end{rem}

\section{\label{sec:Probabilistic-setup-and}Probabilistic setup and mean
field approximations}

Given a Hausdorff locally compact topological space $\Omega$ we will
denote by $\mathcal{M}(\Omega)$ the space of all signed (Borel) measures
on $\Omega$ and by $\mathcal{M}_{1}(\Omega)$ the subspace of all
measures with unit total mass, i.e. the space of probability measures
on $\Omega.$ 

\subsection{\label{subsec:Gamma-convergence-and-Legendre-F}Gamma-convergence
and Legendre-Fenchel transforms}

We start with some functional analytic preparations in the context
of Gamma-convergence, as introduced by De Georgi (see the book \cite{bra}
for background on Gamma-convergence).
\begin{defn}
A sequence of functions $E_{N}$ on a topological space \emph{$\mathcal{M}$
is said to $\Gamma-$converge }to a function $E$ on $\mathcal{M}$
if 
\begin{equation}
\begin{array}{ccc}
\mu_{N}\rightarrow\mu\,\mbox{in\,}\mathcal{M} & \implies & \liminf_{N\rightarrow\infty}E_{N}(\mu_{N})\geq E(\mu)\\
\forall\mu & \exists\mu_{N}\rightarrow\mu\,\mbox{in\,}\mathcal{M}: & \lim_{N\rightarrow\infty}E_{N}(\mu_{N})=E(\mu)
\end{array}\label{eq:def of gamma conv}
\end{equation}
Given $\mu$ a sequence $\mu_{N}$ as in the last point above is called
a\emph{ recovery sequence.} The limiting functional $E$ is automatically
lower semi-continuous on $\mathcal{M}.$ 

In the present setting we will take $\mathcal{M}=\mathcal{M}(X)$
for a compact manifold $X$ and define $E_{N}$ by setting $E_{N}=\infty$
on the complement of the image of the map $\delta_{N}$ and 
\begin{equation}
E_{N}(\delta_{N}(x_{1},...,x_{N})):=H^{(N)}(x_{1},...,x_{N})/N\label{eq:energi on M_1 induced by Hamilt}
\end{equation}
\end{defn}

\begin{rem}
Note that in the previous example it is not the case that $\mu_{N}:=\delta_{N}(x_{1},...,x_{N})\rightarrow\mu$
implies that $\limsup_{N}E_{N}(\mu_{N})\leq E(\mu).$ Indeed, for
any sequence where two points, say $x_{1}$ and $x_{2},$ coincide,
$E_{N}(\mu_{N})=\infty!$ This phenomenon explains the asymmetry between
the first and the second condition in the definition of Gamma-convergence.
\end{rem}

A criterion for Gamma-convergence on $\mathcal{M}_{1}(X)$ can be
obtained using duality in topological vector spaces, as next explained.
Let $f$ be a function on a topological vector space $V.$ The \emph{Legendre-Fenchel
transform} $f^{*}$ of $f$ is defined as following convex lower semi-continuous
function $f^{*}$ on the topological dual $V^{*}$ 
\[
f^{*}(w):=\sup_{v\in V}\left\langle v,w\right\rangle -f(v)
\]
in terms of the canonical pairing between $V$ and $V^{*}.$ In the
present setting we will take $V=C^{0}(X)$ and $V^{*}=\mathcal{M}(X),$
the space of all signed Borel measures on a compact topological space
$X.$ 
\begin{prop}
\label{prop:crit for gamma conv}Let $E_{N}$ be a sequence of functions
on the space $\mathcal{M}_{1}(X)$ of probability measures on a compact
space $X$ and assume that
\[
\lim_{N\rightarrow\infty}E_{N}^{*}(w)=f(w)
\]
for any $w\in C(X)$ and that $f$ defines a Gateaux differentiable
function on $C(X).$ Then $E_{N}$ converges to $E:=f^{*}$ in the
sense of $\Gamma-$convergence on the space $\mathcal{M}_{1}(X),$
equipped with the weak topology. 
\end{prop}

See \cite{berm8} for the proof (the existence of a recovery sequence
is based on an application of a general approximation result of Brøndsted-Rockafellar,
which applies in certain topological vector spaces).
\begin{rem}
In the present setting it will convenient to rather look at the following
transform
\[
-E_{N}^{*}(-u)=\inf_{X^{N}}N^{-1}\left(H^{(N)}(x_{1},...,x_{N})+\sum_{i=1}^{N}u(x_{i})\right)
\]
\end{rem}

\subsection{\label{subsec:Probabilistic-setup}Probabilistic setup}

We start by recalling some basic notions of probability theory (covered
by any standard textbook; see in particular \cite{d-z} for an introduction
to large deviation techniques). A \emph{probability space }is a space
$\Omega$ equipped with a probability measure $p.$ The space $\Omega$
is called the\emph{ sample space} and a measurable subset $\mathcal{B}\subset\Omega$
is called an \emph{event} with 
\[
\mbox{Prob}\mathcal{B}:=p(\mathcal{B}),
\]
interpreted as the probability of observing the event $\mathcal{B}$
when sampling from $(\mathcal{X},\Omega).$ A measurable function
$Y:\,\Omega\rightarrow\mathcal{Y}$ on a probability space $(\Omega,p)$
is called a\emph{ random element with values in $\mathcal{Y}$ }and
its\emph{ law} $\Gamma$ is the probability measure on $\mathcal{Y}$
defined by the push-forward measure
\[
\Gamma:=Y_{*}p
\]
 (the law of $Y$ is often also called the distribution of $Y$).
A sequence of random elements $Y_{N}:\,\Omega_{N}\rightarrow\mathcal{Y}$
taking values in the same topological space $\mathcal{Y}$ are said
to\emph{ convergence in law towards a deterministic element $y$ in
$\mathcal{Y}$ }if the corresponding laws $\Gamma_{N}$ on $\mathcal{Y}$
converge  to a Dirac mass at $y:$ 
\[
\lim_{N\rightarrow\infty}\Gamma_{N}=\delta_{y}
\]
 in the weak topology. In the present setting $\mathcal{Y}$ will
always be a separable metric space with metric $d$ and then $Y_{N}$
converge in law towards the deterministic element $y$ iff $Y_{N}$
\emph{converge in probability} towards $y,$ i.e. for any fixed $\epsilon>0$
\[
\lim_{N\rightarrow\infty}p\{d(Y_{N},y)>\epsilon\}=0.
\]
\begin{rem}
The\emph{ expectation} of a random variable $Y$ it defined by 
\[
\E(Y):=\int_{\Omega}Yp
\]
 (aka the\emph{ sample mean} of $Y)$ which defines an element in
$\mathcal{Y}.$ The statement that $Y_{N}$ converges in law towards
a deterministic element $y$ equivalently means that $\E(Y_{N})\rightarrow y$
and that $Y_{N}$ satisfies the (weak) law of large numbers, i.e.
the probability that $Y_{N}$ deviates from its mean tends to zero,
as $N\rightarrow\infty.$
\end{rem}

A\emph{ random point process} \emph{with $N$ particles }on a space
$X$ is, by definition, a probability measure $\mu^{(N)}$ on the
$N-$fold product $X^{N}$ (the $N-$particle space) which is symmetric,
i.e. invariant under action of the group $S_{N}$ by permutations
of the factors of $X^{N}.$The\emph{ empirical measure} of a given
random point process is the following random measure 
\begin{equation}
\delta_{N}:\,\,X^{N}\rightarrow\mathcal{M}_{1}(X),\,\,\,(x_{1},\ldots,x_{N})\mapsto\delta_{N}(x_{1},\ldots,x_{N}):=\frac{1}{N}\sum_{i=1}^{N}\delta_{x_{i}}\label{eq:empirical measure text}
\end{equation}
on $(X^{N},\mu^{(N)}).$ The law of $\delta_{N}$ thus defines a probability
measure on the space$\mathcal{M}_{1}(X),$ that we shall denote by
$\Gamma_{N}:$ 
\begin{equation}
\Gamma_{N}:=(\delta_{N})_{*}\mu^{(N)}\label{eq:def of Gamma N}
\end{equation}
 
\begin{rem}
\label{rem:The-point-correlation}\emph{The $j-$point correlation
measure $(\mu^{(N)})_{j}$ }of the $N-$particle random point process
is the probability measure on $X^{j}$ defined as the push-forward
to $X^{j}$ of $\mu^{(N)}$ under projection $X^{N}\rightarrow X^{j},$
where $(x_{1},...,x_{N})\mapsto(x_{i_{1}},...,x_{i_{j}})$ for any
choice of $j$ different indices $i_{1},...,i_{j}.$ In particular,
\begin{equation}
\E(\delta_{N})=(\mu^{(N)})_{1}\label{eq:expect as one-point correl measure}
\end{equation}

In the statistical mechanical setup recalled in the beginning of Section
\ref{subsec:Statistical-mechanics-and} a random point process on
with $N$ points on $X$ is induced by the data $(H^{(N)},dV,\beta).$
We denote by 
\[
\mu_{\beta}^{(N)}=\frac{1}{Z_{N,\beta}}e^{-\beta H^{(N)}}dV^{\otimes N}
\]
 the corresponding Gibbs measure, which defines a symmetric probability
measure on $X^{N}.$ 
\end{rem}

\subsubsection{\label{subsec:The-notion-of}Large Deviation Principles (LDP)}

The notion of a \emph{Large Deviation Principle (LDP)}, introduced
by Varadhan, allows one to give a notion of exponential convergence
in probability. The general definition of a Large Deviation Principle
(LDP) for a general sequence of measures \cite{d-z} is modeled on
the classical Laplace principle of ``saddle approximation'':
\begin{defn}
\label{def:large dev}Let $\mathcal{Y}$ be a Polish space, i.e. a
complete separable metric space.

$(i)$ A function $I:\mathcal{\,Y}\rightarrow]-\infty,\infty]$ is
a \emph{rate function} if it is lower semi-continuous. It is a \emph{good}
\emph{rate function} if it is also proper.

$(ii)$ A sequence $\Gamma_{k}$ of measures on $\mathcal{Y}$ satisfies
a \emph{large deviation principle} with \emph{speed} $r_{k}$ and
\emph{rate function} $I$ if

\[
\limsup_{k\rightarrow\infty}\frac{1}{r_{k}}\log\Gamma_{k}(\mathcal{F})\leq-\inf_{\mu\in\mathcal{F}}I
\]
 for any closed subset $\mathcal{F}$ of $\mathcal{Y}$ and 
\[
\liminf_{k\rightarrow\infty}\frac{1}{r_{k}}\log\Gamma_{k}(\mathcal{G})\geq-\inf_{\mu\in G}I(\mu)
\]
 for any open subset $\mathcal{G}$ of $\mathcal{Y}.$ 
\end{defn}

In the present setting $\Gamma_{N}$ will arise as the sequence of
probability measures on $\mathcal{M}_{1}(X)$ defined as laws of the
empirical measures $\delta_{N}$ (formula \ref{eq:empirical measure text}).
Once the LDP has been established we can apply the following basic
\begin{lem}
Assume that the laws $\Gamma_{N}$ above satisfy a LDP with a good
rate functional $I_{\beta}$ which admits a\emph{ unique }minimizer
$\mu_{\beta}.$ Then the random measures $\delta_{N}$ converge in
law towards the deterministic measure $\mu_{\beta}.$ More precisely,
\[
\mbox{Prob}\{d(\delta_{N},\mu_{\beta})\geq\epsilon\}\leq C_{\epsilon}e^{-N/C_{\epsilon}}
\]
\end{lem}

\begin{proof}
First applying the LDP to $\mathcal{F}=\mathcal{G}=\mathcal{Y}$ gives
$I_{\beta}(\mu_{\beta})=0.$ Since $\mu_{\beta}$ is the unique minimizer
of $I_{\beta}$ it follows that $\inf I_{\beta}>0$ on the closed
subset $\mathcal{F}_{\epsilon}$ of $\mathcal{Y}$ where $d(\cdot,\mu_{\beta})\geq\epsilon.$
Applying $(i)$ in the LDP to $\mathcal{F}_{\epsilon}$ thus concludes
the proof of the desired. deviation inequality.
\end{proof}
In other words, the lemma says that risk that $\delta_{N}$ deviates
from $\mu_{\beta}$ is exponentially small. In order to establish
the LDP we will have great use for the following alternative formulation
of a LDP (see Theorems 4.1.11 and 4.1.18 in \cite{d-z}):
\begin{prop}
\label{prop:d-z}Let $\mathcal{Y}$ be a compact  metric space and
denote by $B_{\epsilon}(y)$ the ball of radius $\epsilon$ centered
at $y\in\mathcal{Y}.$ Then a sequence $\Gamma_{N}$ of probability
measures on $\mathcal{P}$ satisfies a LDP with speed $r_{N}$ and
a rate functional $I$ iff 
\begin{equation}
\lim_{\epsilon\rightarrow0}\liminf_{N\rightarrow\infty}\frac{1}{r_{N}}\log\Gamma_{N}(B_{\epsilon}(y))=-I(y)=\lim_{\epsilon\rightarrow0}\limsup_{N\rightarrow\infty}\frac{1}{r_{N}}\log\Gamma_{N}(B_{\epsilon}(y))\label{eq:ldp in terms of balls in prop}
\end{equation}
\end{prop}

In the present setting of a sequence of random point process with
$N$ particles the previous proposition may be symbolically summarized
as follows: 
\[
\text{Prob }\left(\frac{1}{N}\sum_{i=1}^{N}\delta_{x_{i}}\in B_{\epsilon}(\mu)\right)\sim e^{-r_{N}I(\mu)}
\]
when first $N\rightarrow\infty$ and then $\epsilon\rightarrow0.$

We will also use the following classical result of Sanov, which is
the standard example of an LDP for random point processes (describing
the case when the particles $x_{1},...,x_{N}$ define independent
variables with identical distribution $\mu_{0}$):
\begin{prop}
(Sanov) \label{prop:sanov}Let $X$ be a topological space and $\mu_{0}$
a finite measure on $X.$ Then the laws $\Gamma_{N}$ of the empirical
measures $\delta_{N}$ defined with respect to the product measure
$\mu_{0}^{\otimes N}$ on $X^{N}$ satisfy an LDP with speed $N$
and rate functional the relative entropy $D_{\mu_{0}}.$ 
\end{prop}

We recall that the \emph{relative entropy} $D_{\mu_{0}}$ (also called
the \emph{Kullback\textendash Leibler divergence }or the\emph{ information
divergence} in probability and information theory) is the functional
on $\mathcal{M}_{1}(X)$ defined by 
\begin{equation}
D_{\mu_{0}}(\mu):=\int_{X}\log\frac{\mu}{\mu_{0}}\mu,\label{eq:def of rel entropy}
\end{equation}
 when $\mu$ has a density $\frac{\mu}{\mu_{0}}$ with respect to
$\mu_{0}$ and otherwise $D_{\mu_{0}}(\mu):=\infty.$ If $\mu_{0}$
is a probability measure, then $D_{\mu_{0}}(\mu)\geq0$ and $D_{\mu_{0}}(\mu)=0$
iff $\mu=\mu_{0}$ (by Jensen's inequality). 
\begin{rem}
\label{rem:The-entropy-is}The ``physical entropy'' is usually defined
as 
\[
S(\mu):=-D_{\mu_{0}}(\mu)
\]
In fact, Sanov's theorem can be seen as a mathematical justification
of Boltzmann's original formula expressing the entropy $S$ as the
logarithm of the number of microscopic states consistent with a given
macroscopic state (using the characterization of a LDP in Prop \ref{prop:d-z}).
\end{rem}

\subsection{\label{subsec:Determinantal-point-processes}Determinantal point
processes, their deformations and random matrices}

In this section we compare the definition of the probability measure
\ref{eq:prod meas mubetaN intro} with the general setup of determinantal
point processes \cite{h-k-p}. A brief comparison with random matrix
theory is also made. 

Let $(M,\nu)$ be a measure space. A random point process with $N$
particles on $M$ is said to be\emph{ a projectional determinantal
point processes }if there exists a complex-valued Hermitian function
$\mathcal{K}(x,y),$ which defines a projection operator on $L^{2}(X,\nu)$
of rank $N$ on $M\times M,$ such that 
\begin{equation}
\mu^{(N)}=\frac{1}{N!}\det\left(\mathcal{K}(x_{i},x_{j})\right)_{i,j\leq N}\nu^{\otimes N}\label{eq:def of determ prob measure general}
\end{equation}
(the kernel $\mathcal{K}$ need only be defined a.e. wrt $\nu\otimes\nu).$
In the most general setup of determinantal point processes it is not
assumed that $\mathcal{K}$ defines a projection operator, but this
will always be assumed in the following and we will thus drop the
adjective projectional. A remarkable feature of a determinantal point
process is that all correlation measures \ref{rem:The-point-correlation}
can be expressed as determinants of the kernel $\mathcal{K}$ and,
in particular,
\begin{equation}
(\mu^{(N)})_{1}=\E(\delta_{N})=\frac{1}{N}\mathcal{K}(x,x)\nu\label{eq:expect in terms of kernel}
\end{equation}
Moreover, 

\[
\det\left(\mathcal{K}(x_{i},x_{j})\right)_{i,j\leq N}=\left||f(x_{1},...,x_{N})\right|^{2},
\]
 where 
\[
f(x_{1},...,x_{N})=\det\left(f_{i}(x_{j})\right)_{i,j\leq N_{k}}
\]
expressed in terms of a fixed orthonormal base $f_{1},...,f_{N}$
in the Hilbert space $\mathcal{H}_{N}\subset L^{2}(X,\nu)$ defined
as the image of the projection operator. More generally, if $f_{i}$
is a, possibly non-orthonormal, basis in $\mathcal{H}_{N},$ then
\[
\mu^{(N)}=\frac{1}{Z_{N}}\left||f(x_{1},...,x_{N})\right|^{2}\nu^{\otimes N},\,\,\,Z_{N}:=\int_{X^{N}}\left||f(x_{1},...,x_{N})\right|^{2}\nu^{\otimes N}
\]
The systematic study of determinantal point processes was initiated
by Macchi in the seventies who called them \emph{fermionic} point
processes, inspired by the properties of fermion gases in quantum
mechanics. From the physics point of view the functions $f_{i}$ play
the role of wave functions representing single fermions and $f$ (called
a Slater determinant) is the corresponding $N-$body wave function
(which is totally anti-symmetric, since the particles are fermions).
\begin{example}
\label{exa:The-Vandermonde determinantal point}The Vandermonde probability
measure \ref{eq:prod meas mubetaN intro} defines, when $\beta=k,$
a (projectional) determinantal random point process on $\C^{n}$ with
$N_{k}$ particles, $\nu=dV$ and $\mathcal{H}_{N_{k}}$ the space
of all weighted polynomials $p_{k}e^{-k\phi/2}$ of degree at most
$k.$ In physical terms $p_{k}e^{-k\phi/2}$ represents a fermion
(electron) on $\C^{n}$ subject to an exterior magnetic field represented
by the two-form $k\omega^{\phi}.$ Indeed, $p_{k}e^{-k\phi/2}$ is
in the null space of the Dirac operator coupled to a gauge field $A$
with curvature form $dA=i\omega^{\phi}$ (see \cite{berm 1 komma 5}
for the relation to boson-fermion correspondences). When $n=1$ the
corresponding probability measures appear in the collective description
of the integer Quantum Hall Effect (QHE) introduced by Laughlin, while
the case when $\beta=km$ for an integer $m$ appears in the fractional
QHE (see the survey \cite{kl}). 
\end{example}

Random Matrix Theory offers several famous examples of determinantal
point processes (see the book \cite{me} for further background): 
\begin{example}
Consider a random Hermitian matrix of rank $N,$ where the entries
are taken to be independent and normally distributed. Then the corresponding
eigenvalue distribution defines a determinantal point process on $\R.$
More precisely, if $\gamma_{N}$ denotes the centered Gaussian measure
of variance $N^{-1}$ on the space $H_{N}$ of all Hermitian matrices
of rank $N,$ then the push-forward of $\gamma_{N}$ to $\R^{N}$
(obtained by fixing an ordering of the eigenvalues) is given by the
determinantal point process on $X:=\R$ defined by the Vandermonde
expression \ref{eq:prod meas mubetaN intro} for $n=1$ and $\beta=k$
and the weight $\phi(x):=|x|^{2}/2,$ but with the measure $dV$ on
$\C$ replaced by the Lesbegue measure on $\R\subset\C.$ The corresponding
limiting equilibrium measure $\mu_{(\phi,\R)}$ is the semi-circle
law, $\mu_{(\phi,\R)}=1_{[-1,1]}(1-|x|^{2})^{1/2}2/\pi.$ The case
when $\phi(z):=|z|^{2}/2$ and $dV$ is Lebesgue measure on $\C$
also appears naturally in random matrix theory as the eigenvalue distribution
of a general rank $N$ matrix with random complex entries, i.e. the
centered Gaussian measure of variance $N^{-1}$ on $gl(N,\C)$ (this
is called the Ginibre ensemble). In this case $\mu_{\phi}$ is the
uniform probability measure on the unit-disc. Moreover, if $(A,B)\in gl(N,\C)^{2}$
are drawn independently from the Ginibre ensemble, then the eigenvalues
of $AB^{-1}$ are distributed according to the probability measure
\ref{eq:prod meas mubetaN intro} with $\phi=\phi_{FS}$ and $dV=d\lambda$
on $\C$ \cite{kr} (and then $\mu_{\phi}$ may be identified with
the invariant probability measure on the Riemann sphere). 
\end{example}

Given a positive number $p$ one can consider a more general class
of point processes by setting 
\[
\mu_{p}^{(N)}=\frac{1}{Z_{p}}\det\left(\mathcal{K}(x_{i},x_{j})\right)_{i,j\leq N}^{p/2}\nu^{\otimes N},
\]
 which specializes to the determinantal case when $p=2.$ Equivalently,
this means that $\mu^{(N)}$ can be expressed as 
\[
\mu_{p}^{(N)}=\frac{1}{Z_{N,p}}\left||f(x_{1},...,x_{N})\right|^{p}\nu^{\otimes N},\,\,\,Z_{N,p}:=\int_{X^{N}}\left||f(x_{1},...,x_{N})\right|^{p}\nu^{\otimes N},
\]
 which suggests that this setting can be viewed as an $L^{p}-$generalization
of determinantal point process. In the random matrix literature one
usually writes $p=\beta$ and calls ($M^{N},\mu_{p}^{(N)})$ a \emph{$\beta-$ensemble}
or a \emph{$\beta-$deformation} of a determinantal point process
(the parameter $\beta$ is called the Dyson index). In the present
setting in $\C^{n},$ introduced in Section \ref{subsec:Statistical-mechanics-and},
it is, however, convenient to use a rescaling and set
\[
p:=2\beta/k,
\]
 Thus the determinantal case appears when $\beta$ is taken to be
depend on $k$ as $\beta=k.$ 
\begin{example}
The cases $p=1,2$ and $4$ with $\phi(x)=|x|^{2}/2$ on the real
line appear naturally in Random Matrix Theory, as Gaussian ensembles
GOE, GUE and GSE \cite{me}. They are classified under their symmetry
group (the Orthogonal, Unitary and Symplectic group, respectively).
More generally, for any $p>0$ random matrix realizations were introduced
in \cite{d-e}, in the case of real eigenvalues, based on random Jacobi
matrices. Using such realization the large $N-$limit with $p=\beta N^{-1}$
was studied in \cite{d-s}, proving the analog of Theorem \ref{thm:conv in law in C^n intro}
for $n=1$ (and identifying the limiting measure $\mu_{\beta}$ with
the spectral measure of a semi-infinite Jacobi matrix). However, random
matrix interpretations in the complex case do not seem to be known,
for $p\neq2.$ 
\end{example}

\subsubsection{Determinantal point process for polarized manifolds}

Determinantal point processes may also be attached to a polarized
complex manifold $(X,L)$ when $L$ is endowed with an Hermitian metric
$\left\Vert \cdot\right\Vert $ and a $X$ with a measure $\nu,$
which is assumed to not charge pluripolar subsets. The corresponding
large $N-$limit was studied in \cite{berm 1 komma 5}. To explain
the setup we endow the $N_{k}-$dimension Hilbert space $H^{0}(X,kL)$
with the $L^{2}-$norm induced from $(\nu,\left\Vert \cdot\right\Vert )$
and denote by $K_{k}$ the corresponding Bergman kernel. In other
words, $K_{k}$ is the holomorphic section of $kL\boxtimes k\bar{L}\rightarrow X\times\overline{X}$
defined as the integral kernel of the orthogonal projection from $L^{2}(X,L)$
onto $H^{0}(X,kL).$ The corresponding determinantal probability measure
on $\mu^{(N_{k})}$ on $X^{N_{k}}$ may be defined as 
\[
\mu^{(N)}=\frac{1}{N!}\left\Vert \det\left(K_{k}(x_{i},x_{j})\right)_{i,j\leq N}\right\Vert \nu^{\otimes N}
\]
To see that this fits into the definition \ref{eq:def of determ prob measure general},
fix $e_{k}\in H^{0}(X,kL),$ not vanishing identically on $X,$ and
denote by $U$ the subset of $X$ where $e_{k}$ is non-vanishing.
On $U$ we can write $K_{k}=K_{k}(x,y)e_{k}\otimes\bar{e}_{k},$ where
$K_{k}(x,y)$ defines a holomorphic function on $X\times\overline{X}$
and set 
\[
\mathcal{K}(x,y):=K_{k}(x,y)\left\Vert e_{k}\right\Vert (x)\left\Vert e_{k}\right\Vert (y),
\]
which is defined a.e. wrt $\nu\otimes\nu$ (since, by assumption,
$\nu$ does not charge pluripolar subsets), as desired. As a consequence,
by formula \ref{eq:expect in terms of Bergman}, 
\begin{equation}
\E(\delta_{N})=\rho_{k}\nu,\,\,\,\rho_{k}:=\frac{1}{N}\left\Vert K_{k}(x,x)\right\Vert \label{eq:expect in terms of Bergman}
\end{equation}
 where the right hand side is called the \emph{Bergman measure} in
\cite{b-b-w} and $\rho_{k}$ is called the\emph{ Bergman density
function. }It admits the alternative representation 
\begin{equation}
\rho_{k}(x)=N_{k}^{-1}\sup_{\Psi\in H^{0}(X,kL)}\frac{\left\Vert \Psi(x)\right\Vert ^{2}}{\int_{X}\left\Vert \Psi\right\Vert ^{2}\nu}\label{eq:Bergman dens as quotient}
\end{equation}
\begin{example}
When $(X,L)=(\P^{n},\mathcal{O}(1))$ and $e_{k}$ is the $k$ th
tensor power of the section $Z_{0}$ the corresponding determinantal
point process on $M:=\C^{n}$ is precisely the one defined described
in example \ref{exa:The-Vandermonde determinantal point}. 
\end{example}

The corresponding deformed determinatal point processes will be studied
in Section \ref{sec:Proofs-for-the limit N} and referred to as \emph{temperature-deformed
determinantal point processes,} since $\beta$ plays the role of the
inverse temperature. 

\subsection{\label{subsec:Mean-field-approximations}Mean field approximations
and large deviations}

The notion of mean field approximations (and mean field theory) is
wide-spread in the physics literature and first appeared in the study
of phase transitions in ferromagnetic spins system (see the Appendix).
But here the notion will be used in a generalized sense, which goes
beyond the more standard ``linear'' setting of pair-interactions
encountered in the physics literature. 

Let us start with some heuristic arguments. Assume that the following
``Mean Field Approximation'' holds: 
\begin{equation}
E^{(N)}(x_{1},...x_{N}):=\frac{1}{N}H^{(N)}(x_{1},...,x_{N})\approx E(\frac{1}{N}\sum_{i=1}^{N}\delta_{x_{i}}),\,\,\,\,N>>1\label{eq:mean field apprx heurstic form}
\end{equation}
for some functional $E$ on $\mathcal{M}_{1}(X).$ 

By definition, given $\mu\in\mathcal{M}_{1}(X)$ and $\epsilon>0$
\[
\text{Prob }\left(\frac{1}{N}\sum_{i=1}^{N}\delta_{x_{i}}\in B_{\epsilon}(\mu)\right):=Z_{N,\beta}^{-1}\int_{\delta_{N}^{-1}\left(B_{\epsilon}(\mu)\right)}e^{-\beta NE^{(N)}}dV^{\otimes N}
\]
Hence, formally, as $N\rightarrow\infty$ and $\epsilon\rightarrow0,$
we can take out the factor $e^{-\beta NE^{(N)}}$ to get
\begin{equation}
\int_{\delta_{N}^{-1}\left(B_{\epsilon}(\mu)\right)}e^{-\beta NE^{(N)}}dV^{\otimes N}\sim e^{-\beta NE(\mu)}\int_{\delta_{N}^{-1}\left(B_{\epsilon}(\mu)\right)}dV^{\otimes N}\label{eq:take out}
\end{equation}
Applying the Sanov's LDP result \ref{prop:sanov} to the integral
thus suggests that the non-normalized measure 
\[
(\delta_{N})_{*}\left(e^{-\beta H^{(N)}}dV^{\otimes N}\right)
\]
satisfies a LDP with rate functional 
\[
F_{\beta}(\mu):=E(\mu)+\beta^{-1}D_{dV}(\mu)
\]

In order to make this argument rigorous two issues need to be confronted.
First, the nature of the convergence in the ``Mean Field Approximation''
\ref{eq:mean field apprx heurstic form} has to be specified. Secondly,
appropriate conditions on $H^{(N)}$ need to be introduced, ensuring
that the ``taking out'' argument \ref{eq:take out} is justified.
The simplest way to handle both issues is to assume the following
``regularity property'': 

\begin{equation}
\lim_{N\rightarrow\infty}\sup_{X^{N}}\left|E^{(N)}(x_{1},...x_{N})-E(\delta_{N})\right|=0\label{eq:reg prop}
\end{equation}
for a continuous functional $E(\mu)$ on $\mathcal{M}_{1}(X)$ \cite{e-h-t}.
However, this fails as soon as $H^{(N)}$ is singular, i.e. not point-wise
bounded. For example, this happens in the classical ``linear'' setting
of \emph{pair-interactions }with a mean field scaling:\emph{ }
\begin{equation}
E^{(N)}(x_{1},...,x_{N})=\frac{1}{N(N-1)}\frac{1}{2}\sum_{i\neq j}g(x_{i},x_{j})\label{eq:linear case}
\end{equation}
as soon as $g(x_{i},x_{j})$ is singular on the diagonal (as in the
case of the complex plane, formula \ref{eq:E and D in terms of g intro}$)$.
Here linearity refers to fact that the differential $dE_{|\mu}$ of
the corresponding functional $E$ is linear with respect to $\mu:$
\[
E(\mu)=\frac{1}{2}\int_{X}g(x,y)\mu(y)\otimes\mu(y)
\]
 and
\[
dE_{|\mu}=-\psi_{\mu}(x):=-\int_{X}g(x,y)\mu(y)
\]
 Still, using a different variational method introduced in \cite{m-s},
based on Gibbs variational principle, it follows from \cite{d-l-r,berm10,gz}
that the LDP in question does hold, if $g$ is assumed lower-semi
continuous (or, more generally, if an appropriate exponential integrability
assumption on $g$ is assumed, as shown in \cite{berm10}, building
on \cite{clmp,k2}). 

For a general Hamiltonian it was as shown in \cite{berm10} that,
under certain technical assumptions, the LDP holds if the ``Mean
Field Approximation'' assumption \ref{eq:mean field apprx heurstic form}
is taken to hold in the following sense: 
\begin{equation}
\lim_{N\rightarrow\infty}\int_{X^{N}}E^{(N)}\mu^{\otimes N}=E(\mu)\label{eq:conv of mean energies}
\end{equation}
 (which is trivially the case in the ``linear case'' \ref{eq:linear case}).
However, in the present setting the Hamiltonian $H^{(N)}$ (as defined
by formula \ref{eq:def of H N weighted intro}) is both singular and
highly non-linear. Moreover, the convergence \ref{eq:conv of mean energies}
is yet to be established (see Section \ref{subsec:The-case-of infin}).
On the other hand, the ``Mean Field Approximation''\ref{eq:mean field apprx heurstic form}
does hold in the sense of Gamma-convergence. Exploiting that $H^{(N)}$
is uniformly super-harmonic, the ``taking out'' argument in formula
\ref{eq:take out} can then be rigorously justified (see Section \ref{sec:Proofs-for-the limit N}).
\begin{rem}
The regularity property \ref{eq:reg prop} is preserved when the sign
of $H^{(N)}$ (or, equivalently, the sign of $\beta)$ is switched.
However, in the present singular setting the sign is crucial and the
``negative temperature case'' is yet to be established (see Section
\ref{subsec:The-case-of-neg}).
\end{rem}

\subsubsection{Mean field equations}

Assume that $\mu$ is a critical point of the free energy functional
$F_{\beta}$ in the interior of $\mathcal{M}_{1}(X)$ and that the
variational derivative of $E(\mu)$ exists at $\mu,$ in a suitable
sense: 
\[
dE_{|\mu}=-\varphi_{\mu},
\]
for a function $\varphi_{\mu}$ on $X,$ defined up to an additive
constant. Then $\mu$satisfies the following equation, known as a\emph{
mean field equation} in the physics literature,
\begin{equation}
\mu=e^{\beta\varphi_{\mu}}/Z\label{eq:mean field eq in heur}
\end{equation}
for some positive constant $Z.$ This follows directly from the fact
that the differential of $D_{\mu_{0}}$ at $\mu$ is represented by
the function $\log(\mu/\mu_{0})$ if $D_{\mu_{0}}(\mu)<\infty.$ However,
if $E$ is singular and non-linear it is, in general, highly non-trivial
to prove that a general minimizer of $E$ satisfies the equation \ref{eq:mean field eq in heur}.
In the present complex geometric geometric this was accomplished in
\cite{berm6}, building on \cite{bbgz}.

\section{\label{sec:Proofs-for-the limit N}Proofs for the limit $N\rightarrow\infty$}

We now turn to the proofs of the results stated in Section \ref{sec:A-birds-eye}.
We will use the compactification of $\C^{n}$ by $\P^{n}$ to reduce
the proofs to the setting of a compact polarized manifold.

Let thus $(X,L)$ be a polarized compact complex manifold and set
\[
N_{k}:=\dim H^{0}(X,kL)=Vk^{n}+o(k^{n}),
\]
 where by assumption $V>0$ (see Section \ref{subsec:Polarized-compact-manifolds}).
We fix once and for all the back-ground data $(\left\Vert \cdot\right\Vert ,dV)$
consisting of a Hermitian metric $\left\Vert \cdot\right\Vert $ on
$L$ and a volume form $dV$ on $X.$ We will denote by $\omega$
the curvature two-form of the metric $\left\Vert \cdot\right\Vert $
on $L$. The data $(\left\Vert \cdot\right\Vert ,dV)$ induces a Hilbert
space structure on $H^{0}(X,kL).$ The corresponding\emph{ Slater
determinant} $\det\Psi^{(k)}$ is the holomorphic section of $(kL)^{\boxtimes N_{k}}\rightarrow X^{N_{k}}$
determined, up to a choice of sign, by demanding that it be totally
anti-symmetric and normalized: 
\begin{equation}
(N_{k}!)^{-1}\int_{X^{N_{k}}}\left\Vert \det\Psi^{(k)}\right\Vert ^{2}dV^{\otimes N_{k}}=1\label{eq:normalization of Slater}
\end{equation}
Concretely,
\begin{equation}
\det\Psi^{(k)}(x_{1},...,x_{N_{k}}):=\det\left(\Psi_{i}^{(k)}(x_{j})\right)_{i,j\leq N_{k}}\label{eq:det section as a determinant}
\end{equation}
where $\Psi_{1}^{(k)},...,\Psi_{N_{k}}^{(k)}$ is a fixed orthonormal
base in the Hilbert space $H^{0}(X,kL).$ 
\begin{rem}
\label{rem:generator}The section $\det\Psi^{(k)}$ is uniquely determined,
up to a multiplicative non-zero complex number, by the property that
it generates the top exterior power $\Lambda^{N_{k}}H^{0}(X,kL)$
of the complex vector space $H^{0}(X,kL),$ viewed as a one-dimensional
subspace of $H^{0}(X^{N_{k}},(kL)^{\boxtimes N_{k}})$ (the totally
anti-symmetric part). In particular, if $\det\widetilde{\Psi}^{(k)}(x_{1},...,x_{N_{k}})$
denotes any other generator, then 

\[
\det\widetilde{\Psi}^{(k)}(x_{1},...,x_{N_{k}})=C_{k}\det\Psi^{(k)}(x_{1},...,x_{N_{k}}),\,\,\,C_{k}\in\C^{*}
\]
 (hence, $|C_{k}|=1$ if the normalization condition \ref{eq:normalization of Slater}
holds). 
\end{rem}

\begin{example}
\label{exa:Vandermonde homogen}By homogenization (see Section \ref{subsec:Compactification-of-})
the Vandermonde determinant $D^{(N_{k})}$ on $(\C^{n})^{N_{k}}$
may be identified with a generator for $\Lambda^{N_{k}}H^{0}(\P^{n},\mathcal{O}(k))$
that we shall denote by the same symbol $D^{(N_{k})}.$ Concretely,
over $\P^{n}$ $D^{(N_{k})}$ is of the form \ref{eq:det section as a determinant}
with $\Psi_{i}^{(k)}$ a basis of homogeneous multinomials of degree
$k.$ However, this basis is not orthonormal with respect to $(\left\Vert \cdot\right\Vert ,dV),$
for any volume form on $\P^{n}.$
\end{example}

We define the \emph{energy $E_{\omega}^{(N)}(x_{1},...,x_{N_{k}})$}
of a configuration $x_{1},...x_{N}$ of $N$ points on $X^{N},$ relative
$\omega,$ by

\[
E_{\omega}^{(N)}(x_{1},...,x_{N_{k}}):=-\frac{1}{N_{k}k}\log\left\Vert \det\Psi^{(k)}(x_{1},...,x_{N_{k}})\right\Vert ^{2},
\]
 which is canonically attached to the data $(X,L,\left\Vert \cdot\right\Vert ,dV)$
or more precisely to $(X,L,\omega,dV)$ (since $\det\Psi^{(k)}$ is
uniquely determined up to a sign). 

To the data $(\left\Vert \cdot\right\Vert ,dV,\beta),$ where $\beta$
is a given positive number, we associate the following symmetric probability
measure on $X^{N_{k}}:$
\begin{equation}
\mu_{\beta}^{(N_{k})}:=\frac{\left\Vert \det\Psi^{(k)}(x_{1},...,x_{N_{k}})\right\Vert ^{2\beta/k}}{Z_{N_{k},\beta}}dV^{\otimes N_{k}}.\label{eq:temper deformed text}
\end{equation}

The corresponding random point process on $X$ will be called the
\emph{temperature deformed determinantal point process} attached to
$(\left\Vert \cdot\right\Vert ,dV,\beta).$ The probability measure
\ref{eq:temper deformed text} is the Gibbs measure on $X^{N_{k}},$
at inverse temperature $\beta,$ associated to the Hamiltonian $NE_{\omega}^{(N)}:$\emph{
\[
\mu_{\beta}^{(N)}:=\frac{e^{-\beta NE_{\omega}^{(N)}}}{Z_{N,\beta}}dV^{\otimes N}
\]
}

By homogenity, $\mu_{\beta}^{(N)}$ is independent of the choice of
generator of $\Lambda^{N_{k}}H^{0}(X,kL)$ (compare Remark \ref{rem:generator}).
But the point of choosing the particular generator $\det\Psi^{(k)}$
is that the corresponding determinantal energy $E_{\omega}^{(N)}$
admits a large $N-$limit. This is made precise by the following result
(the definition of the projection operator $P_{(K,\omega)}$ is given
in Section \ref{subsec:The-PSH-projection-}):
\begin{thm}
\label{thm:Gamma conv polarized}Let $(X,L)$ be a compact polarized
manifold. Then $E_{\omega}^{(N)}$ Gamma-converges towards the pluricomplex
energy $E_{\omega}$ on $\mathcal{M}_{1}(X),$ as $N\rightarrow\infty.$
As a consequence, if $(x_{1},...,x_{N_{k}})$ is a minimizer $E_{\omega}^{(N)}$
on $X^{N},$ i.e. it maximizes $\left\Vert \det\Psi^{(k)}\right\Vert ,$
then the corresponding empirical measures 
\[
\delta_{N}:=N^{-1}\sum_{i=1}^{N}\delta_{x_{i}}
\]
 converge weakly towards the unique minimizer $\mu_{\omega}$ of $E_{\omega}^{(N)}$
on $\mathcal{P}(X).$ More generally, if $K$ is a non-polar subset
of $X$ and $(x_{1},...,x_{N_{k}})$ is a minimizer $E_{\omega}^{(N)}$
on $K^{N},$ then the corresponding empirical measures converge weakly
towards the unique minimizer $\mu_{(K,\omega)}$ of $E_{\omega}$
on $\mathcal{M}_{1}(X),$ where
\begin{equation}
\mu_{(K,\omega)}=MA_{\omega}\varphi_{(K,\omega)},\,\,\,\varphi_{(K,\omega)}:=P_{(K,\omega)}0\label{eq:min of E as MA}
\end{equation}
\end{thm}

\begin{proof}
This result follows from combining the general convergence criterion
Prop \ref{prop:crit for gamma conv} with Theorem A and B in \cite{b-b}.
To see this denote by $f_{N}$ the functional on $C(X)$ defined as
the Legendre-Fenchel transform of $E_{\omega}^{(N)}:$ 
\begin{equation}
f_{N}(u):=\sup_{X^{N}}-\frac{1}{N_{k}k}\log\left\Vert \det\Psi^{(k)}(x_{1},...,x_{N_{k}})\right\Vert _{ku}^{2}\label{eq:def of f_N in terms of H}
\end{equation}
 where $\left\Vert \cdot\right\Vert _{ku}$ denotes the metric on
$kL$ defined by $\left\Vert \cdot\right\Vert _{ku}^{2}=\left\Vert \cdot\right\Vert ^{2}e^{-ku}$
(as in Section \ref{subsec:Polarized-compact-manifolds}). By \cite[Thm A]{b-b}
\begin{equation}
f_{N}(u)\rightarrow f(u),\,\,\,-f(-u)=\mathcal{E}\circ P_{\omega},\label{eq:f in terms of beautf e}
\end{equation}
Moreover, by \cite[Thm B]{b-b}, the functional $\mathcal{E}(P_{\omega}u)$
is Gateaux differentiable on $C(X).$ Hence, all that remains is to
verify that the Legendre-Fenchel transform of $f$ coincides with
$E_{\omega}(\mu).$ But this follows readily from the definition \ref{eq:def of e as sup}
of $E_{\omega}(\mu),$ only using that $P_{\omega}$ is an increasing
projection from $C(X)$ onto $PSH(X,\omega)\cap C(X).$ Moreover,
by general convex duality theory (see \cite{berm6}) it also follows
that $E_{\omega}(\mu)$ admits a unique minimizer of $\mathcal{M}_{1}(X),$
namely the differential of $\mathcal{E}\circ P_{\omega}$ at $u=0,$
which is given by formula \ref{eq:min of E as MA} (by \cite[Thm B]{b-b}).
The latter formula holds as long as $K$ is not pluripolar (which
ensures that $P_{(K,\omega)}$ is well-defined).

For completeness, let us also outline the proof of Theorem A in \cite{b-b},
i.e. the convergence \ref{eq:f in terms of beautf e}, in the present
setting where $L$ is assumed ample. Given $u\in C(X)$ we set \emph{
\begin{equation}
\mathcal{F}_{k,L^{p}}[u]:=-\frac{1}{kN_{k}}\log\left\Vert \det\Psi^{(k)}\right\Vert _{L^{p}(X^{N_{k}},u,dV)}^{2}\label{eq:def of beati F p}
\end{equation}
}defined in terms of $L^{p}-$norm $\left\Vert \cdot\right\Vert _{u}$
on $H^{0}(X^{N_{k}},(kL)^{N_{k}})$ induced by $(u,dV),$ for $p\in[1,\infty]$
(which is defined to be the ordinary sup-norm for $p=\infty$ and
thus independent of $dV).$ For $p=\infty,$ which is the case appearing
in the convergence \ref{eq:f in terms of beautf e}, first observe
that 
\begin{equation}
\mathcal{F}_{k,L^{\infty}}[u]=\mathcal{F}_{k,L^{\infty}}[P_{\omega}u].\label{eq:norm wrt phi as wrt P phi}
\end{equation}
Indeed, by the definition of the operator $P$ we have
\[
\sup_{X}(e^{\varphi-u)})=\sup_{X}e^{(\varphi-P_{\omega}u)}
\]
 for any $\varphi\in PSH(X,\omega).$ Applying the previous equality
to each factor of $X^{N_{k}}$ then gives the equality \ref{eq:norm wrt phi as wrt P phi}.
Moreover, the functional $\mathcal{F}_{k,L^{\infty}}$ is equicontinuous
with respect to the sup-norm on $C(X)$ (as follows from the fact
that the $\mathcal{F}_{k,L^{\infty}}$ is increasing and satisfies
$\mathcal{F}_{k,L^{\infty}}[u+C]=\mathcal{F}_{k,L^{\infty}}[u]+C,$
when $C\in\R).$ Accordingly, writing $P_{\omega}u$ as the uniform
limit of Kähler potentials $u_{j},$ i.e. smooth $u_{j}\in PSH(X,\omega)$
such that $\omega_{u_{j}}>0$ (using the approximation property in
Prop \ref{prop:omega psh compact appr} ) it is enough to establish
the convergence of $\mathcal{F}_{k,L^{\infty}}[\varphi]$ when $\varphi$
is a Kähler potential. To this end one uses that
\begin{equation}
\mathcal{F}_{k,L^{\infty}}[\varphi]=\mathcal{F}_{k,L^{2}}[\varphi]+o(1),\label{eq:distortion of norms}
\end{equation}
where the error term $o(1)$ (tending to zero) only depends on the
modulus of continuity of $\varphi.$ Indeed, this follows directly
from applying the standard submean property of holomorphic functions
on small coordinate balls on $X,$ for each factor of $X^{N_{k}}.$
Now, a direct calculation reveals that the differential of $\mathcal{F}_{k,L^{2}}$
at any $u\in C(X)$ is given by
\[
d(\mathcal{F}_{k,L^{2}})_{|u}=\frac{1}{N_{k}}\rho_{ku}dV
\]
 where the function $\rho_{ku}$ is the restriction to the diagonal
of the point-wise norm of the Bergman kernel of the Hilbert space
$\left(H^{0}(X,kL),\left\Vert \cdot\right\Vert _{L^{2}(X,u,dV)}\right).$
The asymptotics of $\rho_{ku},$ when $u$ is a Kähler potential,
i.e. when it defines a smooth metric on $L$ with strictly positive
curvature $\omega_{u}$ are well-known (and due to Bouche and Tian,
independently) and, in particular, give that 
\begin{equation}
(i)\,\lim_{k\rightarrow\infty}\frac{1}{N_{k}}\rho_{ku}dV=\frac{1}{V}(\omega_{u})^{n}\,\,\,(ii)\,\frac{1}{N_{k}}\rho_{ku}\leq C\label{eq:B-T}
\end{equation}
in the weak topology (see \cite{berm 1 komma 1} for an elementary
proof in the $\C^{n}-$setting, based on formula \ref{eq:Bergman dens as quotient}
and the sub-mean property of holomorphic functions). Using the defining
property \ref{eq:def of energyfunc as primitive} of the functional
$\mathcal{E}$ (and integrating along a line segment in the space
of all positively curved metrics) this gives
\[
\mathcal{F}_{k,L^{\infty}}[u]=\mathcal{E}_{u}(u)+o(1),
\]
 which proves the convergence \ref{eq:f in terms of beautf e}, thanks
to \ref{eq:norm wrt phi as wrt P phi} and \ref{eq:distortion of norms}. 

Finally, we recall that the differentiability statement (Thm B in
\cite{b-b}) is proved using the ``orthogonality relation'' (see
also \cite{l-ng} for an elegant simplification of the original proof).
\end{proof}
\begin{rem}
From a probabilistic point of view, an interesting feature of the
previous proof is that it reduces the (deterministic) Gamma-convergence
problem to establishing the weak convergence of the the expectations
$\E(\delta_{N})$ for the corresponding determinantal point process
(corresponding to $p=2).$ Indeed, by formula \ref{eq:expect in terms of Bergman},
$\E(\delta_{N_{k}})$ coincides with the corresponding Bergman measure
$\frac{1}{N_{k}}\rho_{k}dV$ and $\mathcal{F}_{k,L^{2}}(u)$ (formula\ref{eq:def of beati F p})
essentially coincides with the logarithmic moment generating function
of the determinantal point process. To make the connection to Kähler
geometry one can use formula \ref{eq:det section as a determinant}
to express the functional $\mathcal{F}_{k,L^{2}}$ as
\[
\mathcal{F}_{k,L^{2}}(u)=-\frac{1}{kN_{k}}\log\left(N_{k}!\det_{i,j\leq N_{k}}\left\langle \Psi_{i}^{(k)},\Psi_{j}^{(k)}\right\rangle _{(ku,dV)}\right),
\]
 in terms of the scalar product on $H^{0}(X,kL)$ induced by $(u,dV).$
This reveals that $\mathcal{F}_{k,L^{2}}$ is essentially Donaldson's
L-functional, which plays a prominent role in Kähler geometry (see
\cite{b-b,b-b-w} and references therein).

We next consider the \emph{determinantal energy }
\[
E^{(N_{k})}(z_{1},...z_{N_{k}}):=-\frac{1}{N_{k}k}\log\left|D^{(N_{k})}(z_{1},...z_{N_{k}})\right|^{2}
\]
in $\C^{n},$ introduced in connection to the study of Fekete points
on compact subsets of $\C^{n}$ in Section \ref{subsec:Fekete-points-in}.
By the next corollary $E^{(N_{k})}$ Gamma-converges towards the (non-weighted)
pluricomplex energy $E(\mu),$ defined in Section \ref{subsec:The-(non-)weighted-pluricomplex}.
\end{rem}

\begin{cor}
Let $K$ be a compact subset of $\C^{n}$ which is not pluripolar.
Then the determinantal energy $E^{(N_{k})}$ converges towards the
(non-weighted) pluricomplex energy $E(\mu)$ on $\mathcal{M}_{1}(K).$
As a consequence, if $(z_{1},...,z_{N_{k}})$ are minimizers of $E^{(N_{k})}$
on $K^{N_{k}},$ i.e. Fekete points for $K,$ then the corresponding
empirical measure converges weakly towards the pluripotential equilibrium
measure of $K.$
\end{cor}

\begin{proof}
Take a metric $\left\Vert \cdot\right\Vert $ on $\mathcal{O}(1)\rightarrow\P^{n}$
which coincides with the trivial metric on an open subset containing
$K,$ under the standard embedding of $\C^{n}$ in $\P^{n}.$ Identifying
the Vandermonde determinant with a section over $(\P^{n})^{N_{k}}$
(using the identifications in Section \ref{subsec:Compactification-of-})
, we can, over $K\Subset\P^{n},$ rewrite 
\[
\log|D^{(N_{k})}|^{2}=\log\frac{\left\Vert D^{(N_{k})}\right\Vert ^{2}}{\left\Vert D^{(N_{k})}\right\Vert _{L^{2}(X^{N_{k}})}^{2}}+\log\left\Vert D^{(N_{k})}\right\Vert _{L^{2}(X^{N_{k}})}^{2}
\]
By homogenity and Remark \ref{rem:generator} we may as well replace
$D^{(N_{k})}$ with $\det\Psi^{(k)}$ in the quotient appearing in
the right hand side above. Hence, 
\[
-k^{-1}N_{k}^{-1}\log|D^{(N_{k})}|^{2}=E_{\omega}^{(N)}-k^{-1}N_{k}^{-1}\log\left\Vert D^{(N_{k})}\right\Vert _{L^{2}(X^{N_{k}})}^{2}
\]
By the Gamma-convergence in the previous theorem, combined with Prop
\ref{prop:crit for gamma conv}, all that remains is to verify that

\begin{equation}
-\lim_{N_{k}\rightarrow\infty}k^{-1}N_{k}^{-1}\log\left\Vert D^{(N_{k})}\right\Vert _{L^{2}(X^{N_{k}})}^{2}=\mathcal{E}_{\omega}(P_{T^{n}}0)\label{eq:asympt of D in pf}
\end{equation}
But this follows from \cite[Thm A]{b-b-w}. Indeed, $D^{(N_{k})}$
is the Slater determinant associated to a the basis in $H^{0}(\P^{n},k\mathcal{O}(1))$,
which is orthonormal with respect to the weighted measure $(\nu,0),$
where $\nu$ is the standard invariant measure on the unit $n-$torus
$T^{n}$ in $\C^{n}\Subset\P^{n}$ and $0$ denotes the trivial metric
on $\mathcal{O}(1)\rightarrow\C^{n}\Subset\P^{n}.$ This measure is
a Bernstein-Markov measure, i.e. the the corresponding Bergman density
$\rho_{k}$ has sub-exponential growth. In fact, by the $T^{n}-$symmetry,
$\rho_{k}=1.$ Hence, \ref{eq:asympt of D in pf} follows from Theorem
\cite[Thm A]{b-b-w} (the proof for Bernstein-Markov measures is essentially
the same as the proof of the case when $\nu=dV$ on $X,$ outlined
in the proof of the previous theorem, using that \ref{eq:distortion of norms}
still holds for a Bernstein-Markov measure). 
\end{proof}
Turning to the probabilistic setting we next explain the proof of
the following general result in \cite{berm8} (which covers Theorem
\ref{thm:LDP subharm intro}, stated in Section \ref{sec:A-birds-eye}). 
\begin{thm}
\label{thm:LDP subh text}Assume that $X$ is a compact Riemannian
manifold without boundary and that the sequence $\frac{1}{N}H^{(N)}(x_{1},...,x_{N})$
Gamma converges to a functional $E(\mu)$ on $\mathcal{M}_{1}(X).$
If moreover, $H^{(N)}$ is uniformly quasi-superharmonic, i.e. there
exists a constant $C$ such that for all $N$
\[
\Delta_{x_{i}}H^{(N)}(x_{1},x_{2},...x_{N})\leq C\,\,\,\,X^{N},
\]
 then the measures 
\[
(\delta_{N})_{*}\left(e^{-\beta H^{(N)}}dV^{\otimes N}\right)
\]
satisfy a Large Deviation Principle (LDP) on\emph{ $\mathcal{M}_{1}(X)$
}with speed $\beta N$ and rate functional $F_{\beta}(\mu).$ As a
consequence, the laws of the random measures $\delta_{N}$ on $(X^{N},\mu_{\beta}^{(N)})$
satisfy a Large Deviation Principle (LDP) with speed $\beta N$ and
rate functional $F_{\beta}(\mu)-C_{\beta},$ where 
\[
C_{\beta}=\inf_{\mathcal{M}_{1}(X)}F_{\beta}
\]
In particular, if $E$ is convex and $\beta>0$ then $\delta_{N}$
converges in law towards the unique minimizer $\mu_{\beta}$ of $F_{\beta}.$
\end{thm}

\begin{proof}
The idea of the proof is to make rigorous the formal mean field approximation
argument given in Section \ref{subsec:Mean-field-approximations},
by exploiting the assumed quasi-superharmonicity of $H^{(N)}.$ Using
the characterization of a LDP in Proposition \ref{prop:d-z}, the
upper bound in the first LDP stated in the theorem follows almost
immediately from the liminf property of the Gamma-convergence together
with Sanov's theorem. Indeed, 
\[
\int_{\delta_{N}^{-1}\left(B_{\epsilon}(\mu)\right)}e^{-\beta NE^{(N)}}dV^{\otimes N}\leq\sup_{\delta_{N}^{-1}\left(B_{\epsilon}(\mu)\right)}e^{-\beta NE^{(N)}}\int_{\delta_{N}^{-1}\left(B_{\epsilon}(\mu)\right)}dV^{\otimes N},
\]
Hence, using the assumed Gamma-convergence gives
\[
\limsup_{N\rightarrow\infty}N^{-1}\log\int_{\delta_{N}^{-1}\left(B_{\epsilon}(\mu)\right)}e^{-\beta NE^{(N)}}dV^{\otimes N}\leq-\beta\inf_{B_{\epsilon}(\mu)}E(\mu)+\limsup_{N\rightarrow\infty}N^{-1}\log\int_{\delta_{N}^{-1}\left(B_{\epsilon}(\mu)\right)}dV^{\otimes N},
\]
Finally, letting $\epsilon\rightarrow0$ and invoking Sanov's LDP
result \ref{prop:sanov} concludes the proof of the upper bound in
the LDP.

In order to prove the lower bound in the LDP we first observe that
it is enough to establish an estimate of the following form: there
exists a constant $C>0$ and $r(\epsilon)$ tending to zero, as $\epsilon\rightarrow0,$
such 
\begin{equation}
e^{CN\epsilon}\int_{\delta_{N}^{-1}\left(B_{\epsilon}(\delta_{N}(\boldsymbol{x}^{(N)})\right)}e^{-\beta NE^{(N)}}dV^{\otimes N}\geq\sup_{\delta_{N}^{-1}\left(B_{r(\epsilon)}(\delta_{N}(\boldsymbol{x}^{(N)}))\right)}e^{-\beta NE^{(N)}}\int_{\delta_{N}^{-1}\left(B_{r(\epsilon)}(\mu)\right)}dV^{\otimes N},\label{eq:suffiecnt estimate}
\end{equation}
for any given $\boldsymbol{x}^{(N)}:=(x_{1},..,x_{N})\in X^{N}.$
Indeed, by the Gamma-convergence assumption there exists a sequence
$\boldsymbol{x}^{(N)}:=(x_{1},..,x_{N})\in X^{N}$ such that $\delta_{N}(\boldsymbol{x}^{(N)})\rightarrow\mu$
and $E^{(N)}(\boldsymbol{x}^{(N)})\rightarrow E(\mu).$ In particular,
\[
\delta_{N}^{-1}\left(B_{\epsilon}(\mu)\right)\supseteq\delta_{N}^{-1}\left(B_{r(\epsilon)}(\delta_{N}(\boldsymbol{x}^{(N)})\right),
\]
 when $N$ is sufficiently large. Hence, if \ref{eq:suffiecnt estimate}
holds, then
\[
e^{CN\epsilon}\int_{\delta_{N}^{-1}\left(B_{\epsilon}(\delta_{N}(\boldsymbol{x}^{(N)})\right)}e^{-\beta NE^{(N)}}dV^{\otimes N}\geq e^{-\beta NE^{(N)}(\boldsymbol{x}^{(N)})}\int_{\delta_{N}^{-1}\left(B_{r(\epsilon)}(\mu)\right)}dV^{\otimes N}.
\]
 From here the rest argument proceeds essentially as in the proof
of the upper bound. The good news is that the estimate \ref{eq:suffiecnt estimate}
does hold with $r(\epsilon)=\epsilon^{2},$ when the metric $d$ on
$\mathcal{M}_{1}(X)$ is taken as the Wasserstein $L^{2}-$metric
$d_{W_{2}}$ induced from the Riemannian metric $g$ on $X.$ This
is the content of \cite[Prop 3.9]{berm8}, which thus concludes the
proof of the LDP in question. 

Finally, let us point out that the starting point of the proof of
the inequality \ref{eq:suffiecnt estimate} is the well-known fact
that the embedding $\delta_{N}$ of $X^{N}/S_{N}$ into $(\mathcal{M}_{1}(X),d_{W_{2}})$
is an isometry when $X^{N}/S_{N}$ is endowed with the quotient space
metric induced from the Riemannian metric $g_{N}$ on $X^{N},$ defined
as $N^{-1}$ times the product Riemannian metric. The quasi-subharmonicity
assumption on $E^{(N)}$ is equivalent to 
\[
\Delta_{g_{N}}E^{(N)}\geq-\lambda
\]
 on $X^{N}.$ Moreover, the scaling of $g_{N}$ also ensures that
the Ricci curvature of $g^{(N)}$ is bounded from below by a uniform
constant times the dimension of $X^{N}.$\emph{ }The inequality\ref{eq:suffiecnt estimate}
now follows from the general sub-mean inequality in \cite[Thm 2.1]{berm8}
for Riemannian quotients (orbifolds) $Y:=M/G$ (which yields a distortion
factor with sub-exponential growth in the dimension). In turn, the
latter inequality is proved using geometric analysis on the orbifold
$Y,$ by generalizing a fundamental inequality of Li-Schoen in Riemannian
geometry \cite{li-sc}.

The previous theorem, applied to the temperature-deformed determinantal
point processes \ref{eq:temper deformed text}, yields the following
LDP obtained from \cite{berm8}:
\end{proof}
\begin{thm}
\label{thm:LDP for temp deform on polar}Let $(X,L)$ be a compact
polarized manifold and consider the temperature deformed determinantal
point processes $(X^{N},\mu_{\beta}^{(N)})$ on $X$ associated to
given data $(\left\Vert \cdot\right\Vert ,dV,\beta)$ with $\beta>0.$
Then the the laws of the random measures $\delta_{N}$ on $(X^{N},\mu_{\beta}^{(N)})$
satisfy a Large Deviation Principle (LDP) with speed $\beta N$ and
rate functional $F_{\beta}(\mu)-C_{\beta},$ where 
\[
F_{\beta}(\mu)=E_{\omega}(\mu)+\beta^{-1}D_{dV}(\mu)-C_{\beta}
\]
 with 
\[
C_{\beta}=\inf_{\mathcal{M}_{1}(X)}F_{\beta}
\]
Moreover, $F_{\beta}$ admits a unique minimizer $\mu_{\beta}$ on
$\mathcal{M}_{1}(X)$ and it can be expressed as $\mu=e^{\beta u_{\beta}}dV,$
where $u_{\beta}$ is the unique finite energy solution to the complex
Monge-Ampère equation
\begin{equation}
MA_{\omega}(u)=e^{\beta u}dV.\label{eq:ma eq in thm temp deformed text}
\end{equation}
More generally, the corresponding statement holds when $dV$ is replaced
by a singular finite measure $dV_{s}$ of the form 
\[
dV_{s}=e^{v}dV
\]
 for $v$ a function on $X$ with analytic singularities, i.e locally
$v=u_{0}+v,$ where $u_{0}$ is continuous and $v$ has analytic singularities,
i.e. locally 
\begin{equation}
v=\sum_{i=1}^{m}\lambda_{i}\log|h_{i}|^{2}\label{eq:analytic sin}
\end{equation}
 for some $\lambda_{i}\in\R$ and holomorphic functions $h_{i}.$ 
\end{thm}

\begin{proof}
The LDP in the case of a volume form $dV$ follows directly from the
previous theorem together with the Gamma-convergence in Theorem \ref{thm:Gamma conv polarized}
(the quasi-suberharmonicity of $E^{(N)}$ is a consequence of quasi-plurisubharmonicity
in example \ref{exa:hol section becomes psh}). The singular case
$dV_{s}$ also follows in a similar way (see \cite{berm8 comma 5}).
Since $E$ is convex and $S$ is strictly convex $F_{\beta}$ admits
precisely one minimizer. Finally, to see that the minimizer satisfies
the mean field type equation \ref{eq:temper deformed text} one can
use the sub-gradient property in Prop \ref{prop:sub-gradient} and
the convexity of $F_{\beta}$ (see \cite{berm8 comma 5}).
\end{proof}
We next turn to the $\C^{n}-$version of the previous theorem, which
is a new result (which covers Theorem \ref{thm:conv in law in C^n intro},
stated in Section \ref{sec:A-birds-eye}). We will denote by $E_{\phi}(\mu)$
the weighted pluricomplex energy of a probability measure $\mu$ on
$\C^{n},$ defined in Section \ref{subsec:The-(non-)weighted-pluricomplex}.
\begin{thm}
Given a weighted measure $(\phi,dV)$ which is admissible with respect
to $\beta$ (formula \ref{eq:adm}) the random measures $\delta_{N_{k}}$
on the probability spaces $((\C^{n})^{N},\mu_{\beta}^{(N_{k})})$
converge in law , as $N_{k}\rightarrow\infty,$ to the unique minimizer
$\mu_{\beta}$ of the following free energy type functional on $\mathcal{M}_{1}(\C^{n}):$
\[
F_{\phi,\beta}(\mu):=E_{\phi}(\mu)+\beta^{-1}D_{dV}(\mu),
\]
Moreover, the minimizer satisfies $\mu=e^{\beta\psi_{\beta}}dV,$
where $\psi_{\beta}$ is the unique solution in $\mathcal{E}^{\text{1}}(\C^{n})$
to 
\[
MA(\psi)=e^{\beta(\psi-\phi)}dV
\]
\end{thm}

\begin{proof}
We will identify $\phi$ with a lsc on $X:=\P^{n}$ which is equal
to $\infty$ on $A:=\P^{n}-\C^{n}.$ Accordingly, we identify the
probability measure $\mu_{\beta}^{(N)}$ on $(\C^{n})^{N_{k}}$ with
a probability measure on $X^{N_{k}}.$ We are going to prove the stronger
result that the laws of the corresponding empirical measure $\delta_{N}$
satisfy an LDP with rate functional $F_{\phi,\beta}$ on $\mathcal{M}_{1}(X).$
To this end, first recall that we may assume that $dV=d\lambda$ so
that the growth condition becomes
\[
\phi\geq\left(1+n/\beta+\epsilon\right)\phi_{FS}-C
\]
 on $\C^{n}$ (as pointed out in the discussion before the statement
of Theorem \ref{thm:conv in law in C^n intro})

\emph{Step 1: The LDP holds in the case when $\phi(z)=C_{\beta}\phi_{FS}+u$
for some constant $C_{\beta}$ and function $u$ such that $u$ extends
to a continuous function on $\P^{n}.$}

In this case we can rewrite 
\[
e^{-\beta\phi}dV=e^{-(\beta\phi_{FS}+u)}dV_{s},\,\,\,dV_{s}:=e^{-v}dV_{FS},\,\,\,v(z)=C'\phi_{FS}(z)
\]

The LDP now follows from the previous theorem, using that $v$ has
analytic singularities, viewed as a function on $\P^{n}.$ Indeed,
expressed in terms of homogeneous coordinates $Z_{i}$ on $\P^{n},$
viewed as holomorphic sections of $\mathcal{O}(1)\rightarrow\P^{n},$
we have have $\phi_{FS}:=\log(1+|z|^{2})=\log(|Z_{0}|^{2}+...+|Z_{n}|^{2})$
on $\P^{n}$ which is of the form \ref{eq:analytic sin}.

\emph{Step 2: The LDP holds when $\phi$ has at most polynomial growth}

Given a weight function $\phi$ we will denote by $Z_{N}[\phi]$ the
corresponding partition function (we shall drop the subscript $\beta$
in the following). It will be enough (and, in fact, equivalent) to
prove that the following holds for all admissible weight functions:
\begin{equation}
-\lim_{N\rightarrow\infty}\beta^{-1}N^{-1}\log Z_{N}[\phi]=\inf_{\mathcal{M}_{1}(X)}F_{\phi}\label{eq:asympto for free energy in pf}
\end{equation}
Indeed, applying the formula to $\phi+u$ for any $u\in C(X)$ the
LDP then follows from the Gärtner-Ellis theorem (as explained in \cite[Section 4.2]{berm8}).
To this end let us first prove the following bound, only assuming
that $\phi$ is lsc on $\C^{n}$ and that the growth assumption \ref{eq:adm}
holds: 
\begin{equation}
\inf_{\mathcal{M}_{1}(X)}F_{\phi}\leq\liminf_{N\rightarrow\infty}-\beta^{-1}N^{-1}\log Z_{N}[\phi]\label{eq:lower bound for free with phi}
\end{equation}
 First observe that there exists a sequence $\phi_{j},$ satisfying
the assumptions in Step one and increasing to $\phi.$ Indeed, taking
$C_{\beta}:=\left(1+n/\beta+\epsilon\right)$ we have that $\phi(z)=C_{\beta}\phi_{FS}+u,$
where $u$ is a lsc continuous function on $\C^{n}$ which is bounded
from below. Hence, $u$ extends to a lsc function on $\P^{n}.$ This
means that there exists a sequence of functions $u_{j}\in C(\P^{n})$
such that $u_{j}$ increases to $u$ on $\P^{n}$ and we can thus
take $\phi_{j}:=C_{\beta}\phi_{FS}+u_{j}.$ Now, since, $\phi_{j}\leq\phi$
we have, for any $j$ and $N,$ 
\[
-\beta^{-1}N^{-1}\log Z_{N}[\phi_{j}]\leq-\beta^{-1}N^{-1}\log Z_{N}[\phi]
\]

Using the LDP in Step one thus gives 
\[
\inf_{\mathcal{M}_{1}(X)}F_{\phi_{j}}\leq\liminf_{N\rightarrow\infty}-\beta^{-1}N^{-1}\log Z_{N}[\phi]
\]

Finally, letting $j\rightarrow\infty$ proves the lower bound \ref{eq:lower bound for free with phi}. 

To prove the upper bound we will also use the assumption that $\phi$
has at most polynomial growth. Fix $R>0$ and set
\[
\phi_{R}:=\min\{\phi,R\phi_{FS}\},\,\,\,K_{R}:=\{\phi=\phi_{R}\}
\]
 We may assume that the increasing compact subsets $K_{R}$ exhaust
$\C^{n}$ as $R\rightarrow\infty$ (otherwise we are in the realm
of the previous step). Denote by $\mu_{R}^{(N)}$ the Gibbs measure
with Hamiltonian $H_{\phi_{R}}^{(N)}.$ By Jensen's inequality 
\begin{equation}
-\beta^{-1}N^{-1}\log Z_{N}[\phi]\leq-\beta^{-1}N^{-1}\log Z_{N}[\phi_{R}]+N^{-1}\sum_{i=1}^{N}\int(\phi-\phi_{R})(z_{i})(\mu_{R}^{(N)})\label{eq:decomp of F R}
\end{equation}
Next, we note that, by permutation symmetry,
\[
N^{-1}\sum_{i=1}^{N}\int(\phi-\phi_{R})(z_{i})(\mu_{R}^{(N)})=\int_{\C^{n}}(\phi-\phi_{R})(\mu_{R}^{(N)})_{1},
\]
 where $(\mu_{R}^{(N)})_{1}$ denotes the one-point correlation measure
of $\mu_{R}^{(N)}$ (see Remark \ref{rem:The-point-correlation}).
Thus, in order to prove the upper bound in \ref{eq:asympto for free energy in pf},
it will now be enough to prove 
\begin{equation}
\limsup_{R\rightarrow\infty}\limsup_{N\rightarrow\infty}\int_{\C^{n}}(\phi-\phi_{R})(\mu_{R}^{(N)})_{1}=0\label{eq:limsup R}
\end{equation}
To this end it is, in turn, enough to prove the following

\begin{equation}
\text{claim:\,}(\mu_{R}^{(N)})_{1}\leq Ce^{-\phi_{R-1}}dV_{s}\label{eq:ineq for one point cor with R}
\end{equation}
 for a constant $C,$ independent of $R$ and $N.$ Indeed, then the
integral in formula \ref{eq:limsup R} may be estimated from above
as follows, for $R$ sufficiently large,
\begin{equation}
\int_{K_{R}}\phi(\mu_{R}^{(N)})_{1}\leq C\int_{K_{R}}\phi e^{-\phi_{R-1}}dV_{s}\leq\epsilon_{R},\label{eq:ineq in terms of epsilon R}
\end{equation}
 where $\epsilon_{R}$ is independent of $N$ and tends to zero as
$R\rightarrow\infty.$ In the last inequality we used that $\phi$
is assumed to have at most polynomial growth and hence there exists
a function $g\in L^{1}(\C^{n},dV_{s})$ such that $\phi e^{-\phi_{R-1}}\leq g$
for any sufficiently large $R.$ 

All that remains is thus to prove the claim \ref{eq:ineq for one point cor with R}.
To this end first observe that, by definition,
\[
(\mu_{R}^{(N)})_{1}(z):=dV_{s}(z)\frac{1}{Z_{N}[\phi_{R}]}\int_{\C^{N-1}}e^{-\beta H_{\phi_{R}}^{(N)}(z,z_{2},...,z_{N})}dV_{s}(z_{2})\cdots dV_{s}(z_{N})
\]
Since $-H^{(N)}(\cdot,z_{2},...,z_{N})$ is in $\mathcal{L}(\C^{n})$
it follows that
\[
(\mu_{R}^{(N)})_{1}(z)\leq dV_{s}(z)\sup_{\psi\in\mathcal{L}(\C^{n})}\frac{e^{(\psi-\phi_{R})(z)}}{\int_{\C^{n}}e^{(\psi-\phi_{R})}dV_{s}}
\]
To proof of the claim \ref{eq:ineq for one point cor with R} is thus
reduced to the prove of the following
\begin{equation}
\text{claim: \,}f_{R}(z):=\sup_{\psi\in\mathcal{L}(\C^{n})}\frac{e^{(\psi-\phi_{FS}(z))}}{\int_{\C^{n}}e^{(\psi-\phi_{R})}dV_{s}}\leq C\label{eq:claim 2}
\end{equation}
 for some constant $C,$ independent of $R.$ To prove the latter
claim first rewrite 
\[
\frac{e^{(\psi-\phi_{FS}(z))}}{\int_{\C^{n}}e^{(\psi-\phi_{R})}dV_{s}}=\frac{e^{\varphi(z)}}{\int_{\C^{n}}e^{\varphi}\nu_{R}},\,\,\nu_{R}:=e^{\phi_{FS}-\phi_{R}}dV_{s}
\]
 Since $\nu_{R}$ decreases to $\nu:=e^{\phi_{FS}-\phi}dV_{s}$ as
$R\rightarrow\infty$ we get, by identifying$\nu$ with a measure
on $\P^{n},$ 
\[
f_{R}(z)\leq\sup_{PSH(\P^{n},\omega_{FS})}\frac{e^{\varphi(z)}}{\int_{\P^{n}}e^{\varphi}\nu}=\sup_{PSH(\P^{n},\omega_{FS})_{0}}\frac{1}{\int_{\P^{n}}e^{\varphi}\nu},
\]
 where $PSH(\P^{n},\omega_{FS})_{0}$ denotes the space of all $\omega_{FS}-$psh
$\varphi$ such that $\sup_{\P^{n}}\varphi=0.$ By Prop \ref{prop:omega psh compact appr}
the latter space is compact in $L^{1}(\P^{n}).$ The proof of the
claim \ref{eq:claim 2} is thus concluded by noting that the functional
$\int_{\P^{n}}e^{\cdot}\nu$ is continuous wrt the $L^{1}-$topology
on $PSH(\P^{n},\omega_{FS})$ (indeed, if $\varphi_{j}\rightarrow\varphi$
in $L^{1}$ then $\varphi_{j}\leq C$ and, after perhaps passing to
a subsequence, $\varphi_{j}\rightarrow\varphi$ a.e. so that the dominated
convergence theorem can be invoked).

\emph{Step 3: The case of iterated exponential growth}

First assume that $\phi$ has exponential growth. Take $\phi_{R}$
increasing to $\phi,$ as $R\rightarrow\infty$ and such that $\phi_{R}$
is continuous and has polynomial growth. By Step 2, the asymptotics
\ref{eq:asympto for free energy in pf} hold for the weight functions
$\phi_{R}.$ Now, since, $\phi$ has exponential growth there exists
a function $g\in L^{1}(\C^{n},dV_{s})$ such that $\phi e^{-\phi_{R}}\leq g$
for any sufficiently large $R.$ Hence, we can repeat the argument
in the proof of Step 2 to conclude that the LDP also holds when $\phi$
has exponential growth. Iterating this argument conclude the proof
of Step 3 and hence of the LDP on $\P^{n}.$
\end{proof}
\begin{rem}
The LDP on $\mathcal{M}_{1}(\P^{n})$ implies, in fact, that the LDP
also holds on $\mathcal{M}_{1}(\C^{n})$ (defined as the topological
dual of the space of bounded functions on $\C^{n})$ with a good rate
functional obtained from the embedding of $\C^{n}$ into $\P^{n}$
(thus coinciding with $F_{\beta,\phi}).$ This follows from general
principles (the contraction principle), only using that $\C^{n}=X-A$
where $A$ is pluripolar.
\end{rem}

\section{\label{sec:Proofs-for-the limit beta}Proofs for the limits $\beta\rightarrow\infty$
and $\beta\rightarrow0$}

In this section the line bundle $L$ plays no role. Accordingly, we
will consider the more general ``transcendental'' setting of a compact
complex manifold $X$ endowed with a real closed $(1,1)-$ form $\omega,$
assumed cohomologous to a Kähler form on $X$ (see Section \ref{subsec:Transcendental-polarized-manifol}).

For any given measure $\mu_{0}$ on $X,$ not charging pluripolar
subsets, there exists \cite[Thm 3.2]{berm6} a unique $\varphi_{\beta}\in\mathcal{E}^{1}(X,\omega)$
such that
\begin{equation}
MA_{\omega}(\varphi_{\beta})=e^{\beta\varphi_{\beta}}\mu_{0}\label{eq:MA eq with beta in section proof beta}
\end{equation}
when $\beta>0$ (however, when $\beta=0$ one needs to assume that
$E_{\omega}(\mu_{0})<\infty;$ see Theorem \ref{thm:var sol of ma}). 

Following \cite{bbgz} we recall that a sequence $\varphi_{j}\in PSH(X,\omega)$
is said to \emph{converge to} $\varphi_{\infty}$ \emph{in energy,}
if the convergence holds in $L^{1}(X)$ and $\mathcal{E}(\varphi_{j})$
converges to $\mathcal{E}(\varphi),$ which moreover is assumed finite.

\subsection{The limit $\beta\rightarrow\infty$}

We start by considering the ``zero temperature-limit'' where $\beta\rightarrow\infty.$
Assume that $\mu_{0}$ is supported on a compact subset $K$ of $X.$
We will say that that $\mu_{0}$ is\emph{ determining wrt $(K,\omega)$}
if 
\[
\varphi\leq0\,\,\text{almost everywhere wrt }\mu_{0}\implies\varphi\leq0\,\,\text{on\,\ensuremath{K}}
\]
\begin{thm}
\label{thm:zero temp in sing case polar}Let $X$ be a compact complex
manifold and $\mu_{0}$ a measure on $X$ with support $K$ such that
$\mu_{0}$ does not charge pluripolar subsets and which is determining
for $(K,\omega).$ Then $\varphi_{\beta}$ converges to $\varphi_{(K,\omega)}$
in energy, as $\beta\rightarrow\infty.$ If moreover, $\mu_{0}$ is
determining for $(K,\omega_{u}),$ for any $u\in C(X),$ then 
\[
F_{\beta}:=E_{\omega}+\beta^{-1}D_{\mu_{0}}
\]
 Gamma-converges towards $E_{\omega},$ as $\beta\rightarrow\infty.$ 
\end{thm}

\begin{proof}
The first part is proved essentially as in \cite{berm11}, where it
is assumed that $K=X.$ The proof may be summarized as follows. First,
by a concavity argument, $\varphi_{\beta}$ is a minimizer of the
functional
\begin{equation}
\mathcal{J}_{\beta}(\varphi):=-\mathcal{E}(\varphi)+\mathcal{L}_{\beta}(\varphi),\,\,\,\mathcal{L}_{\beta}(\varphi):=\beta^{-1}\log\int e^{\beta\varphi}\mu_{0}\label{eq:def of J beta}
\end{equation}
on $\mathcal{E}^{1}(X,\omega).$ In fact, the existence of the solution
$\varphi_{\beta}$ can be shown by minimizing $\mathcal{J}_{\beta}$
\cite[Thm 3.2]{berm6}. Now, by the determining property and \cite[Thm 1.14]{b-b-w}
there exists a constant $C$ such that
\begin{equation}
\sup_{K}\varphi-C/\beta-\epsilon\leq\mathcal{L}_{\beta}(\varphi)\leq\sup_{K}\varphi+C/\beta\label{eq:L beta in terms of sup}
\end{equation}
Moreover, the functional $\sup_{K}(\cdot)$ is continuous on $PSH(X,\omega)$
\cite[Cor 1.16]{b-b-w}. Since, 
\[
\mathcal{L}_{\beta}(\varphi_{\beta})=0
\]
a compactness argument allows one to conclude that, after perhaps
passing to a subsequence, $\varphi_{\beta}$ converges in $L^{1}$
to a minimizer of 
\[
\mathcal{J}_{\infty}(\varphi):=-\mathcal{E}(\varphi)+\sup_{K}\varphi
\]
on the subspace $\mathcal{K}$ of $PSH(X,\omega)$ consisting of all
$\varphi$ such that $\sup_{K}\varphi=0.$ Next, one shows that there
exists a unique minimizer, namely $\varphi=P_{(K,\omega)}0$ (using
that $\mathcal{E}$ is strictly increasing on the space of finite
energy functions in $PSH(X,\omega))$$.$ Hence, $\varphi_{\beta}\rightarrow P_{(K,\omega)}0$
in $L^{1}$ and 
\begin{equation}
\inf_{PSH(X,\omega)}\mathcal{J}_{\beta}\rightarrow-\mathcal{E}(P_{K}0)\label{eq:conv of inf beta as beta infty}
\end{equation}
 Using \ref{eq:L beta in terms of sup} this implies that $\mathcal{E}(\varphi_{\beta})\rightarrow\mathcal{E}(P_{(K,\omega)}0),$
i.e the convergence holds in energy, as desired.

To prove the Gamma-convergence of $F_{\beta}$ we first observe that
\begin{equation}
-\inf_{PSH(X,\omega)}\mathcal{J}_{\beta}=\inf_{\mathcal{M}_{1}(K)}E_{\omega}+\beta^{-1}D_{\mu_{0}}\label{eq:inf J beta as another inf}
\end{equation}
This can be checked explicitly, but it also follows from general duality
considerations (see \cite{berm6}).

Now, replacing $\omega$ with $\omega_{u}$ we deduce from \ref{eq:conv of inf beta as beta infty}
that 
\begin{equation}
\inf_{PSH(X,\omega)}\mathcal{J}_{\beta,u}\rightarrow-\mathcal{E}(P_{K}u),\label{eq:conv of inf J beta u as beta inft}
\end{equation}
 where $\mathcal{J}_{\beta,u}$ is defined by replacing $\mu_{0}$
in the definition of $\mathcal{L}_{\beta}$ with $e^{-\beta u}\mu_{0}.$
Moreover, by \ref{eq:inf J beta as another inf}
\begin{equation}
-\inf_{PSH(X,\omega)}\mathcal{J}_{\beta,u}=\inf_{\mathcal{M}_{1}(K)}\left(E_{\beta}(\mu)+\beta^{-1}D_{\mu_{0}}(\mu)+\int u\mu\right)\label{eq:inf J beta u as another inf}
\end{equation}
(using that $D_{e^{-\beta u}\mu_{0}}=D_{\mu_{0}}+\beta\int u\mu).$
Now, combining \ref{eq:inf J beta u as another inf} and \ref{eq:conv of inf J beta u as beta inft}
reveals that the Legendre-Fenchel transform on $C(K)$ of $F_{\beta}$
converges towards the functional $-\mathcal{E}\circ P_{(K,\omega}(-\cdot),$
which is Gateaux differentiable \cite[Thm B]{b-b} (see the proof
of Theorem \ref{thm:Gamma conv polarized}). Hence, the Gamma-convergence
in question follows from Prop \ref{prop:crit for gamma conv}, also
using that $E_{\omega}$ is the Legendre-Fenchel transform of $-\mathcal{E}\circ P_{(K,\omega)}(-\cdot)$
(just as in the proof of Theorem \ref{thm:Gamma conv polarized}).
\end{proof}
The previous theorem immediately implies Theorem \ref{thm:zero temp limit singular case intro}
in $\C^{n},$ using the identifications in Section \ref{subsec:Compactification-of-}.

\subsection{The limit $\beta\rightarrow0$}

Next, we will consider the ``infinite temperature limit'' when $\beta\rightarrow0.$
\begin{thm}
Let $X$ be a compact complex manifold and $\mu_{0}$ a probability
measure on $X$ such that $E_{\omega}(\mu)<\infty.$ When $\beta\rightarrow0$
the functions $\varphi_{\beta}$ converge to $\varphi_{0}$ in energy,
where $\varphi_{0}$ is the unique solution to 
\[
MA_{\omega}(\varphi)=\mu_{0}
\]
 satisfying the normalization condition 
\begin{equation}
\int_{X}\varphi\mu_{0}=0\label{eq:normalization condition in them text}
\end{equation}
More generally, given a lsc function $u_{0}$ on $X$ we consider
the equation \ref{eq:MA eq with beta in section proof beta} with
$\mu_{0}$ replaced by $e^{-\beta u_{0}}\mu_{0}.$ Then the result
above still holds under the assumption that 
\[
E_{\omega,u_{0}}(\mu):=E_{\omega}(\mu)+\int u_{0}\mu<\infty,
\]
if $\varphi$ is replaced by $\varphi-u_{0}$ in the normalization
condition \ref{eq:normalization condition in them text}.
\end{thm}

\begin{proof}
The function $\varphi_{\beta}$ is the unique minimizer of the functional
$\mathcal{J}_{\beta}$ appearing in the beginning of the proof of
Theorem \ref{thm:zero temp in sing case polar}, subject to the normalization
condition 
\begin{equation}
\mathcal{L}_{\beta}(u)=0\label{eq:normal cond in proof thm}
\end{equation}
Setting $\mathcal{L}_{0}(u):=\int u\mu_{0}$ we thus need to prove
that $u_{\beta}$ converges towards the unique minimizer $u_{0}$
of $\mathcal{J}_{0},$ subject to the normalization condition \ref{eq:normal cond in proof thm}. 

\emph{Step 1: After perhaps passing to a subsequence $\varphi_{\beta}$
converges to $\varphi_{0}$ where $\varphi_{0}$ minimizes $\mathcal{J}_{0}$
and moreover
\[
\mathcal{J}_{\beta}(\varphi_{\beta})\rightarrow-E_{\omega}(\mu_{0})
\]
}

To prove this first observe that $\mathcal{L}_{0}(u)\leq\mathcal{L}_{\beta}(u)$
for any $\beta>0$ and function $u.$ Hence, for any given $\varphi\in PSH(X,\omega)\cap C(X)$
we have, using the very definition of $E_{\omega}(\mu_{0}),$ that
\begin{equation}
-E_{\omega}(\mu_{0})\leq\mathcal{J}_{0}(\varphi_{\beta})\leq\mathcal{J}_{\beta}(\varphi_{\beta})\leq\mathcal{J}_{\beta}(\varphi)\leq\mathcal{J}_{0}(\varphi)+C_{\varphi}\beta,\label{eq:ineq with J}
\end{equation}
In particular, using the normalization condition \ref{eq:normal cond in proof thm}
we deduce that 
\begin{equation}
-E_{\omega}(\mu_{0})\leq-\mathcal{E}(\varphi_{\beta})\leq C\label{eq:uniform control on energy}
\end{equation}

Next, we recall the following coercivity estimate \cite{bbgz}: for
a fixed volume form $dV$ on $X$ there exists positive constant $C$
such that
\begin{equation}
\mathcal{J}_{0}(\varphi)\geq C\left(-\mathcal{E}(\varphi)+\int\varphi dV\right)-C\label{eq:coerciv}
\end{equation}
on $PSH(X,\omega).$ Combined with the inequality \ref{eq:uniform control on energy}
this gives

\[
\int\varphi_{\beta}dV\leq C'
\]
By the $L^{1}-$compactness of $PSH(X,\omega)_{0}$ (Prop \ref{prop:omega psh compact appr})
we deduce that, after perhaps passing to a subsequence, $\varphi_{\beta}\rightarrow\varphi_{0}$
in $L^{1}(X).$ Since $\mathcal{J}_{0}$ is lsc the inequalities \ref{eq:ineq with J}
thus imply that $\varphi_{0}$ minimizes $\mathcal{J}_{0},$ as desired.
Moreover, it also follows from the inequalities \ref{eq:ineq with J}
that $\mathcal{J}_{\beta}(\varphi_{\beta})\rightarrow\mathcal{J}_{0}(\varphi_{0})=-E_{\omega}(\mu_{0}).$
Hence, by the normalization condition \ref{eq:normal cond in proof thm}
$\mathcal{E}(\varphi_{\beta})\rightarrow\mathcal{E}(\varphi_{0}),$
which concludes the proof of Step 1. 

\emph{Step 2: $\int\varphi_{\beta}\mu_{0}\rightarrow0$}

First observe that, when $\beta>0$ and $\mu_{\beta}:=MA_{\omega}(\varphi_{\beta})$
\[
-\mathcal{J}_{\beta}(\varphi_{\beta})=F_{\beta}(\mu_{\beta}):=E_{\omega}(\mu_{\beta})+\beta^{-1}D_{\mu_{0}}(\mu_{\beta}),
\]
 as follows from \ref{eq:inf J beta as another inf}. By the previous
step the lhs above converges towards $E_{\omega}(\mu_{0})$ and so
does $E_{\omega}(\mu_{\beta})$ (since $\varphi_{\beta}$ converges
towards $\varphi_{0}$ in energy). Hence,
\[
\beta^{-1}D_{\mu_{0}}(\mu_{\beta})\rightarrow0.
\]
 But, since $\mu_{\beta}=e^{\beta\varphi_{\beta}}\mu_{0}$ we have
$\beta^{-1}D_{\mu_{0}}(\mu_{\beta})=\int\varphi_{\beta}\mu_{\beta},$
which the concludes the proof of Step 2.

Now, by Step 1, $\varphi_{\beta_{j}}\rightarrow\varphi_{0}$ where
$\varphi_{0}$ minimizes $\mathcal{J}_{0}$ and hence solves the equation
\ref{eq:potential of meas} (by Theorem \ref{thm:var sol of ma}).
Next, by Step 2 and the fact that $\mathcal{L}_{0}$ is continuous
wrt convergence in energy \cite{bbgz}, we have $\mathcal{L}_{0}(\varphi_{0})=0.$
Hence, by the uniqueness in Theorem \ref{thm:var sol of ma} the whole
family $\varphi_{\beta}$ converges to $\varphi_{0},$ which has the
required property.

This concludes the proof of the theorem in the case when $u_{0}=0.$
In order to prove the general case we fix a family $u^{R}\in C(X)$
such that $u^{R}$ increases to $u_{0}$ as $R\rightarrow\infty.$
The existence of $\varphi_{\beta}$ follows from the previous case,
by setting

\[
\mathcal{L}_{\beta}(u):=\beta^{-1}\log\int e^{\beta(u-u_{0})}\mu_{0}
\]
and using that the measure $e^{-\beta u_{0}}\mu_{0}$ has finite energy
(since it is dominated by $C_{0}\mu_{0}).$ Moreover, the inequalities
\ref{eq:ineq with J} still hold if $E_{\omega}(\mu_{0})$ is replaced
by $E_{\omega,u_{0}}(\mu)$ and the last inequality is replaced by
\[
\mathcal{J}_{\beta,u_{0}}(\varphi)\leq-\mathcal{E}(\varphi)+\int(\varphi-u^{R})\mu_{0}+C_{u,R}\beta
\]
(only using that $u\geq u^{R}).$ We note that in the limit when first
$\beta\rightarrow0$ and then $R\rightarrow\infty$ the right hand
side above tends to $\mathcal{J}_{0,u_{0}}(\varphi).$ We can then
conclude essentially as before.
\end{proof}
\begin{cor}
Let $\mu_{0}$ be a probability measure on $\C^{n}$and $\phi$ a
continuous function on $\C^{n}$ with super logarithmic growth. Assume
that $E_{\phi}(\mu_{0})<\infty.$ Given $\beta>0$ there exists a
unique $\psi_{\beta}\in\mathcal{E}^{1}(\C^{n})$ such that
\[
MA(\psi_{\beta})=e^{\beta(\psi_{\beta}-\phi)}\mu_{0}
\]
Moreover, when $\beta\rightarrow0$ the functions $\psi_{\beta}$
converge to $\psi_{0}$ in energy, where $\psi_{0}$ is the unique
solution to 
\[
MA(\psi)=\mu_{0}
\]
 satisfying the normalization condition 
\[
\int_{X}(\phi-\psi_{0})\mu_{0}=0
\]
\end{cor}

\begin{proof}
Identifying $\mu_{0}$ with a probability measure on $\P^{n}$ we
get $E_{\omega}(\mu_{0})<\infty,$ when $\omega$ is taken, fore example,
as the curvature form of the Fubini-Study metric on $\mathcal{O}(1)\rightarrow\P^{n}.$
This is shown precisely as in the proof of Lemma \ref{lem:non-weighted energ}.
The case when $\phi=\phi_{FS}$ then follows directly from the previous
theorem with $u_{0}=0.$ In the general case we decompose $\phi=u_{0}+\phi_{FS},$
where, by assumption $u_{0}\geq-C$ on $\C^{n}.$ Hence, we can identify
$u_{0}$ with a lsc function on $\P^{n}$ and apply the general form
of the previous theorem.
\end{proof}

\section{\label{sec:Outlook-and-open}Outlook and open problems}

\subsection{More general background measures and the case of compact domains
in $\C^{n}$}

A natural question is whether the volume form $dV$ appearing in Theorem
\ref{thm:conv in law in C^n intro} can be replaced by a more singular
measure on $\C^{n}.$ And what are the conditions that need to be
imposed on $dV?$ The following conjectural answer is motivated by
the fact that the existence of a finite energy solution to the equation
\ref{eq:ma eq with beta intro} only requires that $dV$ does not
charge pluripolar subsets.
\begin{conjecture}
\label{conj:conv in law when nonpluripolar}The result in Theorem
\ref{thm:conv in law in C^n intro} holds more generally if the assumption
that $dV$ be a volume form is replaced by the assumption that $dV$
does not charge pluripolar subsets.
\end{conjecture}

The conjecture does hold when $n=1,$ as follows from \cite[Thm 1.1, Cor 1.2]{berm10}.
Strictly speaking, in \cite[Thm 1.1]{berm10} it was assumed that
$dV$ has compact support and finite energy, but the assumption on
compact support may be removed by a truncation argument. Moreover,
the assumption that $dV$ has finite energy was only imposed in \cite[Thm 1.1]{berm10}
to ensure that $F_{\beta}$ is bounded from below and hence admits
a minimizer $\mu_{\beta}.$ But this holds more generally if $dV$
does not charge pluripolar subsets, as explained in the introduction
of Section \ref{sec:Proofs-for-the limit beta}, by taking $\mu_{\beta}:=MA_{\omega}(\varphi_{\beta})$. 

A natural case where the assumption in Conjecture \ref{conj:conv in law when nonpluripolar}
is satisfied is the case when 
\[
dV=1_{\Omega}d\lambda
\]
 for $\Omega$ a bounded open domain in $\C^{n}.$ In this case the
corresponding current $\omega^{\psi_{\beta}}$ is a weak solution
to the Kähler-Einstein equation 
\[
\mbox{Ric\,}\ensuremath{\omega=-\beta\omega}
\]
 in $\Omega.$ As a complement to Conjecture rrer it seems natural
to ask for conditions on $\Omega$ ensuring that $\psi_{\beta}$ be
smooth?
\begin{conjecture}
\label{conj:reg}Assume that $\Omega$ is an open bounded domain in
$\C^{n}$ with smooth and strictly pseudoconvex boundary. Then, if
$\beta\in[0,\infty[$ the solution $\psi_{\beta}\in\mathcal{L}(\C^{n})$
of the equation \ref{eq:ma eq with beta intro}, with $dV=1_{\Omega}d\lambda,$
is smooth in $\Omega.$ 
\end{conjecture}

\begin{rem}
This regularity problem is, in general, different than the corresponding
Dirichlet problem for the equation \ref{eq:ma eq with beta intro}
in $\Omega,$ when $n>1.$ The latter problem is known to have a smooth
solution under the assumptions above \cite{c-k-n-s}. However, in
the case when $\Omega$ is a ball it is not hard to see that the two
solutions coincide (up to an additive constant) and that $\omega_{\beta}$
is the restriction to $\Omega$ of a scaled hyperbolic (Poincaré)
metric when $\beta>0.$
\end{rem}

The assumption that $\Omega$ is pseudoconvex may be superfluous.
In fact, in the case when $\Omega$ is invariant under the standard
action of the unit-torus Conjecture \ref{conj:reg} does hold, without
assuming that $\Omega$ is pseudoconvex. To see this first assume
that $\beta=0$ and that $\Omega\subset\C^{*n}.$ Denoting by $\text{Log}$
the projection from $\C^{*n}$ onto $\R^{n}$ in formula \ref{eq:Log map intro}
and setting $\mu:=(\text{Log )}_{*}1_{\Omega}d\lambda$ and $D:=\text{Log \ensuremath{(\Omega)}}$
one can check that $\psi\in\mathcal{L}(\C^{n})$ solves the equation
\ref{eq:ma eq with beta intro} and is in $C^{\infty}(\Omega)$ iff
$\psi=\text{Log}^{*}\phi,$ where $\phi$ is strictly convex and
\[
(\nabla\phi)_{*}\mu=\nu,\,\,\,\nu:=1_{P}d\lambda/n!,
\]
 where $P$ is the standard simplex in $\R^{n}$ (this follows from
combining \cite[Lemma 2.4]{berm12} and \cite[Lemma 3.12]{berm12}).
This equivalently means that the gradient map $T:=\nabla\phi$ is
the optimal map transporting $\mu$ to $\nu$ in the sense of the
theory of Optimal Transport (see Section \ref{subsec:Optimal-transport}).
By the regularity results for the Optimal Transport problem in \cite{ca0}
it follows, using that the target $P$ is convex, that $\phi$ is
in $C^{\infty}(D).$ Hence, $\psi\in C^{\infty}(\Omega),$ as desired.
The case when $\Omega$ is not contained in $\C^{*n}$ can be handled
by a change of variables argument. Moreover, the case $\beta>0$ can
also be reduced to well-known regularity results in the real setting,
using similar arguments. 

\subsection{\label{subsec:A-Green's-formula}A Green's formula for the Monge-Ampère
equation $(\beta=0)$}

Assume given a normalized volume form $dV$ on $\C^{n}$ satisfying
the integrability property \ref{eq:integrability condition}. Then
the unique finite energy solution $\psi\in\mathcal{E}^{1}(\C^{n})$
to the complex Monge-Ampère equation 
\begin{equation}
MA(\psi)=dV,\label{eq:ma eq in outlook}
\end{equation}
satisfying the normalization condition \ref{eq:norm cond in thm zero},
for a given weight function $\phi,$ may be obtained as the double
limit $\psi:=\lim_{\beta\rightarrow\infty}\lim_{k\rightarrow\infty}\psi_{\beta}^{(k)},$
where 
\[
\psi_{\beta}^{(k)}:=\frac{1}{\beta}\log\frac{\int_{(\C^{n})^{N_{k}-1}}\left|D^{(N_{k})}(\cdot,z_{2},...z_{N_{k}})\right|^{2\beta/k}(e^{-\beta\phi}dV)^{\otimes(N_{k}-1)}}{dV}-\log Z_{N_{k}}
\]
(by combining Corollary \ref{cor:conv of one pt correl measr etc intro}
with Theorem\ref{thm:zero temp limit singular case intro}). Formally
interchanging the two limits thus suggests the following 
\begin{conjecture}
\label{conj:c-y eq}Let $dV$ be a probability measure on $\C^{n}$
as above. Then there exists a sequence of constants $C_{k}$ such
that
\begin{equation}
\psi^{(k)}:=\frac{1}{k}\frac{\int_{(\C^{n})^{N_{k}-1}}\log\left|D^{(N)}(\cdot,z_{2},...z_{N_{k}})\right|^{2}dV{}^{\otimes(N_{k}-1)}}{dV}-C_{k}\label{eq:form of psi k in conj green}
\end{equation}
converges in $L_{loc}^{1}$ to a finite energy solution $\psi\in\mathcal{E}^{1}(\C^{n})$
to the complex Monge-Ampère equation \ref{eq:ma eq in outlook}.
\end{conjecture}

This conjecture can be seen as a non-linear generalization of the
classical Green's formula for the solution of the Laplace equation
in $\C.$ Indeed, when $n=1$ formula \ref{eq:form of psi k in conj green}
for $\psi^{(k)}$ coincides with the Green's formula for the solution
of corresponding Laplace equation (for an appropriate choice of constants
$C_{k}),$ as follows immediately from formula \ref{eq:E and D in terms of g intro}. 

It turns out that the validity of the conjecture above would follow
from the following generalization of the Green's formula \ref{eq:greens formul for E in plane}
for the classical energy in $\C:$
\begin{conjecture}
Let $dV$ be a probability measure on $\C^{n}$ as above. Then the
non-weighted pluricomplex energy $E(dV)$ can be expressed as 
\[
E(dV)=\lim_{k\rightarrow\infty}\frac{1}{kN_{k}}\int_{(\C^{n})^{N_{k}}}\log\left|D^{(N)}(z_{1},z_{2},...z_{N_{k}})\right|^{2}dV{}^{\otimes N_{k}}
\]
\end{conjecture}

This amounts to the convergence of the mean energies \ref{eq:conv of mean energies}
in the present setting. As discussed in \cite[Section 4.3]{berm8},
the conjecture would follow from the conjectural LDP in the case $\beta<0,$
discussed in Section \ref{subsec:The-case-of-neg}.

\subsection{\label{subsec:Dynamics}Dynamics }

We start with a general discussion, where $H^{(N)}$ should be thought
of as a generic function on a space $X^{(N)},$ before coming back
to the present setting, where the role of $H^{(N)}$ is played by
the Hamiltonian \ref{eq:def of H N weighted intro} and $N$ is the
number of points (``particles'') in $\C^{n}.$ The standard ``first
order'' approach for finding minimizers to a given function $H^{(N)}$
on a space $X^{(N)}$ is to follow the downward gradient flow of $H^{(N)}:$
\begin{equation}
\frac{d\boldsymbol{x}(t)}{dt}=-\nabla H^{(N)}(\boldsymbol{x}(t))\label{eq:gradient flow}
\end{equation}
 from appropriate initial data $\boldsymbol{x}(0)$ (assuming that
$H^{(N)}$ has been endowed with a Riemannian metric). However, for
numerical applications this approach is often infeasible, as the flow
will typically get trapped in local minimal of $H^{(N)}.$ Accordingly,
a common approach used in numerical computations, is to add a noise
term to the equations: 
\begin{equation}
d\boldsymbol{x}(t)=-\nabla H^{(N)}(\boldsymbol{x}(t))dt+\epsilon_{N}d\boldsymbol{B}(t),\label{eq:sde}
\end{equation}
 where $d\boldsymbol{B}(t)$ denotes Brownian motion on $X^{(N)}$
and $\epsilon_{N}$ is a parameter measuring the strength of the noise
term. From an analytic point of view the Stochastic Differential Equation
(SDE) above, may be described by a linear parabolic Partial Differential
Equation (PDE). Indeed, denoting by the law of the random variable
$\boldsymbol{x}(t),$ its density $\rho_{t}^{(N)}$ evolves according
to the\emph{ forward Kolmogorov equation}
\begin{equation}
\frac{\partial\rho_{t}^{(N)}}{\partial t}=\frac{1}{\beta_{N}}\Delta\rho_{t}^{(N)}+\nabla\cdot(\rho_{t}^{(N)}\nabla H^{(N)}),\label{eq:forward kolm eq}
\end{equation}
where $\epsilon_{N}=\sqrt{\frac{2}{\beta_{N}}}.$ This formulation
reveals that the stationary measure of the evolution is nothing but
the Gibbs measure \ref{eq:def of Gibbs measure intro}, corresponding
to the inverse temperature $\beta_{N}$ (when it is well-defined). 
\begin{rem}
In the statistical mechanical literature the SDE \ref{eq:sde} is
called the (overdamped) \emph{Langevin equation, }describing the microscopic
relaxation towards equilibrium (and the PDE is called the corresponding
point \emph{Fokker-Planck equation}). In the statistics such evolution
equations are extensively used (upon discretization) in the framework
of \emph{Markov Chain Monte Carlo (MCMC)} methods for sampling a given
probability distribution $\rho^{(N)}dV$ on a space $X^{(N)}.$ The
corresponding function $H^{(N)}:=-\log\rho^{(N)}$ is then often referred
to as ``self-information'' or ``surprisal'' (setting $\beta_{N}=1).$ 
\end{rem}

The previous discussion leads one to ask whether there is a dynamical
analog of Theorem \ref{thm:conv in law in C^n intro}? 
\begin{conjecture}
\label{conj:dynamic}Consider $\C^{n}$ endowed with its Euclidean
metric and $\phi$ a smooth function on $\C^{n}.$ Denote by $H_{\phi}^{(N_{k})}$
the corresponding Hamiltonian on $(\C^{n})^{N_{k}},$ defined by formula
\ref{eq:def of H N weighted intro}. If the initial data $x_{i}(0)$
for the SDE \ref{eq:sde} consists of independent and identically
distributed random vectors with law $\mu_{0}\in\mathcal{M}_{1}(\C^{n})$
(which equivalently, means that $\mu^{(N)}(0)=\mu_{0}^{\otimes N}$
in the PDE \ref{eq:forward kolm eq}), then, at any fixed later time
$t,$ the corresponding empirical measures 
\begin{equation}
\delta_{N}(t):=N^{-1}\sum_{i=1}^{N}\delta_{x_{i}(t)}\label{eq:empirical measure with time}
\end{equation}
converge in law, as $N\rightarrow\infty,$ to a measure $\mu_{t}=\rho_{t}dV$
on $\C^{n},$ which is a weak solution of the following equation on
$\C^{n}\times]0,\infty[:$ 
\begin{equation}
\frac{\partial\rho_{t}}{\partial t}=\frac{1}{\beta}\Delta\rho_{t}-\nabla\cdot(\rho_{t}\nabla\psi_{t})\label{eq:evolut eq in conj}
\end{equation}
where $\psi_{t}(z)\in\mathcal{L}(\C^{n})$ is a potential of $\rho_{t}dV,$
i.e.
\[
MA(\psi_{t})=\rho_{t}dV.
\]
\end{conjecture}

Some remarks on the conjecture:
\begin{itemize}
\item The form of the limiting evolution equation \ref{eq:evolut eq in conj}
is suggested by a dynamical analog of the mean field heuristics in
Section \ref{subsec:Mean-field-approximations}, which suggest the
following general evolution equation:
\begin{equation}
\frac{\partial\rho_{t}}{\partial t}=\frac{1}{\beta}\Delta\rho_{t}+\nabla\cdot(\rho_{t}\nabla\varphi_{t}),\,\,\,\varphi_{t}=dE_{|\mu_{t}}\label{eq:otto gradient flow}
\end{equation}
(which, according to Remark \ref{rem:E on volume forms} yields the
equation \ref{eq:evolut eq in conj} in the present setting). In the
kinetic theory literature drift-diffusion equations of the form \ref{eq:otto gradient flow}
are usually called \emph{McKean-Vlasov equation}s \cite{mc,da-g}.
\item In the terminology of Kac, if a deterministic limit $\mu_{t}$ exists,
as $N\rightarrow\infty,$ then \emph{propagation of chaos }is said
to bold. 
\item Formally, Theorem \ref{thm:conv in law in C^n intro} corresponds
to the case $t=\infty,$ since the density $\rho_{\beta}$ of $\mu_{\beta}$
is a stationary solution to the evolution equation \ref{eq:evolut eq in conj}. 
\item The conjecture also naturally extends to the case of a deterministic
evolution, where $\beta_{N}=\infty$ (and hence also $\beta=\infty)$
with random initial data. Alternatively, a purely deterministic formulation
is obtained by requiring that $\delta_{N}(0)\rightarrow\mu_{0}$ weakly
as $N\rightarrow\infty.$
\item A real analog of the conjecture is obtained when $\C^{n}$ is replaced
by Euclidean $\R^{n}$ (and hence $\mu_{0}$ is supported in $\R^{n}).$
Then the role of $\psi_{t}$ in equation \ref{eq:evolut eq in conj}
is played by the restriction $\psi_{t}(x)$ to $\R^{n}\Subset\C^{n}$
of the potential $\psi_{t}(z)$ of $\mu_{t},$ viewed as a probability
measure on $\C^{n},$ supported on $\R^{n}.$ 
\item Note that the conjecture makes sense also when $\phi=0,$ although
the corresponding stationary measure (i.e. the Gibbs measure) is not
a well-defined probability measure then. Heuristically, this reflects
the repulsive nature of $H_{0}^{(N)}:$ since there is no confining
potential $\phi$ an initial cloud of particles will repel each other
and all escape to spatial infinity as $t\rightarrow\infty.$
\end{itemize}
Unfortunately, the dynamic setting of the previous conjecture is technically
considerably harder than the situation in Theorem \ref{thm:conv in law in C^n intro}.
Indeed, it is even non-trivial to make sense of the linear PDE \ref{eq:forward kolm eq},
sense $H_{\phi}^{(N)}$ is singular. Moreover, the existence and uniqueness
of weak solutions to the non-linear equation \ref{eq:evolut eq in conj}
appears to be rather challenging.

In fact, even if $H^{(N)}$ is smooth and uniformly continuous wrt
$N,$ as in in the``regular case'' described in Section \ref{subsec:Mean-field-approximations},
the corresponding result appears to be open, in general. However,
assuming a uniform bound on the Hessian of $H^{(N)}:$ 
\[
C^{-1}\cdot I\leq\nabla H^{(N)}\leq C\cdot I
\]
the corresponding convergence result does follow from \cite{da-g}
(even in the stronger sense of large deviations). 

In the singular ``linear'' case and with the deterministic evolution
\ref{eq:gradient flow} (i.e when $\beta_{N}=\infty)$ general convergence
results have very recently been obtained in \cite{s2} for repulsive
``power-laws'' in $\R^{D}$ of long-range:
\begin{equation}
g(x,y)=\frac{1}{\alpha}|x-y|^{-\alpha},\,\,\,\alpha\in[0,D+1[,\label{eq:power-law}
\end{equation}
 under the regularity assumption that $\mu_{t}$ admits an $L^{\infty}-$density
(when $\alpha<D-1).$ The results in \cite{s2} cover the deterministic
version of Conjecture \ref{conj:dynamic}, when $n=1$ (i.e. $D=2)$
with $\phi=0,$ corresponding to the logarithmic case $\alpha=0$
(when combined with the $L^{\infty}-$estimates on $\rho_{t}$ in
\cite{l-z,s-v}). Interestingly, this case appears in the theory of
supraconductivity \cite{c-r-s} (and superfluidity \cite{e}), where
the role of the points $x_{1},...,x_{N}$ in $\C$ is played by vortices
(formed in a type-II supra conductor, in the realistic limit of large
Ginzburg-Landau parameter). 

However, when turning on a noise ($\beta_{N}<\infty)$ there seem
to be very few general results in the singular linear setting, except
in the case of the real line $\R$ (discussed in Remark \ref{rem:log case on real line and dyson}
below). Recently, a strong well-posedness result was obtain for the
SDE \ref{eq:sde} in the logarithmic setting of $\C$ (when $\phi(z)=|z|^{2})$
with a quantitative convergence towards the corresponding Gibbs measure,
as $t\rightarrow\infty.$ But the conjectural convergence in the limit
when $N\rightarrow\infty$ still appears to be open even in the logarithmic
setting of $\C.$ 

\subsubsection{Wasserstein gradient flows}

A new approach for establishing the convergence as $N\rightarrow\infty$
for the empirical measures of the SDEs \ref{eq:sde} was introduced
in \cite{ber-o}. The starting point is the important observation
of Otto \cite{ot}, that the evolution equation \ref{eq:otto gradient flow}
is the gradient flow of the functional $E$ on the space of all smooth
probability measures on $X,$ endowed with Otto's Riemannian metric.
The latter metric coincides with Wasserstein metric, which defines
a metric space structure on the space $\mathcal{P}_{2}(X)$ of all
probability measures on $X$ with finite second moments. According
to the general theory of gradient flows on metric spaces in \cite{a-g-s},
unique gradient flow solutions in $\mathcal{P}_{2}(X)$ exist if $E$
is assumed to be $\lambda-$convex on the metric space $\mathcal{P}_{2}(\R^{D})$
Using the theory in \cite{a-g-s} it was shown in \cite{ber-o} that
if $H^{(N)}$ is assumed uniformly $\lambda-$convex on $(\R^{D})^{N},$
i.e. if the Hessian of $H^{(N)}$ is uniformly bounded from below,
\[
\lambda\cdot I\leq\nabla H^{(N)}
\]
 then the corresponding empirical measures $\delta_{N}(t)$ indeed
converge in law towards the gradient flow $\mu_{t},$ as $N\rightarrow\infty.$
As shown in \cite{ber-o} this result applies, in particular, to the
``tropicalization'' of the complex setting in Theorem \ref{thm:conv in law in C^n intro},
in the case of negative temperature (discussed in Section \ref{subsec:Tropicalization-at-negative}). 

However, due to lack of convexity, new ideas are needed in the original
complex setting of Conjecture \ref{conj:dynamic}. In the case of
the complex plane, or more precisely when $X$ is a bounded domain
in $\C,$ gradient flow solutions $\mu_{t}$ of the corresponding
logarithmic energy $E$ were constructed in \cite{a-s}, motivated
by application to superconducitivity (when $\beta=\infty).$ Uniqueness
of $L^{\infty}-$solutions was then established in \cite{ma} (for
$X$ convex). However, the convergence as $N\rightarrow\infty$ in
Conjecture \ref{conj:dynamic} remains to be established even in the
case of complex plane (or a bounded domain in $\C).$
\begin{rem}
\label{rem:log case on real line and dyson}The approach in \cite{ber-o}
also applies in the ``linear'' setting with a lower-law \ref{eq:power-law}
when $D=1.$ This covers the logarithmic pair interaction on the real
line (corresponding to the one-dimensional case of the the real analogue
discussed in the fifth point after the statement of Conjecture \ref{conj:dynamic}).
This case has previously been studied extensively when $\phi(x)=|x|^{2}$
in connection to Dyson's Brownian motion on the space of Hermitian
random matrices. 
\end{rem}

Interestingly, the ``attractive'' analog of Conjecture \ref{conj:dynamic},
i.e. the case when $H^{(N)}$ is replaced by $-H^{(N)},$ which appears
in the negative temperature setting discussed below, has been studied
extensively in the case of the complex plane. The corresponding evolution
equation \ref{eq:evolut eq in conj}, obtained by replacing $\psi_{t}$
with $-\psi_{t},$ is then called the \emph{Keller-Segal equation}
in chemotaxis, where the role of the points $x_{1},..,x_{N}$ is played
by bacteria. In this setting a weaker form of the Conjecture \ref{conj:dynamic}
was recently established in \cite{fou}, saying that the result holds
up to passing to a subsequence of $\mu^{(N)}.$

\subsection{\label{subsec:The-case-of-neg}The case of negative temperature $(\beta<0)$
and integrability indices}

Assume given a convex body $P$ in $\R^{n},$ containing $0$ in its
interior and denote by $\mathcal{P}_{kP}(\C^{*n})$ the corresponding
space of Laurent polynomials of $\C^{*n}$ (introduced in Section
\ref{subsec:Tropicalization-at-negative}). Fixing $\beta<0$ and
a weight function $\phi\in\mathcal{L}_{P}(\C^{*n})$ we denote by
$\mu_{\beta}^{(N_{k})}$ the corresponding probability measure on
$(\C^{*n})^{*N_{k}}$ (when it if well-defined): 

\[
\mu_{\beta}^{(N_{k})}=\frac{1}{Z_{N,\beta}}|D^{(N_{k})}(z_{1},...,z_{N_{k}})|_{k\phi}^{2\beta/k}dV^{\otimes N},\,\,\,dV:=e^{-\phi}dV_{\C^{*n}}
\]
It can be shown that $\mu_{\beta}^{(N_{k})}$ is a well-defined probability
measure, i.e. $Z_{N,\beta}<\infty$ if $\beta$ is sufficiently close
to $0.$ The following conjecture is motivated by its tropical analog,
Theorem \ref{thm:tropical}:
\begin{conjecture}
\label{conj:neg}Let $P$ be a convex body in $\R^{n}$ containing
$0$ in its interior. 
\begin{itemize}
\item For $k$ sufficiently large, the corresponding Gibbs measure $\mu_{\beta}^{(N_{k})}$
is a well-defined probability measure iff $\beta>-R_{P},$ where $R_{P}\in]0,1]$
is the invariant of $P$ in formula \ref{eq:inv R of conv bod}. 
\item If $\beta>-R_{P},$ then the random measure $\delta_{N_{k}}$ on $(\C^{*n})^{N_{k}},\mu_{\beta}^{(N_{k})})$
converges in law, as $N\rightarrow\infty,$ towards $\mu_{\beta}=e^{\beta\psi_{\beta}}dV,$
where $\psi_{\beta}$ is the unique solution in $\mathcal{L}_{P,+}(\C^{*n})$
to the equation \ref{eq:ma eq with beta intro} on $\C^{*n}.$
\end{itemize}
\end{conjecture}

As recalled in Section \ref{subsec:Tropicalization-at-negative},
the previous conjecture is motivated by Kähler-Einstein geometry on
toric varieties. Indeed, when $\beta=-1$ a solution $\psi_{\beta}$
of the corresponding equation \ref{eq:ma eq with beta intro}  is
the Kähler potential of a Kähler-Einstein metric $\omega$ on $\C^{*n}$,
which extends to the toric variety $X_{P},$ in the case when $P$
is rational polytope. 

The conjecture holds when $n=1,$ as follows from the stronger LDP
established in \cite{berm10} (applied to $X:=\P^{1}$). This is closely
related to \cite{clmp,k2}, which consider the case of bounded domains
in $\C,$ motivated by Onsager's statistical vortex model. 

Moreover, the necessary condition that $\beta>-R_{P}$ in the first
point of the conjecture follows from its tropical analog, i.e. the
first point in Theorem \ref{thm:tropical}. The question whether the
condition is sufficient can be formulated in terms of properties of
integrability indices of psh functions (known as log canonical thresholds
in the algebraic setting \cite{ko}). Briefly, if $\psi$ is in $\mathcal{L}_{P}(\C^{*m}),$
then its (global)\emph{ integrability index $c_{\psi}$ }is defined
as the positive number
\[
c_{\psi}:=\sup\left\{ c:\,\int_{\C^{*m}}e^{-c(\psi-\phi)}e^{-\phi}dV_{\C^{*n}}<\infty\right\} ,
\]
 for any fixed $\phi\in\mathcal{L}_{P,+}(\C^{*m}).$ When $\psi=k^{-1}\log|p_{k}|^{2}$
for $p\in\mathcal{P}_{kP}(\C^{*m})$ and the convex body $P$ is a
rational polytope the number $c_{\psi}$ coincides with $k^{-1}$
times the log canonical threshold $\text{lct \ensuremath{(\mathcal{D}_{k})}}$
of the divisor $\mathcal{D}_{k}$ in the toric variety $X_{P},$ cut-out
by the zero-locus of $p_{k}.$ Introducing the ``tropical log canonical
threshold'' 
\[
\text{lct}(\mathcal{D}_{k})_{trop}:=c_{\text{Log}^{*}\psi_{trop}},
\]
 where $\psi_{trop}$ is the tropicalization of $p_{k}$ and observing
that 
\begin{equation}
\text{lct\ensuremath{(\mathcal{D}_{k})}\ensuremath{\leq}lct}(\mathcal{D}_{k})_{trop}\label{eq:ineq lct and trop}
\end{equation}
 it is natural to ask for conditions ensuring that equality holds
in \ref{eq:ineq lct and trop}? Indeed, by the first point in Theorem
\ref{thm:tropical}, the first point in Conjecture \ref{conj:neg}
is equivalent to proving that equality holds in \ref{eq:ineq lct and trop},
for $k$ sufficiently large, when $p_{k}=D^{(N_{k})}$ 

More generally, for a given psh function $\psi\in\mathcal{L}_{P}(\C^{*n})$
one can ask for conditions ensuring that $c_{\psi}=c_{\text{Log}^{*}\psi_{trop}},$
where $\psi_{trop}$ is the convex function defined by the limit in
formula \ref{eq:trop as limit of scaling of psi}. The existence of
the limit, follows from \cite[Theorem 2]{ra} (in the terminology
of \cite{ra}, the psh function $\text{Log}^{*}\psi_{trop}=\Psi_{0}$\emph{
}is called the \emph{indicator }of $\psi$ at $0).$ 

\section{Appendix: Comparison with the Curie-Weiss model for magnetization }

It may be illuminating to compare the present statistical mechanical
setting with the classical example of the Curi-Weiss model, which
provided the first microscopic explanation of the macroscopic phenomenon
of spontaneous magnetization of certain materials (see \cite[Section IV.4]{el}
and references therein for more details on the mathematical setup).
It corresponds to the simplest possible statistical mechanical setup
where the space $X$ consists of merely two points.

\subsection{The Curie-Weiss model }

Let $X=\{1,-1\}$ (representing spin up/down) and consider the Hamiltonian

\[
H_{f}^{(N)}(x_{1},...,x_{N}):=-N^{-1}\frac{1}{2}\sum_{1\leq i,j\leq N}x_{i}\cdot x_{j},
\]
 representing the interaction between $N$ microscopic \emph{spins
}$x_{1},...x_{N}.$ The subscript $f$ indicates that the system is
\emph{ferromagnetic,} i.e. the corresponding Gibbs measure favors
ordered configuration of spins which are aligned. In particular, $H^{(N)}(x_{1},...,x_{N})$
has precisely two minimizers $(1,...,1)$ and $(-1,...-1)$ of completely
aligned spins, which are exchanged under the natural $\Z_{2}-$action
on $X^{N}$ induced by $x\mapsto-x.$ This is to be contrasted with
the anti-ferromagnetic setting corresponding to the Hamiltonian 
\[
H_{af}^{(N)}=-H_{f}^{(N)}
\]
 which favors ``disordered'' configurations of spins. 

In the setting of the Curie-Weiss model the ``Mean Field Approximation''
\ref{eq:mean field apprx heurstic form} is exact\footnote{However, in the physics literature the Hamiltonian $H^{(N)}$ is itself
viewed as a mean field approximation of more complicated spin models,
such as the Ising model, on a lattice $\Z^{D}$ (which improves as
the dimension $D$ is taken large \cite{b-c}). }:
\[
E_{f}^{(N)}(x_{1},...x_{N})=E_{f}(\delta_{N}),\,\,\,\,E_{f}(\mu)=-\frac{1}{2}\int x\cdot y\mu(x)\otimes\mu(y)
\]
Since $X$ has only two elements a probability measure $\mu\in\mathcal{M}_{1}(X)$
is uniquely determined by its first moment 
\[
m_{\mu}:=\int_{X}x\mu(x)\in[-1,1]
\]
(representing the \emph{mean magnetic moment} of the macroscopic spin
state $\mu).$ Accordingly, the empirical measure $\delta_{N}$ is
uniquely determined by the random variable 
\[
m_{N}:=m_{\delta_{N}}:=N^{-1}\sum_{i=1}^{N}x_{i}
\]
(representing the \emph{mean spin} of the microscopic state), taking
values in $[-1,1].$ Expressing $E$ in terms of $m$ gives, by separation
of variables,
\[
E_{f}(\mu)=-\frac{1}{2}m^{2}
\]

Hence, writing $I_{\beta}(m):=F_{\beta}(\mu)$ the corresponding rate
functional is given by 
\[
F_{\beta}(m):=-\frac{1}{2}m^{2}+\beta^{-1}\left(\frac{m+1}{2}\log\left(\frac{m+1}{2}\right)+\frac{1-m}{2}\log\left(\frac{1-m}{2}\right)\right)
\]
for $m\in[-1,1].$ The corresponding critical point equation yields
the following classical mean field equation for the magnetization
$m:$
\[
m=\tanh(\beta m)
\]
(which is consistent, as it much with equation \ref{eq:mean field eq in heur},
using that $\psi_{\mu}(x)=-x\cdot m_{\mu}).$ Inspection of the graph
of $\tanh s$ reveals that the $m=0$ is the unique solution when
$\beta\in[0,1]$ (and thus the unique minimizer of $F_{\beta}).$
However, when 
\[
\beta>\beta_{c}:=1
\]
(the ``critical inverse temperature'') there are two additional
solutions $m_{\pm}(\beta)$ which are exchanged by the $\Z_{2}-$symmetry
(i.e. $m_{-}(\beta)=-m_{+}(\beta)).$ The solutions $m_{\pm}(\beta)$
minimize of $F_{\beta},$ while $m=0$ is now a local maximize. Moreover,
$m_{\pm}(\beta)\rightarrow\pm1$ when $\beta\rightarrow\infty,$ coinciding
with the two minimizers of $E_{f}(m)$ (representing a macroscopic
state of all spins up, respectively down). 

The LDP for the laws of $m_{N}$ thus implies that
\[
\beta<\beta_{c}\implies m_{N}\rightarrow0,\,\,\,N\rightarrow\infty
\]
in law (which means that the system is not magnetized for sufficiently
large temperatures). However, when $\beta>\beta_{c}$ it follows,
from the LDP combined with the $\Z_{2}-$symmetry of $H_{f}^{(N)}$
that the laws of $m_{N}$ converge to $\delta_{\mu_{+}(\beta)}/2+\delta_{-\mu_{+}(\beta)}/2.$
Hence, $m_{N}$ does \emph{not }admit a deterministic limit, even
though the expectations of $m_{N}$ vanish identically for any $N$
(i.e. the law of large numbers fails in this case). This is the standard
example of a (first order) phase transition, appearing at the critical
inverse temperature $\beta=\beta_{c}.$ 

\subsubsection{Spontaneous symmetry breaking}

Determinism is restored in the presence of an exterior magnetic field,
which leads to the phenomenon of\emph{ spontaneous symmetry breaking.}
Indeed, turning on an exterior magnetic field $h\in\R$ amounts to
adding a linear exterior potential $\phi_{h}(x):=-x\cdot h$ to the
Hamiltonian $H_{f}^{(N)}.$ Then the corresponding rate functional
$I_{\beta,h}(m)$ has a unique minimizer $m(\beta,h)$ for any $h\neq0.$
Moreover, 
\[
\lim_{h\rightarrow0^{\pm}}m(\beta,h)=m_{\pm}(\beta)
\]
This means that the random variables $m_{N}$ on the corresponding
probability spaces $(X^{N},\mu_{\beta,h}^{(N)})$ do have a deterministic
limit:
\[
\beta>\beta_{c}\implies m_{N}\rightarrow m_{\pm}(\beta,h),\,\,\,N\rightarrow\infty
\]
in law, which is viewed as a microscopic explanation of the phenomenon
of spontaneous magnetization, observed in certain materials (such
as iron).

\subsection{Comparison with the present complex setting }

To make a comparison with the present setting of interpolation nodes
in $\C^{n}$ it is more convenient to start with the anti-ferromagnetic
Hamiltonian setting described by the Hamiltonian
\begin{equation}
H^{(N)}:=N^{-1}\frac{1}{2}\sum_{1\leq i,j\leq N}x_{i}\cdot x_{j}.\label{eq:cw af ham}
\end{equation}
Loosely speaking, the Hamiltonian $H^{(N)}$ is ``repulsive'', since
the ferromagnetic Hamiltonian $-H^{(N)}$ is ``attractive'' in the
sense that $H^{(N)}(x_{1},...,x_{N})$ is minimal precisely when all
$x_{i}$ coincide. In this setting the critical inverse temperature
$\beta_{c},$ where magnetization appears, is\emph{ negative,} $\beta_{c}=-1.$
The critical $\beta_{c}$ may abstractly be defined as 
\[
\beta_{c}:=\inf\left\{ \beta:\,F_{\beta}\,\text{has a unique minimizer}\right\} 
\]
Thus $\beta_{c}$ is analogous to the critical negative inverse temperature
$\beta_{c}$ for the repulsive determinantal Hamiltonian $H^{(N)}$
(formula \ref{eq:def of H N weighted intro}) which appears in the
present setting of $\C^{n},$ discussed in Section \ref{subsec:Tropicalization-at-negative}.
In the latter setting the role of symmetry group $\Z^{2}$ is played
by the group $\C^{*n}.$ This analogy becomes even more striking in
the tropical setting in $\R^{n},$ where the symmetry group is te
additive group $\R^{n}$ and where it can be shown that the corresponding
attractive Hamiltonian $H_{trop}^{(N)}(x_{1},...,x_{N})$ is regular
and minimal precisely when all points coincide, $x_{1}=....=x_{N}.$ 
\begin{rem}
Negative temperature states were first observed in nuclear spins in
the famous experiment \cite{h-n-v} (and then in \cite{h-n-v}, for
an antiferromagnetic model for silver, whose mean field Hamiltonian
coincides with \ref{eq:cw af ham}). However, the existence of negative
temperature states was originally proposed by Onsager for motional
degrees of freedoms in the vortex model (see the historical account
in \cite{e-s}). Incidentally, Onsager's Hamiltonian coincides with
the present one (formula \ref{eq:def of H N weighted intro}) in the
non-weighted setting of $\C$ (i.e. when $n=1).$ Note that the term
``negative temperature'' is somewhat misleading, though, as such
states are hotter then infinite temperature states. Indeed, the inverse
temperature $\beta$ interpolates continuous between positive and
negative temperature states, passing through the infinite temperature
state $\beta=0$ (see \cite{a-p} for a recent theoretical discussion
of negative temperature states in physics).
\end{rem}

\end{document}